\documentclass[10pt]{article}
\usepackage[a4paper,margin=1in]{geometry}
\usepackage{amsmath,amssymb,amsthm,mathtools}
\usepackage[T1]{fontenc}
\usepackage{lmodern}
\usepackage{microtype}
\usepackage{booktabs}
\usepackage{enumitem}
\usepackage{hyperref}
\hypersetup{colorlinks=true,linkcolor=blue,citecolor=blue,urlcolor=blue}
\setlist{nosep}
\allowdisplaybreaks
\newtheorem{theorem}{Theorem}[section]
\newtheorem{lemma}[theorem]{Lemma}
\newtheorem{proposition}[theorem]{Proposition}
\newtheorem{corollary}[theorem]{Corollary}
\newtheorem{assumption}[theorem]{Assumption}
\theoremstyle{remark}

\numberwithin{equation}{section}
\title{Sharp Critical Minimax Laws and No-Learning Thresholds\\
in Continuous-Time Adaptive Control}
\author{Chen Jia\\
\texttt{jiachenwestlake@gmail.com}}
\date{September 25, 2026}
\begin{document}
\maketitle
\begin{abstract}
We study episodic continuous-time control with an unknown vector control gain,
a scalar state, quadratic action cost, and a smooth convex terminal cost.
In the scalar Gaussian experiment, let $\Delta(H)$ be the minimax
improvement over zero control on $[-H,H]$ and set $\delta=\sqrt2H^2-1$.
We prove the matching critical law
$\Delta(H)\asymp\delta^4\sqrt{\log(1/\delta)}$ as $\delta\downarrow0$.
The upper bound follows from the uniform deficit--energy inequality
$\mathcal S_H(a)\ge cQ^{4/3}/[\log(e/Q)]^{1/6}$ for every admissible
algorithm with $0<Q\le1$, uniformly over $H\in[2^{-1/4},1]$,
where $Q$ is expected action energy.
A time-dependent fourth-moment test function proves this inequality
without restricting actions to Markov feedback or imposing an amplitude cap.
A moving soft threshold attains the critical improvement over the entire
parameter interval and shows that the logarithmic correction is necessary.
For parameters of the form $\theta=N^{-1/4}h$, $\|h\|\le H$, we compare the
local minimax expected regret over $N$ episodes with a fixed-horizon Gaussian
sequential control problem. For every finite parameter dimension, the
normalized values differ by at most $K_{H,g}/\sqrt N$, with no additional
dimension factor. The terminal task enters the limiting value only through
$c_g=\operatorname{Var}(g(Z))$, where $Z$ is standard normal. The Gaussian
problem has one shared scalar observation channel, and its zero-action
policy is minimax if and only if $H^4\le d^2/2$. This gives the exact
asymptotic no-learning boundary $c_gH^4=d^2/2$ for the control tasks.
The comparison allows full within-episode adaptation and unbounded action
amplitudes. Direct finite-sample bounds locate the undecided region
between exact zero-control optimality and strict improvement within
$O(N^{-1/2})$ of this boundary. At every fixed strictly subcritical
radius, zero control is exactly minimax for all sufficiently large $N$,
so the finite-sample value and its cumulant expansion are identified.
For one episode the exact boundary instead is
$H^4\|g'\|_\infty^2=d^2/2$; rare observable states make the extremal
task slope relevant.
\end{abstract}

\section{Introduction}
An adaptive controller pays for information through the same actions that
determine its performance. Near a parameter at which the optimal action
generates no information, the benefit of learning depends jointly on the
parameter uncertainty, the observation channel, and the task. Regret rates
alone do not determine whether a finite local experiment should acquire
information at all, nor do they identify the task-dependent constant at
which experimentation first improves the worst-case risk.

This paper gives an exact local decision problem for a class of nonquadratic
control tasks. We consider a scalar diffusion whose drift is the inner
product of an unknown vector and a chosen vector action. The diffusion is
observed continuously and restarts at zero between episodes. The terminal
cost is convex, smooth, and Lipschitz, while actions incur a quadratic cost.
Although the state is scalar, the unknown gain and the control are
$d$-dimensional, and all directions share the same observation channel.

Our main results determine both the local decision problem and the
order at which learning becomes beneficial above its scalar threshold.
The connection to terminal control tasks is a quantitative comparison
of minimax \emph{values}.
After normalization by $\sqrt N$, the error is at most
$K_{H,g}/\sqrt N$ on the entire local parameter ball
$\theta=N^{-1/4}h$, $\|h\|\le H$. Thus the error in cumulative minimax
regret is bounded independently of $N$ for a fixed local radius.
The comparison is uniform over the original history-dependent control
class, including controls that concentrate their energy in short intervals.
It does not require a process limit for every nearly optimal policy.

The limiting experiment is
\[
 dY_t=\langle h,a_t\rangle\,dt+dB_t,\qquad
 \frac12\mathbb E_h\int_0^1\|a_t+h\|^2\,dt.
\]
We prove that its zero-action policy is minimax exactly when
$H^4\le d^2/2$. A product prior on sphere-supported vertices gives the
lower bound. Revealing the other coordinates only within the proof reduces
each posterior estimate to a two-point calculation on the original
observation record. A random initial direction and a bounded scalar
feedback give strict improvement above the boundary over the whole
parameter ball.

We also determine finite-sample decisions more sharply than an additive
value comparison permits. After subtracting the exact zero-control
baseline, we freeze a coordinate posterior at each episode boundary.
The resulting approximation error is proportional to the actual
control energy. This gives exact zero-control optimality below a
finite-sample sufficient boundary. A bounded implementation of an explicit
second-variation direction gives strict improvement on the other side.
The two conditions leave an undecided region of width $O(N^{-1/2})$
around the asymptotic boundary, without specifying a sharp transition
width or the size of an improvement inside that region.
At every fixed strictly subcritical radius, the minimax value eventually
equals a log-Laplace expression exactly. Its next terms depend on
task cumulants beyond $c_g$.

For a single episode we obtain an exact, different answer:
zero control is minimax if and only if $H^4\|g'\|_\infty^2\le d^2/2$.
A controller may first wait without acquiring parameter information and
then start a small experiment on a rare observed state where the task
gradient is close to an extremal slope. A second-variation argument
turns this mechanism into strict improvement over the whole parameter
ball. For every nonaffine task in our class, this boundary is
strictly below the asymptotic boundary governed by $c_g$.

We determine the precise order of the improvement near the scalar boundary:
for $\delta=\sqrt2H^2-1\downarrow0$,
\[
 \Delta(H)\asymp\delta^4\sqrt{\log(1/\delta)}.
\]
The upper bound rests on a uniform deficit--energy law with a necessary
logarithmic correction. The deficit separates into feedback misalignment
and nonlinear posterior saturation. We retain when these losses occur
and construct a truncated fourth-moment test function whose profile
depends on the remaining logarithmic time. Its spatial equation compensates
the drift produced by a slowly moving threshold. Completing the square
for every real action then bounds all admissible rules, including
randomized and non-Markov rules with no common amplitude bound.
For the lower bound, a soft posterior threshold grows as $t^{7/4}$.
The same time exponent makes a fourth-moment decay estimate pay for
feedback misalignment and converts a tilted diffusion into a controlled
occupation-time estimate. A separate likelihood argument places the
worst-case risk at the endpoints of the full parameter interval.

The construction improves by an unbounded logarithmic factor over the
fourth-order guarantee obtained by a specified untruncated
probe-and-hyperbolic-tangent family. We retain a sharpness theorem for that
family and a quantified failure result for permanent stopping.
Together with the all-algorithm upper bound, the moving feedback attains
the optimal critical order. Its actual expected energy also attains the
necessary order $\delta^3\sqrt{\log(1/\delta)}$ for positively improving
rules. The same construction shows that the factor
$[\log(e/Q)]^{-1/6}$ in the uniform deficit law cannot be removed
or replaced by a strictly larger asymptotic factor.

\subsection{Related work and scope of the contribution}
Continuous-time episodic learning in linear-convex systems provides the
control-theoretic setting for this work. Szpruch, Treetanthiploet, and Zhang
\cite{STZ} study the exploration--exploitation tradeoff in a substantially
broader linear-convex model. Our model is more specialized: the state and
noise are scalar, the unknown quantity is a control gain, and there is no
state-dependent running cost. Within this model, our focus is the complete
local minimax value, an exact finite-horizon experimentation boundary,
and a matching critical improvement law,
rather than a general regret-rate guarantee.

Controlling the intensity of experimentation and choosing not to learn
are classical ideas. Moscarini and Smith \cite{MS} study optimally chosen
observation intensity with endogenous stopping and convex information
costs. Keller and Rady \cite{KR,KRworking} establish regimes of experimentation,
including a fixed-state case in which a confounding action can sustain
nonlearning. Their regime boundary need not depend on the particular
action interval once it contains the myopic optima. Accordingly, neither
the existence of a no-learning regime nor the absence of an amplitude cap
is, by itself, our contribution.
For an unknown-gain integrator, {\AA}str\"om and Helmersson \cite{AH}
already derive a rationally attenuated one-stage cautious control and
study the role of active probing. Their normalized rule has the form
$-\eta\beta/(1+\beta^2)$. Our feedback has the same algebraic shape under
a different normalization; the rational denominator itself is not new.

The classical Bene\v{s}--Rishel problems share our unknown-gain
observation dynamics. Bene\v{s}, Karatzas, and Rishel \cite{BKR}
study infinite-horizon discounted Bayesian regulation with quadratic
state cost and controls in $[-1,1]$, and obtain a bang--bang optimal
law in their wide-sense class. Karatzas and Ocone \cite{KO92}
develop the associated resolvent analysis and extend the
state-cost treatment beyond the quadratic case.
Their finite-horizon treatment \cite{KO93} solves a degenerate
parabolic HJB equation, as described in the authors' abstract and
in \cite[Section~1.2 and Remark~3.7]{CKM}.
In a related problem with discretionary stopping, Karatzas and
Ocone \cite{KO02} study running and termination costs depending on
distance from the origin, with both known and unknown control gain.
They obtain explicit stopping solutions for power-law costs in
the fully observed case and lower bounds on the stopping region in
the partially observed case. Our description of \cite{KO02} is based
on the authors' and institutional abstracts.

These state-regulation formulations differ from our
interval-minimax action-loss problem \eqref{eq:canonical}.
In the material examined, we have not identified a theorem
that directly yields our no-learning threshold or critical bounds.
The full text of \cite{KO93} was unavailable to us; our comparison
for that paper relies on its authors' abstract as reproduced in
bibliographic records, the primary predecessor \cite{BKR}, and
\cite[Section~1.2 and Remark~3.7]{CKM}.
This comparison does not exclude unexamined results or possible
reductions to our problem.

The distinction is in the quantitative decision problem. In our two-point
Bayesian comparison, the action-dependent part of instantaneous loss is
$a^2/2+ma$, where $m$ is the posterior mean. It combines sensing and
exploitation over a deterministic physical horizon. It is not a
parameter-independent purchase price for a terminal experiment.
Ekstr\"om and Karatzas \cite{EK} allow quadratic observation costs and obtain
optimal constant-intensity observation with discretionary stopping for a
terminal estimation objective. Their Bernoulli example also has an exact
no-observation threshold. That objective does not identify the sensing
action with the continuously evaluated estimation action used here.
The confounding-demand model of Keller and Rady additionally has a
state-dependent quadratic payoff coefficient and an infinite discounted
horizon. These differences are retained in our formulation.

General finite-horizon Bayesian control with separable drift uncertainty
is treated by Cohen, Knochenhauer, and Merkel \cite{CKM}.
Their Theorem~5.8 gives a viscosity HJB characterization, equality of
strong and weak values, and piecewise-constant nearly optimal controls.
It includes bounded-action Bayesian versions of our running-cost problem:
for any fixed $M<\infty$, take $U=[-M,M]$, hidden parameter $\lambda=h$,
$b=u$, $\sigma=1$, zero terminal cost, and
$k(t,y,u,\ell)=(u+\chi_H(\ell))^2/2$, where
$\chi_H(\ell)=(-H)\vee(\ell\wedge H)$.
This agrees with our loss on the prior support and meets their global
growth condition. Thus neither this bounded-action Bayesian model nor
its finite-dimensional posterior representation is new.
The general HJB result does not supply the explicit critical improvement
scale, the uniform deficit--energy remainder, or the full-interval risk
guarantee proved here, and it does not
by itself pass to the unrestricted common-well-posedness class uniformly
as $M$ increases.

Adusumilli \cite{Wald} derives minimax sequential sampling and stopping
policies, together with local asymptotic results for general outcome
distributions. The risk there is terminal treatment choice plus sampling
cost. Adusumilli \cite{Representation} develops continuous-time
representations of adaptive finite-arm experiments, including equivalence
of in-sample regret for arbitrary policy sequences under the stated
assumptions. The allocation clocks in that formulation are 1-Lipschitz in
physical time. An energy bound on our
unbounded actions does not imply such a restriction on the information
clock. We therefore compare values directly and retain the complete
likelihood after time change. The auxiliary approximation arguments use
classical stochastic-control tools; we do not claim a general new
representation theorem.

The coordinate-conditioning argument used for our vector lower bound is
related to product-prior lower bounds for adaptive sensing
\cite{ACD}. Those results concern terminal estimation or support recovery
under measurement budgets. Here the posterior estimate must also control
the continuous action--parameter cross term sharply enough to match a
whole-ball upper bound. The task-value comparison and exact threshold
identify the canonical experiment; the critical results give a matching
bound for all adaptive rules and a uniform logarithmic deficit remainder.
The proof uses classical It\^o verification, completion of squares,
convexity, stochastic exponentials, and time change.
The model-specific step in the upper bound is a family of fourth-moment
profiles that simultaneously compensates threshold drift and bounds
curvature against both actual deficit measures, uniformly over all
states and remaining time scales. This yields
$q\lesssim s^3\sqrt{\log(e/s)}$ for normalized energy $q$ and relative
deficit $s$. We do not claim the general verification method or
the use of time-dependent test functions as new.
Inverse-square-root decay under Brownian half-line occupation
penalisation is classical; see Roynette, Vallois, and Yor \cite{RVY},
Theorem~6.1 and Example~6.4. The occupation distribution used below
is L\'evy's arcsine law; see, for example, Salminen and Stenlund
\cite{SS}, Corollary~6.9. Our adapted-clock estimate is proved by
reducing to a deterministic-time Brownian occupation event. This
reduction is included because the clock depends on the control and may
have flat intervals. We make no separate novelty claim for the
inverse-square-root exponent or for the Brownian identities.
State-dependent attenuation of fluctuations also has direct control
precedents. Ankirchner, Blanchet-Scalliet, and Jeanblanc \cite{ABJ}
maximize the expected occupation time of an exponential martingale
above a fixed threshold. Their optimal feedback uses the smaller
volatility above the threshold and the larger volatility below it.
That result optimizes a linear occupation reward over volatilities
bounded between two positive constants. Our feedback instead pays
an action-alignment cost, and the exponential loss-clock estimate is
used to bound posterior moments rather than a linear occupation reward.

\section{Model and decision rules}\label{sec:model}
Let $d,N\ge1$. In episode $k$, the state starts at $X_{k,0}=0$ and obeys
\begin{equation}\label{eq:task}
 dX_{k,t}=\langle\theta,U_{k,t}\rangle\,dt+dW_{k,t},
 \qquad 0\le t\le1,\quad \theta,U_{k,t}\in\mathbb R^d.
\end{equation}
The scalar Brownian motions $W_k$ are independent across episodes.
The controller observes the state path, its past actions, and its own
randomization. A decision rule is a parameter-independent nonanticipative
rule on this record; actions have a predictable version. All private
randomization can be included in an initial seed independent of the
parameter and of the Brownian drivers. The controller does not observe the
drivers separately.

\begin{assumption}[Admissible experiments]\label{ass:control}
For every parameter in the comparison set, the same decision rule defines
a nonexplosive experiment with a unique law. This also holds for the rule
that permanently sets the action to zero when its cumulative squared
Euclidean norm first reaches any fixed finite cap. The laws are defined
on a common parameter-independent record space, with the usual filtration
conditions. No parameter-dependent completion may add observations.
There is no deterministic amplitude bound on admissible rules.
\end{assumption}
We evaluate controls with finite worst-case risk; controls with infinite
worst-case risk cannot improve the minimax value. Coercivity proved below
implies finite expected total energy for every finite-risk control.
Assumption~\ref{ass:control} supplies the identification of stopped
likelihoods without imposing a global Novikov condition. For Bayesian
arguments, the record laws and the rule are measurable in the parameter,
as in the usual canonical formulation of controlled diffusions.

\begin{assumption}[Terminal task]\label{ass:g}
The function $g:\mathbb R\to\mathbb R$ is nonconstant, convex, and $C^4$.
For finite constants $L>0$ and $M\ge0$,
\[
 \|g'\|_\infty\le L,\qquad 0\le g''\le M,
\]
and $g^{(3)},g^{(4)}$ are bounded.
\end{assumption}
The assumptions allow affine functions with nonzero slope, asymmetric
convex functions, and functions unbounded below. They exclude quadratic
terminal costs. A constant terminal cost has minimax regret zero and will
be treated separately.

For known $\theta$, let $V_\theta(t,x)$ denote the optimal remaining cost
with running cost $\|U\|^2/2$ and terminal cost $g$. The expected cumulative
pseudo-regret of a rule $\mathcal A$ is
\begin{equation}\label{eq:regret}
 R_N^g(\mathcal A,\theta)=
 \mathbb E_\theta\sum_{k=1}^N
 \left[\frac12\int_0^1\|U_{k,t}\|^2\,dt+g(X_{k,1})\right]
 -NV_\theta(0,0).
\end{equation}
This is an expectation, not a realized-regret statement. For $H<\infty$,
define the local normalized value
\begin{equation}\label{eq:CN}
 C_{N,d}^g(H)=\inf_{\mathcal A}\sup_{\|h\|\le H}
 \frac{R_N^g(\mathcal A,N^{-1/4}h)}{\sqrt N}.
\end{equation}
If the original parameter domain is restricted to
$\|\theta\|\le\bar\theta$, statements involving $N,H$ require
$N^{-1/4}H\le\bar\theta$. On the whole space there is no additional
restriction on $N$.

For $T>0$, the canonical value is
\begin{equation}\label{eq:canonical}
\begin{split}
 dY_t&=\langle h,a_t\rangle\,dt+dB_t,\quad 0\le t\le T,\\
 \mathcal J_{T,h}(a)&=\frac12\mathbb E_h\int_0^T\|a_t+h\|^2\,dt,\qquad
 C_d^T(H)=\inf_a\sup_{\|h\|\le H}\mathcal J_{T,h}(a).
\end{split}
\end{equation}
The standard decision rules in \eqref{eq:canonical} satisfy the same
record, randomization, and stopped-well-posedness conventions.
There is one scalar $Y$, irrespective of $d$. In the scalar, unit-horizon
case write $\mathcal L_h=\mathcal J_{1,h}$ and
$\Delta(H)=H^2/2-C_1^1(H)$.
The notation $\mathbb E_\pi$ includes the draw of a parameter from $\pi$.

\section{Main results}\label{sec:results}
\begin{theorem}[Quantitative task-value comparison]\label{thm:value}
Under Assumptions~\ref{ass:control}--\ref{ass:g}, put
$c_g=\operatorname{Var}(g(Z))$, $Z\sim N(0,1)$. Then $0<c_g\le L^2$.
For every finite $d\ge1$ and every admissible choice of $N,H$,
\begin{equation}\label{eq:value}
 \left|C_{N,d}^g(H)-\sqrt{c_g}\,C_d^1(c_g^{1/4}H)\right|
 \le \frac{K_{H,L,M}}{\sqrt N},
\end{equation}
where
\begin{equation}\label{eq:K}
 K_{H,L,M}=
 2H^3LM\sqrt{4H^2L^2+4}+H^4L^2M+\frac12H^2L^2.
\end{equation}
The displayed constant has no additional dependence on $d$ or on the
amplitude, grid, or energy cap used to implement an approximating policy.
\end{theorem}

The theorem characterizes the entire limiting value for every fixed
radius, including the region where experimentation is beneficial.
It does not supply a closed form for $C_d^1$ above the threshold.
Dimension-free comparison error does not assert dimension-free
computational complexity. If $H$ grows with dimension, the displayed
constant can of course grow through $H$.

\begin{theorem}[Exact Gaussian no-learning boundary]\label{thm:threshold}
For every finite $d\ge1$, $H\ge0$, and $T>0$,
\begin{equation}\label{eq:threshold}
 C_d^T(H)=\frac{TH^2}{2}
 \quad\Longleftrightarrow\quad H^4T\le\frac{d^2}{2}.
\end{equation}
The equality case belongs to the zero-action region. If the inequality
fails, a bounded, randomized, observation-adapted policy has worst-case
risk strictly below $TH^2/2$ over the whole Euclidean ball.
\end{theorem}

\begin{corollary}[Task-dependent boundary]\label{cor:taskboundary}
For fixed $d,H,g$ as above,
\[
 \lim_{N\to\infty}C_{N,d}^g(H)=\frac{c_gH^2}{2}
 \quad\Longleftrightarrow\quad c_gH^4\le\frac{d^2}{2}.
\]
In this region zero control attains the limiting minimax value.
Above the boundary the limiting minimax value is strictly smaller.
\end{corollary}

\subsection{Matching scalar critical and deficit--energy laws}
Put $H_c=2^{-1/4}$ and $\delta=\sqrt2H^2-1$. In the scalar experiment,
let $\pi_H=(\delta_{-H}+\delta_H)/2$,
$m_t=\mathbb E_{\pi_H}[h\mid\mathcal F_t]
=H\tanh(H\int_0^ta_s\,dY_s)$, and
\[
 Q=\mathbb E_{\pi_H}\int_0^1a_t^2\,dt,\qquad
 \mathcal S_H(a)=\frac{H^2Q}{\sqrt2}
           +\mathbb E_{\pi_H}\int_0^1m_ta_t\,dt .
\]
Proposition~\ref{crit:prop:identity} expresses $\mathcal S_H$ as the sum
of two nonnegative losses. The endpoint Bayes improvement is
$b_H(a)=\delta Q/2-\mathcal S_H(a)$, and only the direction
$\Delta(H)\le\sup_a b_H(a)$ is used for the all-algorithm bound.

\begin{theorem}[Matching unrestricted critical law]\label{thm:critical-new}
There are absolute constants $c,C,\delta_0>0$ such that
\begin{equation}\label{eq:critical-new}
 c\delta^4\sqrt{\log(1/\delta)}
 \le\Delta(H)\le C\delta^4\sqrt{\log(1/\delta)},
 \qquad 0<\delta\le\delta_0.
\end{equation}
The upper bound covers the complete class in Assumption~\ref{ass:control},
including randomized rules, adaptive stopping and restarting, and
unbounded action amplitudes. The moving feedback below attains the lower
bound over $[-H,H]$. The statement determines the order, not a leading
asymptotic constant.
\end{theorem}

\begin{theorem}[Uniform logarithmic deficit--energy law]\label{thm:energy-new}
There is an absolute $c_0>0$ such that, uniformly over $H\in[H_c,1]$
and all admissible algorithms with $0<Q\le1$,
\begin{equation}\label{eq:energy-new}
 \mathcal S_H(a)\ge
 c_0\frac{Q^{4/3}}{[\log(e/Q)]^{1/6}}.
\end{equation}
There are also absolute $C_0,\delta_1>0$ such that every algorithm
with $b_H(a)>0$ and $0<\delta\le\delta_1$ satisfies
\begin{equation}\label{eq:positive-energy-sharp}
 Q\le C_0\delta^3\sqrt{\log(e/\delta)}.
\end{equation}
The first assertion does not require positive improvement.
The second concerns mixture expected energy, not a pathwise cap.
\end{theorem}

For the constructive result, define
\begin{equation}\label{eq:moving-design}
 \begin{gathered}
 \tau=\delta/8,\quad L_\tau=\log(1/\tau),\quad
 \kappa=\frac1{128(1+\log24)},\\
 f=\kappa\sqrt{1+L_\tau},\qquad e=f\tau^{7/2},\qquad k=24f .
 \end{gathered}
\end{equation}
With an independent uniform sign $\rho$, set
\begin{equation}\label{eq:moving-policy}
 a_t^{\mathrm{mov}}=
 \begin{cases}
  \rho\sqrt{e/\tau},&0\le t\le\tau,\\[1mm]
  -\dfrac{m_t}{H^2\sqrt{2t}}\,
       \dfrac1{1+(m_t/H)^2/(kt^{7/2})},&\tau<t\le1 .
 \end{cases}
\end{equation}
The design depends on the known radius $H$ and observable score, not on
the unknown $h$. For interior parameters, $m_t$ in this formula means the
same score functional, not the posterior under the fixed-parameter law.
The soft threshold for $|m_t|/H$ is $\sqrt{k}\,t^{7/4}$; the action remains
nonzero whenever $m_t\ne0$, including beyond that threshold.

\begin{theorem}[Full-interval moving-feedback guarantee]\label{thm:moving}
For every $0<\delta\le1/16$, \eqref{eq:moving-policy} belongs to the
admissible class and
\[
 \sup_{|h|\le H}\mathcal L_h(a^{\mathrm{mov}})
 =\mathcal L_H(a^{\mathrm{mov}})
 =\mathcal L_{-H}(a^{\mathrm{mov}})
 \le\frac{H^2}{2}-c_1\delta^4\sqrt{\log(1/\delta)},
 \qquad c_1=\frac1{2^{20}(1+\log24)} .
\]
In particular, $\Delta(H)/\delta^4\to\infty$ as $\delta\downarrow0$.
\end{theorem}

\begin{corollary}[Sharpness of the logarithmic deficit correction]
\label{cor:no-four-thirds}
Along the policies \eqref{eq:moving-policy} as $\delta\downarrow0$,
\[
 Q\asymp\delta^3\sqrt{\log(1/\delta)},\qquad
 \mathcal S_H(a^{\mathrm{mov}})
 \asymp\frac{Q^{4/3}}{[\log(e/Q)]^{1/6}}.
\]
Thus the energy order in \eqref{eq:positive-energy-sharp} is attained,
and the right side of \eqref{eq:energy-new} cannot be multiplied by
any factor tending to infinity as $Q\downarrow0$, uniformly over a
critical neighborhood of $H_c$. In particular, no uniform positive
constant gives $\mathcal S_H(a)\ge cQ^{4/3}$ there.
\end{corollary}

\subsection{Exact finite-sample decisions and task corrections}

Write $\mu_g=\mathbb E g(Z)$, $c=c_g$, and let $\kappa_j$ be the
$j$th cumulant of $g(Z)$. In particular,
$\kappa_3=\mathbb E(g(Z)-\mu_g)^3$ and
$\kappa_4=\mathbb E(g(Z)-\mu_g)^4-3c^2$.
Set
\[
 \psi_g(u)=\log\mathbb E e^{-u(g(Z)-\mu_g)},\qquad
 Z_{N,d}^g(H)=\sup_{\|h\|\le H}
          \frac{R_N^g(0,N^{-1/4}h)}{\sqrt N}.
\]
The exact zero-control baseline is
\begin{equation}
 Z_{N,d}^g(H)=
 \sqrt N\,\frac{\psi_g(H^2/\sqrt N)}{H^2/\sqrt N},
 \label{eq:finite-baseline}
\end{equation}
with continuous value zero at $H=0$. It is independent of $d$
for fixed $H$ and strictly increasing in $H>0$.

\begin{theorem}[Finite-sample exact zero-control regions]
\label{thm:finite-zero}
Under Assumptions~\ref{ass:control}--\ref{ass:g}, for every finite
$d,N$ and admissible $H$, one has
$C_{N,d}^g(H)=Z_{N,d}^g(H)$ if either
\begin{equation}
 H^4L^2\le\frac{d^2}{2},
 \label{eq:finite-zero-global}
\end{equation}
or
\begin{equation}
 \frac{H^2}{d}\sqrt{\frac c2}
 +\frac{H^2}{\sqrt N}
      \left(\frac{L}{d\sqrt2}+\frac M2\right)
 \le\frac12 .
 \label{eq:finite-zero-local}
\end{equation}
Both equality cases are included. The lower bound covers the entire
admissible class, including unbounded action amplitudes.
\end{theorem}

For the opposite direction, define for $N\ge2$
\begin{equation}
 \beta_g=\frac{\kappa_3}{3c},\qquad
 a_N=\frac1N\sum_{k=2}^N
 \sqrt{\frac{2\sqrt{k-1}}{\sqrt k+\sqrt{k-1}}}.
 \label{eq:finite-design-numbers}
\end{equation}
The deterministic number $a_N$ is unrelated to the canonical
action process $a_t$.

\begin{theorem}[Finite-sample strict improvement]
\label{thm:finite-positive}
Under the same assumptions, if $N\ge2$ and
\begin{equation}
 \frac{H^2}{d}
 \left(\sqrt{\frac c2}\,a_N-\frac{\beta_g}{\sqrt N}\right)
 >\frac12 ,
 \label{eq:finite-positive}
\end{equation}
then a bounded, randomized rule, adapted to the original observations,
has worst-case normalized regret strictly smaller than $Z_{N,d}^g(H)$
over the entire parameter ball. In particular,
$C_{N,d}^g(H)<Z_{N,d}^g(H)$.
\end{theorem}

These are exact finite-$N$ decisions, with distinct sufficient
conditions. To state their localization, put
$\delta_d=\sqrt{2c}\,H^2/d-1$ and
$\mathsf H_n=\sum_{j=1}^n j^{-1}$. The explicit estimate
\[
 0<1-a_N\le\frac{1+\mathsf H_{N-1}/4}{N}
\]
shows that, for fixed $g,d$ and bounded radii, the region not classified
by \eqref{eq:finite-zero-local} and \eqref{eq:finite-positive}
lies within $O(N^{-1/2})$ of $\delta_d=0$.
This does not assert a sharp window, a nonzero threshold shift,
monotonicity of the whole finite-sample zero-control set, or a
uniform size of the strict improvement. The coefficient $\beta_g$
belongs to our explicit direction and need not be an optimal shift.

\begin{corollary}[Exact subcritical value and cumulant expansion]
\label{cor:finite-expansion}
Fix $d,g$ and $H$ with $cH^4<d^2/2$, and put
$\gamma(H)=1/2-(H^2/d)\sqrt{c/2}>0$.
Whenever the parameter domain permits and
\[
 N\ge
 \left[
 \frac{H^2\{L/(d\sqrt2)+M/2\}}{\gamma(H)}
 \right]^2,
\]
the minimax value is exactly \eqref{eq:finite-baseline}.
Consequently, uniformly on every fixed interval
$0\le H\le H_*<(d^2/(2c))^{1/4}$,
\begin{equation}
 C_{N,d}^g(H)=\frac{cH^2}{2}
  -\frac{\kappa_3H^4}{6\sqrt N}
  +\frac{\kappa_4H^6}{24N}
  +O_{g,H_*}(N^{-3/2}).
 \label{eq:finite-expansion}
\end{equation}
\end{corollary}

Thus $c_g$ determines the leading value but not its finite-sample
corrections. An explicit comparison makes this distinction concrete.

\begin{proposition}[Equal task variance, different minimax corrections]
\label{prop:finite-nonuniversal}
Let $\varphi(x)=\log\cosh x$,
$c_\varphi=\operatorname{Var}(\varphi(Z))$, $g_0(x)=x$, and
\[
 g_\eta(x)=
       \frac{x+\eta\varphi(x)}{\sqrt{1+\eta^2c_\varphi}}.
\]
For every sufficiently small fixed $\eta>0$, $g_\eta$ satisfies
Assumption~\ref{ass:g}, is nonaffine, and has $c_{g_\eta}=c_{g_0}=1$
and $\kappa_3(g_\eta(Z))>0$.
For every fixed $0<H^4<d^2/2$, both tasks have zero control exactly
minimax for all sufficiently large $N$, and
\[
 C_{N,d}^{g_\eta}(H)-C_{N,d}^{g_0}(H)
 =-\frac{\kappa_3(g_\eta(Z))H^4}{6\sqrt N}
       +O(N^{-1})<0.
\]
Each regret is measured relative to its own task oracle.
\end{proposition}

\begin{theorem}[Exact single-episode boundary]
\label{thm:one-episode}
For $N=1$ and every finite $d\ge1$, let $\ell=\|g'\|_\infty$ be the smallest
Lipschitz constant of $g$. Then
\begin{equation}
 C_{1,d}^g(H)=Z_{1,d}^g(H)
 \quad\Longleftrightarrow\quad H^4\ell^2\le\frac{d^2}{2}.
 \label{eq:one-episode}
\end{equation}
If the inequality fails, a bounded rule using a finite deterministic
grid, a random coordinate direction, and an independent uniform sign
strictly improves the worst-case risk over the whole parameter ball.
If $g$ is nonaffine, then $c_g<\ell^2$,
so this exact single-episode boundary differs strictly from
the asymptotic task boundary.
\end{theorem}

\subsection{Comparison with untruncated feedback}
For comparison and for the threshold proof, use the scalar policy
\begin{equation}\label{eq:tanh}
 a_t^{\tau,e}=
 \begin{cases}
 \sqrt{e/\tau}\,\rho,&0\le t\le\tau,\\[1mm]
 -\dfrac{\tanh(H\Xi_t)}{H\sqrt{2t}},&\tau<t\le1,
 \end{cases}
 \qquad \Xi_t=\int_0^t a_s\,dY_s,
\end{equation}
where $\rho$ is an independent uniform sign, $0<\tau<1$, and $e>0$.
The design uses $H$, not the unknown $h$. Its energy is at most
$e+\log(1/\tau)/(2H^2)$ on every path.

\begin{theorem}[Elementary critical bounds]\label{thm:deficit}
Let $\delta=\sqrt2H^2-1$. For $0<\delta\le1/16$,
\begin{equation}\label{eq:deficit}
 \frac{35\sqrt2}{1536(1+\delta)^4}\,\delta^4
 \le\Delta(H)\le
 \frac{9}{8H^2(1-2\delta)}\,\delta^2.
\end{equation}
The upper bound alone holds for $0<\delta<1/2$ and all admissible
policies. The lower bound is attained as a guarantee by
\eqref{eq:tanh} with
\[
 \tau=\frac{\delta}{2(1+\delta)},\quad
 e=\frac{\delta\tau^{5/2}}{8K_H},\quad
 K_H=\frac{12H^4}{35\sqrt2}.
\]
\end{theorem}

\begin{theorem}[Sharp order within the specified feedback family]\label{thm:restricted}
For $0<\delta=\sqrt2H^2-1<1$, define
\[
 C_{\mathrm{tanh}}(H)=
 \inf_{\substack{\delta/3\le\tau\le2\delta/3\\ e>0}}
       \sup_{|h|\le H}\mathcal L_h(a^{\tau,e}).
\]
There exist $c,C,\delta_0>0$ such that
\[
 c\delta^4\le H^2/2-C_{\mathrm{tanh}}(H)\le C\delta^4,
 \qquad 0<\delta\le\delta_0.
\]
\end{theorem}
Theorems~\ref{thm:critical-new} and \ref{thm:restricted} have different
quantifiers. The latter rules out an unbounded improvement factor obtained
solely by retuning the probe energy in its specified duration range.
The moving feedback changes the feedback law and lies outside that family.
The elementary bounds in Theorem~\ref{thm:deficit} remain useful as explicit
calibration and as ingredients of the stronger results.

\subsection{Proof structure}
The heat-semigroup gradient
$S(t,x)=\mathbb E[g'(x+\sqrt{1-t}Z)]$ defines a task clock. An exact
Bellman identity first converts regret into squared action error.
Replacing the oracle gradient by $S$ incurs a uniform controlled-energy
error. Time change then yields the canonical experiment up to a random
terminal time close to $c_g$. The complete likelihood, rather than only
the marginal Brownian property, survives the change of clock.

A one-sided energy truncation and a causal conditional-kernel construction
identify the weak and standard canonical values uniformly over the
parameter ball. Nonnegative losses allow comparison of the random terminal
time with $c_g$ without controlling the energy in a small final interval.
This proves Theorem~\ref{thm:value}. The threshold follows from two-point
posterior estimates and a random-direction policy.
For the finite-sample results we instead subtract the exact zero-control
cost before comparing tasks. Freezing each coordinate posterior at
episode starts makes the error proportional to the expected energy.
An explicit second-variation direction, truncated before choosing its
small amplitude, yields the opposite finite-sample condition.
The single-episode construction starts that variation on a rare state
event where the task gradient approaches an extremal slope.
Appendices~\ref{fin:sec:proof} and \ref{one:sec:proof} give these arguments.
For Theorem~\ref{thm:energy-new}, normalize the posterior to $X=m/H$
and the action to $u=Ha$, and put $q=\mathbb E\int u^2$ and
$s=\mathcal S_H/q$. Expected-energy shape and a convex truncated-moment
inequality supply the initial mass of a fourth-moment test function.
An explicit profile compensates a small threshold drift and gives an
inequality for every action, with the true posterior nonlinearity
paid by the saturation deficit. The resulting estimate is
\[
 \frac{q^2}{\tau^2}
 \le A_1\sqrt{1+\log(T/\tau)}\,\mathcal S_H
       +A_2\frac q{\tau^3}\mathcal S_H,\qquad T=1/2.
\]
Choosing the deterministic analysis time $\tau=Ds$, after fixing
all constants, absorbs the second term. This proves
$q\lesssim s^3\sqrt{\log(e/s)}$, the uniform deficit law, and the
critical upper bound. The analysis time is not a new control policy.
For Theorem~\ref{thm:moving}, the decay of a normalized fourth moment
pays the alignment loss, while a fourth-moment change of measure and
an occupation-time estimate reduce the nonlinear loss. The actual second
moment is controlled separately to retain the useful energy.
The likelihood then locates the maximum risk over the full interval.
Appendices~\ref{sharp:sec:proof} and \ref{new:sec:moving} give these proofs.

\section{Examples and statistical interpretation}\label{sec:examples}
For $g(x)=b+mx$, $m\ne0$, the task clock is deterministic and the gradient
replacement is exact. For every valid $N$,
\[
 C_{N,d}^g(H)=C_d^{m^2}(H)=|m|\,C_d^1(|m|^{1/2}H).
\]
Thus the $N^{-1/4}$ scaling is already an exact change of units in the
affine example; it is not itself an asymptotic representation result.
For $g(x)=\log\cosh x$, $L=M=1$ is valid and numerical Gaussian integration
gives $c_g\approx0.1897674492$. The boundary is
$H_{\mathrm{crit},d}=(2c_g)^{-1/4}\sqrt d\approx1.2740517570\sqrt d$.
The numerical value is illustrative; the theorems use the exact variance.
For this same task with $d=N=1$, Theorem~\ref{thm:one-episode}
gives the smaller exact critical radius $2^{-1/4}\approx0.8408964153$.
The variance comparison and the single-episode boundary answer
different finite-sample questions.

The value comparison also quantifies the size of a finite-sample
improvement away from the boundary. In the scalar case let
$Z_N^g(H)=\sup_{|h|\le H}R_N^g(0,N^{-1/4}h)/\sqrt N$ and
$G_N^g(H)=Z_N^g(H)-C_{N,1}^g(H)\ge0$. Then
\begin{equation}\label{eq:finite-transfer}
 \left|G_N^g(H)-\sqrt{c_g}\Delta(c_g^{1/4}H)\right|
 \le\frac{K_{H,L,M}+H^4L^2M}{\sqrt N}.
\end{equation}
Indeed the same gradient bound used for Theorem~\ref{thm:value} gives
$|Z_N^g(H)-c_gH^2/2|\le H^4L^2M/\sqrt N$, and subtraction proves
\eqref{eq:finite-transfer}. A canonical improvement gives a positive
finite-$N$ value gap whenever its scaled lower bound exceeds the displayed
comparison error. This statement does not identify a necessary width of a
finite-sample critical window or an optimal finite-$N$ implementation rate.
Theorems~\ref{thm:finite-zero} and \ref{thm:finite-positive} give the
stronger $O(N^{-1/2})$ localization of exact zero versus strict improvement
by a separate argument; their small-amplitude construction does not
supply the corresponding improvement magnitude.

\section{Limitations and open problems}\label{sec:discussion}
The results rely on a known scalar diffusion coefficient, independent
resets, a gain-only uncertainty, and quadratic action cost.
The terminal-cost class is nonquadratic but does not include unbounded
gradients. The comparison is local in the parameter and does not establish
constant global regret. It also does not compute the supercritical
canonical value or prove optimality of a random fixed direction.

The finite-sample localization conditions are sufficient and generally
do not coincide. We have not determined the exact boundary for arbitrary
$N>1$, its optimal shift, or a sharp critical window. The cumulant
expansion identifies the minimax corrections only at fixed strictly
subcritical radii. The exact one-episode theorem does not quantify
the potentially very small improvement obtained on a rare state event.

The scalar critical power and square-root logarithmic factor in
\eqref{eq:critical-new} are determined up to fixed multiplicative constants.
We have not identified a leading asymptotic constant, an exactly optimal
feedback, or the critical improvement law in higher dimension.
Other threshold curves or multiple probes can still affect constants,
but cannot exceed the proved scalar order within the stated algorithm class.
The endpoint maximum proved for the moving feedback does not make
the endpoint prior least favorable for the entire supercritical game.

\appendix

\section{Posterior estimates and the exact threshold}\label{sec:thresholdproof}
We use classical Girsanov, It\^o, and martingale time-change arguments
in their usual filtered-space forms \cite{KS}. Each likelihood below
is a likelihood of the entire observation and randomization record.

\begin{lemma}[Two-point posterior on the complete record]\label{lemma:posterior}
In the scalar canonical experiment, put
$\pi_H=(\delta_{-H}+\delta_H)/2$, $H>0$, and assume finite risk at both
endpoints. With $\Xi_t=\int_0^t a_s\,dY_s$ and
$Q_t=\int_0^t a_s^2\,ds$,
\begin{equation}\label{eq:posterior}
 m_t=\mathbb E_{\pi_H}[h\mid\mathcal F_t]=H\tanh(H\Xi_t),\qquad
 dm_t=(H^2-m_t^2)a_t\,dI_t,\quad m_0=0,
\end{equation}
where $I_t=Y_t-\int_0^t m_sa_s\,ds$ is Brownian in the complete
record filtration under the mixture law.
\end{lemma}
\begin{proof}
Finite risk implies
$\mathbb E_hQ_T\le4\mathcal J_{T,h}(a)+2H^2T<\infty$
at the two endpoints. Stop actions at $Q=K$. The stopped record laws
are equivalent, with likelihood ratio
$\exp(2H\Xi_{t\wedge\sigma_K})$ between $H$ and $-H$:
Girsanov is a true-martingale change of measure since the stopped
drift difference has bounded energy. Uniqueness of the stopped
experiment identifies the resulting law. Bayes' formula on the complete
stopped record gives the claimed hyperbolic tangent. The initial seed
has the same law under both parameters.

The stopped and original rules agree on the stopped record.
Finite path energy ensures that, almost surely, sufficiently large
caps do not stop before $T$. Boundedness of the posterior and optional
sampling therefore identify the same formula in the original experiment.
In particular $\Xi$ is continuous and finite on the time interval.
The process $I$ is a continuous martingale in the record filtration:
testing increments against past-measurable bounded variables projects
the drift onto $m_sa_s$. The drift is integrable by the energy bound.
Its quadratic variation is $t$, so L\'evy's characterization makes it
Brownian in that filtration. Applying It\^o to $H\tanh(H\Xi)$,
with $d\Xi=m a^2dt+a\,dI$, cancels the two drift terms and proves
\eqref{eq:posterior}. Energy localization followed by bounded-posterior
convergence gives the asserted identity for unbounded controls.
\end{proof}

\subsection{A vector lower bound with one observation channel}
We first take $T=1$. Set $c=H/\sqrt d$ and draw independent uniform
signs $\sigma_i$, with $h_i=c\sigma_i$. All prior support points have
norm $H$. Let $m_t=\mathbb E[h\mid\mathcal F_t]$. For coordinate $i$,
enlarge the filtration only in the proof:
\[
 \mathcal G_t^{(i)}=\mathcal F_t\vee\sigma(h_{-i}),\qquad
 n_{i,t}=\mathbb E[h_i\mid\mathcal G_t^{(i)}].
\]
The original action is unchanged and does not have access to the
revealed coordinates.

Here is the likelihood justification for using a scalar posterior in
this larger filtration. Fix $h_{-i}=k$ and write
$g_t=\sum_{j\ne i}k_ja_{j,t}$.
For the rule stopped at total energy
$\int_0^t\|a_s\|^2ds=K$, use the zero-vector parameter as reference.
The two conditional record densities are
\[
 Z_t^\pm=\exp\left\{\int_0^t(g_s\pm ca_{i,s})\,dY_s
       -\frac12\int_0^t(g_s\pm ca_{i,s})^2ds\right\}.
\]
Their drift energies are at most $H^2K$, so the densities are true
martingales and identify the stopped laws. Their log ratio is
\[
 2cS_{i,t},\qquad
 S_{i,t}=\int_0^t a_{i,s}\,dY_s-\int_0^t a_{i,s}g_s\,ds.
\]
Consequently $n_{i,t}=c\tanh(cS_{i,t})$.
This calculation is on the original complete record plus $k$.
It does not require recovering the original control from the
drift-subtracted process $Y-\int g\,dt$.

The same stopped-record argument as in Lemma~\ref{lemma:posterior}
removes the cap. Drift projection in $\mathcal G^{(i)}$ gives an
innovation Brownian motion $I^{(i)}$, and
\[
 dn_{i,t}=(c^2-n_{i,t}^2)a_{i,t}\,dI_t^{(i)},\qquad n_{i,0}=0.
\]
Different coordinates use different conditional filtrations; their
innovations are not asserted to be independent. It\^o isometry and
conditional Jensen imply
\begin{equation}\label{eq:vector-posterior}
 \mathbb E\|m_t\|^2
 \le\sum_i\mathbb E n_{i,t}^2
 \le c^4\mathbb E\int_0^t\|a_s\|^2ds.
\end{equation}
For the localization step, the left sides converge by boundedness and
the nonnegative stochastic-integral energies by monotone convergence.

Write $q(t)=\mathbb E\int_0^t\|a_s\|^2ds$. This is an absolutely
continuous deterministic function, with $q(0)=0$ and
$q'(t)=\mathbb E\|a_t\|^2$ almost everywhere. Twice applying
Cauchy--Schwarz gives
\[
 \mathbb E\int_0^1\langle m_t,a_t\rangle dt
 \ge-c^2\int_0^1\sqrt{q(t)q'(t)}\,dt
 \ge-\frac{c^2}{\sqrt2}q(1),
\]
because $\int_0^1qq'\,dt=q(1)^2/2$. Thus
\begin{equation}\label{eq:bayes-threshold}
 \mathbb E_\pi\mathcal J_{1,h}(a)
 \ge \frac{H^2}{2}
       +\left(\frac12-\frac{H^2}{d\sqrt2}\right)q(1).
\end{equation}
This proves the lower half of Theorem~\ref{thm:threshold}, including
the equality case. Zero action attains $H^2/2$ over the ball.

\subsection{Strict improvement over the entire ball}
The following uniform calculation uses the scalar policy
\eqref{eq:tanh}. Throughout this subsection $H,\tau$ are fixed and
$e=\varepsilon^2\downarrow0$; no joint small-$\tau$ limit is used.
For any true scalar parameter $v\in[-H,H]$,
\[
 \Xi_\tau=v\varepsilon^2+\varepsilon\rho Z,\qquad
 d\Xi_t=vf(t,\Xi_t)^2dt+f(t,\Xi_t)dB_t,\quad
 f(t,x)=-\frac{\tanh(Hx)}{H\sqrt{2t}},\ t>\tau .
\]
The feedback and drift are globally Lipschitz and bounded on
$[\tau,1]$, uniformly in $v$. The experiment is strongly well posed,
and its stopped versions are also well posed.
For $p=2,3,4$, It\^o localization and Gronwall give
\[
 \sup_{|v|\le H}\sup_{\tau\le t\le1}
 \mathbb E_v|\Xi_t|^p\le C_{p,H,\tau}\varepsilon^p.
\]
Indeed $|f|\le C|x|$, $|f|\le C$, and $|vf^2|\le C|x|$.
The global inequalities
\[
 \left|f(t,x)+\frac{x}{\sqrt{2t}}\right|
 \le\frac{H^2|x|^3}{3\sqrt{2t}},\qquad
 \left|f(t,x)^2-\frac{x^2}{2t}\right|
 \le\frac{H^2|x|^4}{3t}
\]
then yield, for $F_v(t)=\mathbb E_v\Xi_t^2$ and
$q_v(t)=\mathbb E_v\int_0^t a_s^2ds$,
\[
 F_v'=\frac{F_v}{2t}+O_{H,\tau}(\varepsilon^3),\quad
 F_v(\tau)=\varepsilon^2+v^2\varepsilon^4,\quad
 q_v'=\frac{F_v}{2t}+O_{H,\tau}(\varepsilon^4).
\]
Hence both $F_v(t)$ and $q_v(t)$ equal
$\varepsilon^2\sqrt{t/\tau}+O_{H,\tau}(\varepsilon^3)$ uniformly.
The exact identity $\mathbb E_v\Xi_t=vq_v(t)$ gives
$\mathbb E_va_t=-v\varepsilon^2/\sqrt{2\tau}
 +O_{H,\tau}(\varepsilon^3)$ for $t>\tau$.
The probe cross term is exactly zero because its sign is randomized.
Therefore
\begin{equation}\label{eq:uniform-tanh}
 \mathcal L_v(a^{\tau,\varepsilon^2})
 =\frac{v^2}{2}+\frac{\varepsilon^2}{\sqrt\tau}
   \left[\frac12-\frac{v^2(1-\tau)}{\sqrt2}\right]
   +O_{H,\tau}(\varepsilon^3),
\end{equation}
with a remainder uniform on $[-H,H]$.

Choose an independent random unit vector $D$ satisfying
$\mathbb E DD^\top=I_d/d$; choosing a coordinate axis uniformly suffices.
Use $a_t=Db_t$, where $b$ is \eqref{eq:tanh} applied to the same
scalar observation. Conditional on $D=u$, the scalar parameter is
$v=\langle h,u\rangle$. Exactly,
\[
 \frac12\|ub_t+h\|^2
 =\frac{\|h\|^2-v^2}{2}+\frac12(b_t+v)^2 .
\]
Averaging \eqref{eq:uniform-tanh} gives, uniformly on the whole ball,
\[
 \mathcal J_{1,h}(a)
 =\frac{\|h\|^2}{2}+\frac{\varepsilon^2}{\sqrt\tau}
   \left[\frac12-\frac{\|h\|^2(1-\tau)}{d\sqrt2}\right]
   +O_{H,\tau}(\varepsilon^3).
\]
If $H^4>d^2/2$, first fix
$0<\tau<1-d/(\sqrt2H^2)$ and put
$\eta=H^2(1-\tau)/(d\sqrt2)-1/2>0$.
For sufficiently small $\varepsilon$, the coefficient of $\|h\|^2$
in the explicit part is positive and the remainder is at most
$\eta\varepsilon^2/(2\sqrt\tau)$. The worst-case risk is then at most
$H^2/2-\eta\varepsilon^2/(2\sqrt\tau)$.
This proves strict minimax improvement, not just Bayesian improvement.

Finally, the deterministic change of variables
$\widehat Y_s=T^{-1/2}Y_{Ts}$, $b_s=T^{1/4}a_{Ts}$,
$v=T^{1/4}h$ preserves the record and gives
\begin{equation}\label{eq:scaling}
 C_d^T(H)=\sqrt T\,C_d^1(T^{1/4}H).
\end{equation}
This completes the proof of Theorem~\ref{thm:threshold}.

\subsection{A stopped pulse and the role of the full record}
An energy bound does not make the final score Gaussian conditional on its
final energy. In the scalar canonical experiment choose, for fixed $b>0$,
\[
 \sigma=\inf\{t:Y_t=b\}\wedge1,\qquad a_t=1_{\{t<\sigma\}}.
\]
This bounded rule is admissible and has energy $Q_1=\sigma$.
At $h=0$ its score is $\Xi_1=B_\sigma$, and
$\Xi_1=b$ on $\{\sigma<1\}$, an event of positive probability.
Thus conditioning on the energy over this event gives a point mass, not
$N(0,Q_1)$. Its likelihood is nonetheless exactly
$\exp\{hY_\sigma-h^2\sigma/2\}$.

Likewise a control of amplitude $\sqrt n$ over an interval of length
$1/n$ has energy one. The energy measures of such controls may concentrate
at a single physical time. All subsequent arguments retain these controls:
the task clock is distinct from the information clock, and likelihoods
are calculated before any conditioning on a final Gram matrix.

\section{Quantitative comparison of the control tasks}
\label{sec:valueproof}

We prove Theorem~\ref{thm:value}, including its approximation and full-filtration arguments.
Fix integers \(d,N\ge1\), \(0\le H<\infty\), and a nonconstant convex \(g\in C^4(\mathbb R)\)
with \(\|g'\|_\infty\le L\), \(0\le g''\le M\), and bounded \(g^{(3)},g^{(4)}\).
All vector norms are Euclidean.  The local parameters are
\(\theta=N^{-1/4}h\), \(\|h\|\le H\), and must belong to the model's parameter domain.
Each episode starts at zero and satisfies
\[
 dX_{k,t}=\langle\theta,U_{k,t}\rangle\,dt+dW_{k,t},
 \qquad 0\le t\le1,\quad U_{k,t}\in\mathbb R^d.
\]
The scalar Brownian drivers are independent across resets.  Admissible rules use
the entire observed past and a parameter-independent seed, independent of the driver,
which is not separately observed.  Each rule and all its cumulative energy-stopped
versions have nonexplosive unique laws at every comparison parameter.
Records use a common measurable space without parameter-dependent completion.
These conventions also apply to the canonical experiment
\[
 \begin{gathered}
 dY_s=\langle h,a_s\rangle\,ds+dB_s,\\
 \mathcal J_{T,h}(a)=\frac12\mathbb E_h\int_0^T\|a_s+h\|^2\,ds,\qquad
 C_d^T(H)=\inf_a\sup_{\|h\|\le H}\mathcal J_{T,h}(a).
 \end{gathered}
\]
The observation is scalar even when \(d>1\).  We write \(V_\theta(0,0)\) for the
known-parameter, one-episode optimal cost and set
\[
 \begin{gathered}
 r_N(\mathcal A,h)=\frac1{\sqrt N}
 \left\{\mathbb E_h\sum_{k=1}^N
 \left[\frac12\int_0^1\|U_{k,t}\|^2dt+g(X_{k,1})\right]
 -NV_{N^{-1/4}h}(0,0)\right\},\\
 C_{N,d}^g(H)=\inf_{\mathcal A}\sup_{\|h\|\le H}r_N(\mathcal A,h).
 \end{gathered}
\]

\subsection{Verification and uniform analytic estimates}

Write \(P_vf(x)=\mathbb E f(x+\sqrt v Z)\), where \(Z\sim N(0,1)\).
For \(A_k=\int_0^1\|U_{k,t}\|^2dt<\infty\), the Lipschitz bound on \(g\) gives
\begin{equation}\label{val:coercivity}
 \begin{split}
 \frac12A_k+g(X_{k,1})&\ge\frac12A_k+g(0)-L|W_{k,1}|-L\|\theta\|\sqrt{A_k}\\
 &\ge\frac14A_k+g(0)-L|W_{k,1}|-L^2\|\theta\|^2.
 \end{split}
\end{equation}
Thus the cost has an integrable lower bound, even if \(g\) is unbounded below.
Infinite expected energy implies infinite expected cost; infinite path energy
already gives infinite running cost when the nonexplosive terminal state is finite.
Conversely, finite expected energy makes the terminal cost absolutely integrable.
Zero control has finite cost and the oracle is bounded below, so finite-energy analysis discards no finite-risk algorithm.

Put \(\lambda=\|\theta\|^2\).  For \(\lambda>0\), define
\[
 V_\theta(t,x)=-\lambda^{-1}\log P_{1-t}(e^{-\lambda g})(x);
 \qquad V_0(t,x)=P_{1-t}g(x).
\]
Linear growth gives Gaussian exponential integrability; the four bounded derivatives justify the classical equation
\begin{equation}\label{val:hjb}
 V_t+\tfrac12V_{xx}-\tfrac{\lambda}{2}V_x^2=0,\qquad V(1,x)=g(x),
 \quad
 \inf_u\{\tfrac12\|u\|^2+p\langle\theta,u\rangle\}
 =-\tfrac{\lambda}{2}p^2.
\end{equation}
Let \(p_\theta=V_{\theta,x}\) and \(r_\theta=V_{\theta,xx}\).
Under the probability tilt proportional to
\(e^{-\lambda g(x+\sqrt{1-t}Z)}\), differentiation gives
\[
 p_\theta=\mathbb E_\lambda g',\qquad
 r_\theta=\mathbb E_\lambda g''-\lambda\operatorname{Var}_\lambda(g'),
 \qquad |p_\theta|\le L,\quad |r_\theta|\le M+\lambda L^2.
\]
These bounds make \(-\lambda p_\theta\) globally Lipschitz.  Twice differentiating \eqref{val:hjb} gives
\[
 r_t+\tfrac12r_{xx}-\lambda p_\theta r_x-\lambda r^2=0.
\]
For \(d\Xi_u=-\lambda p_\theta(u,\Xi_u)du+dB_u\), bounded-coefficient
Feynman--Kac, applied with the already bounded potential \(-\lambda r\), yields
\[
 r_\theta(t,x)=\mathbb E_{t,x}
 \left[e^{-\lambda\int_t^1r_\theta(u,\Xi_u)du}g''(\Xi_1)\right].
\]
First this proves \(r_\theta\ge0\); then the exponential is at most one, giving \(r_\theta\le M\) for every finite parameter.
With \(S(t,x)=P_{1-t}g'(x)\), the once-differentiated equation gives
\begin{equation}\label{val:gradient}
 \begin{split}
 p_\theta(t,x)-S(t,x)
 &=-\lambda\int_t^1P_{u-t}(p_\theta(u,\cdot)r_\theta(u,\cdot))(x)\,du,\\
 |S|&\le L,\quad 0\le S_x\le M,\quad
 |p_\theta-S|\le\|\theta\|^2LM(1-t).
 \end{split}
\end{equation}
For any finite-energy history-dependent control, It\^o's formula in each episode
has an integrable drift and a square-integrable martingale, since \(|p_\theta|\le L\).
The linear growth of \(V_\theta\) and
\(\mathbb E\sup_{t\le1}|X_{k,t}|^2<\infty\) justify removal of localization.
Completing the square and taking expectations gives
\begin{equation}\label{val:regret}
 r_N(h)=\frac1{2\sqrt N}\mathbb E_h\sum_k\int_0^1
 \|U_{k,t}+N^{-1/4}h\,p_{N^{-1/4}h}(t,X_{k,t})\|^2dt.
\end{equation}
The bounded Lipschitz feedback \(U=-\theta p_\theta\) attains zero excess cost, verifying the oracle in the full class.
Conditioning on preceding episodes proves the identity for adaptive resets.

Define
\[
 e_N=N^{-1/2}\sum_k\int_0^1\|U_{k,t}\|^2dt,\quad B_h=\mathbb E_he_N,\quad
 D_N=N^{-1}\sum_k\int_0^1S(t,X_{k,t})^2dt,
\]
\[
 \widetilde r_N(h)=\frac1{2\sqrt N}\mathbb E_h\sum_k\int_0^1
 \|U_{k,t}+N^{-1/4}hS(t,X_{k,t})\|^2dt.
\]
The inequality \(\|u\|^2\le2\|u+\theta p_\theta\|^2+2\|\theta\|^2L^2\)
and expansion of the two squares, using
\(\mathbb E_h\sum_k\int\|U_{k,t}\|dt\le N^{3/4}\sqrt{B_h}\), imply
\begin{equation}\label{val:errors}
 \begin{split}
 B_h&\le4r_N(h)+2\|h\|^2L^2,\\
 |r_N(h)-\widetilde r_N(h)|&\le\frac{\|h\|^3LM\sqrt{B_h}+\|h\|^4L^2M}{\sqrt N}.
 \end{split}
\end{equation}
The heat martingale and It\^o isometry identify the task constant:
\[
 g(W_1)-\mathbb E g(W_1)=\int_0^1S(t,W_t)dW_t,\qquad
 c_g=\operatorname{Var}(g(Z))=\mathbb E\int_0^1S(t,W_t)^2dt\in(0,L^2].
\]
Positivity follows from continuity, nonconstancy, and full Gaussian support.  Coupling to the actual drivers gives
\[
 |X_{k,t}-W_{k,t}|
 \le N^{-1/4}\|h\|\int_0^t\|U_{k,u}\|du
\]
and \(S^2\) is \(2LM\)-Lipschitz, the expected difference between \(D_N\)
and \(N^{-1}\sum_k\int_0^1S(t,W_{k,t})^2dt\) is at most
\(2LM\|h\|\sqrt{B_h}/\sqrt N\).
The latter summands are independent, lie in \([0,L^2]\), and have mean \(c_g\).
Their variance bound and Cauchy--Schwarz therefore give
\begin{equation}\label{val:clock-concentration}
 \mathbb E_h|D_N-c_g|
 \le\frac{L^2+2LM\|h\|\sqrt{B_h}}{\sqrt N}.
\end{equation}
Only the first total-energy moment was used; individual episodes and short intervals need not have small energy.

\subsection{Weak controls and causal standard implementations}

Let \(Y\) be scalar Brownian motion relative to the \emph{entire} reference filtration \((\mathcal F_s)\).
For predictable \(v\in\mathbb R^d\), write
\[
 \begin{gathered}
 \mathsf M_t^v=\int_0^tv_s\,dY_s,\quad
 \mathsf Q_t^v=\int_0^tv_sv_s^\top ds,\quad q_t^v=\operatorname{tr}\mathsf Q_t^v,\\
 L_h^v(t)=\exp\{h^\top\mathsf M_t^v-\tfrac12h^\top\mathsf Q_t^vh\}.
 \end{gathered}
\]
A weak strategy requires, for every \(\|h\|\le H\), finite expected energy,
finite worst-case risk, and the full-record density
\(\left.d\mathbb P_h^v/d\mathbb P_0\right|_{\mathcal F_t}=L_h^v(t)\),
where \(L_h^v\) is a true martingale.  Auxiliary information must obey this density
and Brownian property, supplying neither another parameter likelihood nor future noise.

\begin{lemma}[Weak-to-standard comparison]\label{val:weak}
For every weak strategy \(a\) on \([0,T]\) and \(\eta>0\), there exists a bounded
standard strategy \(b\), constant on a finite deterministic grid apart from one total-energy cap, such that
\(\sup_{\|h\|\le H}\{\mathcal J_{T,h}(b)-\mathcal J_{T,h}(a)\}<\eta\).
The weak, standard, and bounded grid-with-cap classes have the same minimax value.
\end{lemma}
\begin{proof}
First, standard strategies belong to the weak class.  For
\(\sigma_\kappa=\inf\{t:q_t^a\ge\kappa\}\wedge T\), the stopped stochastic logarithm
has bracket at most \(H^2\kappa\), so its exponential is a true martingale.
Girsanov and uniqueness of the energy-stopped experiment identify its law on
\(\mathcal F_{\sigma_\kappa}\) with the corresponding original parameter law.
The seed distribution is unchanged, since the density has conditional mean one
given \(\mathcal F_0\); the new Brownian driver is independent of that seed.
For \(A\in\mathcal F_T\), \(A\cap\{q_T^a<\kappa\}\in\mathcal F_{\sigma_\kappa}\):
its intersection with \(\{\sigma_\kappa\le t\}\) is empty for \(t<T\).  Consequently
\[
 \mathbb P_h(A\cap\{q_T^a<\kappa\})
 =\mathbb E_0[1_A1_{\{q_T^a<\kappa\}}L_h^a(T)].
\]
Finite path energy under both laws and monotone convergence give the full density with mean one.
The nonnegative local martingale loses no mass and is a true martingale;
its positive terminal value gives equivalent laws and permits a common completion.

For a weak strategy, set \(a^\kappa=a1_{\{s\le\sigma_\kappa\}}\) and
\emph{redefine} its laws by \(d\mathbb P_h^{a^\kappa}=L_h^a(\sigma_\kappa)d\mathbb P_0\).
Optional stopping identifies only the stopped records with those of \(a\).
Writing \(R=\sup_h\mathcal J_{T,h}(a)\) and \(\mathfrak M=4R+2H^2T\), we have
\(\sup_h\mathbb E_h^a q_T^a\le\mathfrak M\).  The original tail loss is nonnegative,
whereas the new zero-action tail is at most \(H^2T/2\); thus
\begin{equation}\label{val:cap}
 \mathcal J_{T,h}(a^\kappa)\le\mathcal J_{T,h}(a)
 +\tfrac12H^2T\,\mathbb P_h^a(q_T^a\ge\kappa)
 \le\mathcal J_{T,h}(a)+\frac{H^2T\mathfrak M}{2\kappa}.
\end{equation}
This is a one-sided comparison, requiring no uniform integrability of original energies.

Fix \(\kappa\), put \(v=a^\kappa\), and \(C=\kappa+1\).
Predictable rectangles \((s,t]\times A\), \(A\in\mathcal F_s\), generate the predictable sigma-field.
Finite simple approximation and truncation give bounded grid processes \(u^n\) with
\(\mathbb E_0\int\|u^n-v\|^2ds\to0\).
Stop \(u^n\) at its \emph{own} energy \(C\), obtaining \(c^n\).
Then \(q_T^{u^n}\to q_T^v\) in \(L^1\), so
\(\mathbb P_0(E_n)\to0\) for \(E_n=\{q_T^{u^n}\ge C\}\), and
\[
 q_T^{c^n}\le C,\qquad
 \mathbb E_0\int\|c^n-v\|^2ds
 \le\mathbb E_0\int\|u^n-v\|^2ds+(2C+2\kappa)\mathbb P_0(E_n)\longrightarrow0.
\]
For every new control \(w=v,c^n\), use the new law
\(d\mathbb P_h^w=L_h^w(T)d\mathbb P_0\); the deterministic cap guarantees its validity.
Vector It\^o isometry and the operator norm inequality imply
\[
 \begin{gathered}
 \mathbb E_0\|\mathsf M_T^{c^n}-\mathsf M_T^v\|^2
 =\mathbb E_0\int\|c^n-v\|^2ds\to0,\\
 \|\mathsf Q_T^{c^n}-\mathsf Q_T^v\|_{\rm op}
 \le(\sqrt C+\sqrt\kappa)\left(\int\|c^n-v\|^2ds\right)^{1/2}.
 \end{gathered}
\]
The supremum log-likelihood difference is at most
\(H\|\Delta\mathsf M\|+\tfrac12H^2\|\Delta\mathsf Q\|_{\rm op}\to0\) in probability.
Choose a \(1/2\)-net \(\mathcal N\) of the unit sphere with
\(|\mathcal N|\le5^d\): a maximal separated set has disjoint radius-\(1/4\)
balls inside the radius-\(5/4\) ball.  Then \(\|x\|\le2\max_{u\in\mathcal N}u^\top x\).
For \(q_T^w\le C\) and \(p>1\), unit-projection exponential martingales, with brackets bounded by \(C\), give
\begin{equation}\label{val:net}
 \begin{split}
 \mathbb E_0\sup_{\|h\|\le H}(L_h^w(T))^p
 &\le\mathbb E_0e^{pH\|\mathsf M_T^w\|}\\
 &\le\sum_{u\in\mathcal N}\mathbb E_0e^{2pH u^\top\mathsf M_T^w}
 \le5^d e^{2p^2H^2C}.
 \end{split}
\end{equation}
Without conditioning on the random Gram matrix, tight log-likelihood suprema,
continuity of the exponential, and \eqref{val:net} give
\(\mathbb E_0\sup_h|L_h^{c^n}(T)-L_h^v(T)|\to0\).
Moreover
\[
 \begin{gathered}
 \Gamma_h(w)=\tfrac12q_T^w+h^\top\int_0^Tw_sds+\tfrac12\|h\|^2T
 \in[0,C+H^2T],\\
 \Delta_n:=\sup_h|\Gamma_h(c^n)-\Gamma_h(v)|\longrightarrow0
 \end{gathered}
\]
in every finite \(L^p(\mathbb P_0)\), by the preceding norm approximation and boundedness.
With \(D=C+H^2T\), splitting the likelihood and loss differences gives
\[
 \sup_h|\mathcal J_{T,h}(c^n)-\mathcal J_{T,h}(v)|
 \le D\mathbb E_0\sup_h|L_h^{c^n}-L_h^v|
 +\|\sup_hL_h^v\|_{L^2}\|\Delta_n\|_{L^2}\longrightarrow0.
\]

To implement a fixed \(c^n\), write its candidate coefficients as \(\beta_j\in\mathbb R^d\)
on \((t_j,t_{j+1}]\).  The standard Borel history space admits reference conditional kernels
\[
 K_j(db\mid\mathsf H_j),\qquad
 \mathsf H_j=(Y_{[0,t_j]},\beta_0,\ldots,\beta_{j-1}).
\]
Choose bounded versions on null histories.  Independence of the next Brownian segment
from \(\mathcal F_{t_j}\) factors the joint law into these kernels and Wiener segment laws.
Preload independent uniforms; measurable kernel sampling at \(t_j\) uses its own uniform and only \(\mathsf H_j\).
Induction over segments reproduces the joint law of \((Y,\beta_0,\ldots,\beta_{m-1})\).
If \(q_j\) is the energy already spent, the execution endpoint \(r_j\) is \(t_j\)
when \(q_j=C\), is \(t_{j+1}\) when \(q_j<C,\beta_j=0\), and otherwise is
\[
 r_j=t_j+\min\{t_{j+1}-t_j,(C-q_j)/\|\beta_j\|^2\}.
\]
It is known at \(t_j\).  After the cap, still sample candidates internally to preserve the reference record, but execute zero.
On this record the score, Gram matrix, and action integral are the Borel functions
\[
 \mathsf M_T=\sum_j\beta_j(Y_{r_j}-Y_{t_j}),\quad
 \mathsf Q_T=\sum_j\beta_j\beta_j^\top(r_j-t_j),\quad
 \int_0^Tc_sds=\sum_j\beta_j(r_j-t_j).
\]
Tilting the common reference law by these same likelihood and loss functions matches every parameter risk.
In the standard filtration \(\mathcal G_t=\sigma(\text{entire seed},Y_{[0,t]})\),
the capped density is a true martingale and
\(\mathbb E_0[L_h(T)\mid\mathcal G_0]=1\).
The entire seed retains its product-uniform law, and Girsanov makes
\(Y-\int_0^\cdot\langle h,c_s\rangle ds\) Brownian relative to \(\mathcal G\),
hence independent of that entire seed.  The observation itself need not be independent.
Given seed and driver, each segment has its chosen constant drift up to \(r_j\), then zero drift.
Recursion gives a unique nonexplosive strong solution, also under further caps; measurable kernels suffice.
Finally choose \(\kappa\) large in \eqref{val:cap}, then \(n\) large, to obtain
the stated one-sided approximation.  The class inclusions and this approximation
give equality of infima without a minimizer or a general weak-to-strong representation.
\end{proof}

\subsection{The task clock and its full information}

Fix an original algorithm of finite worst-case risk.  Concatenate episode
increments into \(\bar X_r\), \(0\le r\le N\), retaining the full record filtration
\(\mathcal F_r\); the reset state is \(X_{k,t}=\bar X_{k-1+t}-\bar X_{k-1}\).
The localization proof in Lemma~\ref{val:weak}, applied on \([0,N]\) with
integrand \(N^{-1/4}U_r\), gives the true full-record density
\begin{equation}\label{val:physical-density}
 \Lambda_h(r)=\exp\left\{
 h^\top N^{-1/4}\int_0^rU_u\,d\bar X_u
 -\tfrac12h^\top\left(N^{-1/2}\int_0^rU_uU_u^\top du\right)h\right\}.
\end{equation}
Equation~\eqref{val:errors} gives finite energies at \(h\) and zero; stopped uniqueness is assumed.
Thus \(\mathbb P_h\) and \(\mathbb P_0\) are equivalent on the full record.
If \(g\) is nonaffine, \(S_x(t,x)=P_{1-t}g''(x)>0\) for \(t<1\);
the positive Gaussian kernel integrates a nonzero nonnegative continuous function.
Thus \(S(t,\cdot)\) has at most one zero.  Reset states at \(t>0\) are atomless
under \(\mathbb P_0\), hence under every \(h\).  Tonelli, ignoring finitely many endpoints, gives
\[
 \int_0^N1_{\{S_r=0\}}dr=0,\qquad
 \int_0^N1_{\{S_r=0\}}\|U_r\|^2dr=0,\qquad S_r=S(t(r),X_r^{\rm ep}).
\]
Finite path energy gives the second equality; zero bracket and drift integral make the score there vanish.
For nonconstant affine \(g\), \(S\) is a nonzero constant and the same conclusions hold.
Consequently \(\tau(r)=N^{-1}\int_0^rS_u^2du\) is continuous and strictly increasing,
with \(\tau(N)=D_N\le L^2\); no lower bound on \(|S|\) is needed.

Add an independent scalar Brownian \(V\) only after physical time \(N\), setting
\[
 \bar\tau(r)=
 \begin{cases}\tau(r),&r\le N,\\D_N+r-N,&r>N,\end{cases}
 \quad
 \mathcal Y_r=
 \begin{cases}N^{-1/2}\int_0^rS_u\,d\bar X_u,&r\le N,\\
 \mathcal Y_N+V_{r-N},&r>N.\end{cases}
\]
The extended filtration is \(\bar{\mathcal F}_r=\mathcal F_r\) for \(r\le N\),
and \(\mathcal F_N\vee\sigma(V_u:u\le r-N)\) thereafter; it never reveals
\(\mathcal F_N\) early.  Define
\[
 \rho_s=\inf\{r:\bar\tau(r)>s\},\qquad
 \mathcal G_s=\bar{\mathcal F}_{\rho_s},\qquad Y_s=\mathcal Y_{\rho_s}.
\]
Continuity and strict increase give \(\bar\tau(\rho_s)=s\) and
\(\{\rho_s\le r\}=\{\bar\tau(r)\ge s\}\); hence \(\rho_s\le N+s\) is a bounded
physical stopping time.  Also
\(\{D_N\le s\}=\{\rho_s\ge N\}\in\mathcal G_s\), so \(D_N\) is a task stopping time.
The physical martingale
\(\mathcal Y_r-N^{-3/4}\int_0^{r\wedge N}S_u\langle h,U_u\rangle du\)
has bracket \(\bar\tau(r)\).  Time change and L\'evy's characterization therefore
give a Brownian \(B^h\) relative to the \emph{whole} \(\mathcal G\), independent of its initial seed.
Define \(\alpha_s=N^{1/4}U_{\rho_s}/S_{\rho_s}\) for \(s\le D_N,S_{\rho_s}\ne0\),
and zero otherwise.  Continuous stopping-time change preserves predictability;
the stochastic substitution follows first for simple integrands and then by local
\(L^2(d\tau)\) approximation.  The zero-set calculation prevents lost energy or score.
For every \(s\), substitution gives
\begin{equation}\label{val:substitution}
 \begin{split}
 dY_s&=\langle h,\alpha_s\rangle ds+dB_s^h,\\
 \int_0^s\alpha_u\,dY_u&=N^{-1/4}\int_0^{\rho_s\wedge N}U_r\,d\bar X_r,\\
 \int_0^s\alpha_u\alpha_u^\top du
 &=N^{-1/2}\int_0^{\rho_s\wedge N}U_rU_r^\top dr.
 \end{split}
\end{equation}
In particular \(\int_0^{D_N}\|\alpha_s\|^2ds=e_N\) and
\(\widetilde r_N(h)=\tfrac12\mathbb E_h\int_0^{D_N}\|\alpha_s+h\|^2ds\).
Optional sampling of the extended density \(\Lambda_h(r\wedge N)\) at
the bounded physical stopping time \(\rho_s\) proves
\begin{equation}\label{val:task-density}
 \left.\frac{d\mathbb P_h}{d\mathbb P_0}\right|_{\mathcal G_s}
 =\Lambda_h(\rho_s\wedge N)
 =\exp\left\{h^\top\int_0^s\alpha_u\,dY_u
 -\tfrac12h^\top\left(\int_0^s\alpha_u\alpha_u^\top du\right)h\right\}.
\end{equation}
This true density covers inverse physical time, the termination flag, and all retained observations.
It asserts neither Gaussianity conditional on a random Gram matrix nor equality with the natural filtration of \(Y\).

\subsection{Two comparisons and the uniform constant}

\begin{proof}[Completion of the proof of Theorem~\ref{thm:value}]
Put \(T=c_g>0\).  For the lower comparison, set \(\zeta=D_N\wedge T\),
take a further independent Brownian \(\beta\), and use precisely
\[
 \begin{gathered}
 \mathcal H_s=\mathcal G_{s\wedge\zeta}
 \vee\sigma\{\beta_{(u-\zeta)_+}:0\le u\le s\},\\
 Y_s^{\rm cut}=Y_{s\wedge\zeta}+\beta_{(s-\zeta)_+},\qquad
 \alpha_s^{\rm cut}=\alpha_s1_{\{s\le\zeta\}},\quad 0\le s\le T.
 \end{gathered}
\]
Use common usual augmentation.  All original information freezes at \(\zeta\), or physical time \(\rho_\zeta\le N\).
The spliced innovation is a continuous martingale of bracket
\(s\wedge\zeta+(s-\zeta)_+=s\), hence Brownian relative to \(\mathcal H\).
Independent splicing adds no likelihood: \eqref{val:task-density} stops at
\(\Lambda_h(\rho_{s\wedge\zeta})\), exactly the score/Gram exponential of \(\alpha^{\rm cut}\).
Thus this is a weak strategy, with energy at most \(e_N\) and risk at most
\(B_h+H^2T\).  Nonnegative discarded loss and the zero-action extension give
\begin{equation}\label{val:lower-prefix}
 \mathcal J_{T,h}(\alpha^{\rm cut})
 \le\widetilde r_N(h)+\tfrac12\|h\|^2\mathbb E_h(T-D_N)_+.
\end{equation}
Lemma~\ref{val:weak} now bounds \(C_d^T(H)\) by the supremum of this right-hand side.

For the upper comparison, choose by that lemma a bounded grid-with-cap standard
strategy \(a\) with worst-case risk at most \(C_d^T(H)+\eta\), \(0<\eta\le1\),
and extend its action by zero after \(T\).  Execute in physical time
\begin{equation}\label{val:recovery}
 U_r=N^{-1/4}S_r a_{\tau(r)}.
\end{equation}
This uses only past observations: with \(F(t,x)=P_{1-t}g(x)\) and \(\mu=F(0,0)\),
It\^o's formula gives the parameter-independent observable path functional
\[
 \mathcal Y_{k-1+t}=N^{-1/2}
 \left[\sum_{j<k}(g(X_{j,1})-\mu)+F(t,X_{k,t})-\mu\right].
\]
At each task-grid hitting time, the available past path and seed select the next coefficient \(b\).
Until the next grid point, reset, or energy cap, the scalar state drift is
\(N^{-1/2}\langle h,b\rangle S(t,x)\), bounded and globally Lipschitz in \(x\).
Finite successive strong constructions give nonexplosion and pathwise uniqueness, also under further energy stopping.
The finite bound on \(b\) proves admissibility only, not the eventual error estimate.
On the proof space, solve the standard grid strategy on all of \([0,T]\) using
the time-changed \(B^h\) and the same independent seed, obtaining \((Y^{\rm G},a^{\rm G})\).
Segmentwise uniqueness identifies it with the recovered path up to \(D_N\wedge T\).
If \(D_N<T\), continue with the independent post-endpoint Brownian, not supplied to the original algorithm.
The complete Gaussian law follows from the driver and seed, without conditioning on \(D_N\)
or requiring it to be a natural-filtration stopping time.
Pathwise prefix comparisons consequently yield
\begin{equation}\label{val:upper-prefix}
 \begin{split}
 e_N&=\int_0^{D_N\wedge T}\|a_s^{\rm G}\|^2ds
       \le\int_0^T\|a_s^{\rm G}\|^2ds,\\
 \widetilde r_N(h)&\le\mathcal J_{T,h}(a)
       +\tfrac12\|h\|^2\mathbb E_h(D_N-T)_+.
 \end{split}
\end{equation}
Only the proxy loss has this exact task-time interpretation.

Zero control bounds both values by \(H^2L^2/2\).
For the lower comparison, consider rules satisfying
\[
 \sup_h r_N(h)\le H^2L^2/2+1.
\]
Larger risks already satisfy the desired lower bound.
Equation~\eqref{val:errors} gives \(B_h\le B_*\), where
\[
 B_*:=4H^2L^2+4.
\]
For the recovery, its complete Gaussian energy satisfies independently
\[
 \mathbb E_h\int_0^T\|a_s\|^2ds
 \le4\mathcal J_{T,h}(a)+2H^2T
 \le4(C_d^T(H)+1)+2H^2T\le B_*.
\]
Thus \eqref{val:upper-prefix} bounds original energy without circular use of original risk.
Adding \eqref{val:errors} and half \(H^2\) times \eqref{val:clock-concentration} gives, in either direction,
\[
 \begin{split}
 K_{H,L,M}&=H^3LM\sqrt{B_*}+H^4L^2M+\tfrac12H^2(L^2+2LMH\sqrt{B_*})\\
 &=2H^3LM\sqrt{B_*}+H^4L^2M+\tfrac12H^2L^2.
 \end{split}
\]
Taking infima in \eqref{val:lower-prefix} and \eqref{val:upper-prefix}, and then
letting \(\eta\downarrow0\) for each fixed \((d,N)\), proves
\[
 |C_{N,d}^g(H)-C_d^{c_g}(H)|\le K_{H,L,M}/\sqrt N.
\]
The net constant affects only approximation precision, introducing no dimension, grid, amplitude, or cap factor into \(K_{H,L,M}\).
No minimax interchange, optimizer, uniform energy integrability, or trajectory limit is used.
Finally, the invertible deterministic change
\[
 \widehat Y_t=T^{-1/2}Y_{Tt},\qquad
 \widehat a_t=T^{1/4}a_{Tt},\qquad \widehat h=T^{1/4}h
\]
preserves the scalar observation, admissibility, and energy stopping.
Its loss is multiplied by \(\sqrt T\), so
\(C_d^T(H)=\sqrt T\,C_d^1(T^{1/4}H)\), completing the claimed comparison.
For \(H=0\) both values are zero.  If \(g(x)=b+mx\), \(m\ne0\), then
\(p_\theta=S=m\) and \(D_N=c_g=m^2\), so both error sources vanish and
\(C_{N,d}^g(H)=C_d^{m^2}(H)\) exactly.  Constant \(g\) instead gives zero minimax
value directly from the action cost; no division by \(S\) is made in that case.
\end{proof}

\section{Near-Critical Bounds and the Full Parameter Interval}
\label{crit:sec:bounds}

We work in the scalar canonical experiment
\[
 dY_t=ha_t\,dt+dB_t,\qquad
 \mathcal L_h(a)=\frac12 \mathbb E_h\int_0^1(a_t+h)^2\,dt,\qquad
 C_1^1(H)=\inf_a\sup_{|h|\le H}\mathcal L_h(a).
\]
The admissible class retains the full observation history, independent parameter-free
randomization, and common well-posedness for the original and energy-stopped experiments.
No common amplitude bound is imposed on this class.  Write
\begin{equation}
 H_c=2^{-1/4},\qquad \delta=\sqrt2H^2-1,\qquad
 \Delta(H)=H^2/2-C_1^1(H),\qquad
 \pi_H=\tfrac12(\delta_{-H}+\delta_H).
 \label{crit:notation}
\end{equation}
Expectations without a parameter subscript are under $\mathbb P_\pi=\tfrac12(\mathbb P_H+\mathbb P_{-H})$.
Set $Q_t=\int_0^t a_s^2\,ds$, $q(t)=\mathbb E_\pi Q_t$, and $Q=q(1)$;
thus $Q_t$ is random, whereas $Q$ is a number.
Finite endpoint risks imply $Q<\infty$, since $\mathbb E_hQ_1\le4\mathcal L_h(a)+2h^2$.
The posterior identity in Lemma~\ref{lemma:posterior} reads
\begin{equation}
 \Xi_t=\int_0^t a_s\,dY_s,\qquad
 m_t=H\tanh(H\Xi_t),\qquad
 dm_t=(H^2-m_t^2)a_t\,dI_t,\qquad m_0=0,
 \label{crit:posterior}
\end{equation}
Here $I$ is Brownian motion in the full mixture filtration, including the seed.
Define the endpoint Bayes improvement
\begin{equation}
 \begin{aligned}
 b_H(a)=\frac{H^2}{2}-\frac{\mathcal L_H(a)+\mathcal L_{-H}(a)}2
       &=-\frac Q2-\mathbb E_\pi\int_0^1m_ta_t\,dt,\\
 \Delta(H)&\le\sup_a b_H(a).
 \end{aligned}
 \label{crit:bayes-direction}
\end{equation}
The last inequality does not assert that the endpoint prior is least favorable.

\subsection{A Potential Bound for All Admissible Algorithms}
\label{crit:subsec:potential}

\begin{lemma}
\label{crit:lem:potential}
For every admissible algorithm with finite endpoint risks, let $v(t)=\mathbb E_\pi m_t^2$.  Then
\begin{equation}
 v(t)\le\kappa_H(q(t)),\qquad
 \kappa_H(u)=\frac{H^4u}{1+H^2u/3},
 \label{crit:potential-bound}
\end{equation}
and, with $z=H^2Q/3$,
\begin{equation}
 b_H(a)\le3\sqrt{z-\log(1+z)}-\frac Q2.
 \label{crit:scalar-energy-bound}
\end{equation}
If $0<\delta<1/2$ and $b_H(a)>0$, then
\begin{equation}
 \begin{gathered}
 Q<\frac{9\delta}{H^2(1-2\delta)},\\
 b_H(a)\le\frac{\delta Q}{2}
       -\frac{H^2(1-2\delta)}{18}Q^2
       \le\frac{9\delta^2}{8H^2(1-2\delta)}.
 \end{gathered}
 \label{crit:positive-energy}
\end{equation}
\end{lemma}
\begin{proof}
Put $x_t=m_t/H$ and $f(x)=x\operatorname{atanh}x$.
Then $f\ge0$, $f''(x)=2/(1-x^2)^2$, and
$dx_t=H(1-x_t^2)a_t\,dI_t$.
Stop at $T_n=\inf\{t:Q_t\ge n\text{ or }|x_t|\ge1-1/n\}\wedge1$, for $n\ge2$.
The stopped stochastic integral in It\^o's formula is square integrable:
$f'$ is bounded on the stopped state interval and $Q_{T_n}\le n$.  Thus
\[
 \mathbb E_\pi f(x_{t\wedge T_n})=H^2\mathbb E_\pi Q_{t\wedge T_n}.
\]
Finite path energy makes $\Xi$ finite and continuous on $[0,1]$, so
$\sup_{t\le1}|x_t|<1$ almost surely and $T_n=1$ eventually on each such path.
Fatou's lemma gives $\mathbb E_\pi f(x_t)\le H^2q(t)$; no uniform integrability
of the stopped potentials is asserted.
Coefficient comparison using $3^{n-1}\ge2n-1$ gives
\[
 f(x)=\sum_{n\ge1}\frac{x^{2n}}{2n-1}
       \ge\frac{x^2}{1-x^2/3}.
\]
Convexity of $y\mapsto y/(1-y/3)$ on $[0,1]$ and Jensen's inequality yield
$\{v(t)/H^2\}/\{1-v(t)/(3H^2)\}\le H^2q(t)$,
which is \eqref{crit:potential-bound}.
Since $q$ is absolutely continuous with $q'(t)=\mathbb E_\pi a_t^2$ almost everywhere,
two Cauchy--Schwarz inequalities give
\[
 b_H(a)\le-\frac Q2+\int_0^1\sqrt{\kappa_H(q(t))q'(t)}\,dt
 \le-\frac Q2+\sqrt{\int_0^Q\kappa_H(u)\,du}.
\]
The last integral equals $9\{z-\log(1+z)\}$, by the chain rule for absolutely
continuous $q$; strict monotonicity is unnecessary.
For $z\ge0$,
\[
 z-\log(1+z)
 =\frac{z^2}{2}-\int_0^z\frac{u^2}{1+u}\,du
 \le\frac{z^2}{2}-\frac{z^3}{3(1+z)}.
\]
Applying $\sqrt{1-w}\le1-w/2$ with $w=2z/[3(1+z)]$ gives
\begin{equation}
 b_H(a)\le\frac Q2
       \left\{\delta-\frac{(1+\delta)z}{3(1+z)}\right\}.
 \label{crit:energy-penalty}
\end{equation}
Positive improvement forces $(1-2\delta)z<3\delta$.
It follows that $(1+\delta)/(1+z)>1-2\delta$, proving the quadratic
bound in \eqref{crit:positive-energy}.
Its maximum occurs at $Q=9\delta/[2H^2(1-2\delta)]$.
Algorithms with $b_H(a)\le0$ automatically satisfy the final positive
upper bound; infinite endpoint risk cannot improve the baseline.
No global Novikov condition or conditioning on $Q_1$ was used.
\end{proof}

\subsection{Uniform Two-Parameter Bounds for the Constructive Policy}
\label{crit:subsec:moments}

The policy of \eqref{eq:tanh}, with arbitrary $0<\tau<1$ and $e>0$, is
\begin{equation}
 a_t^{\tau,e}=
 \begin{cases}
  \sqrt{e/\tau}\,\rho,&0\le t\le\tau,\\
  -\dfrac{\tanh(H\Xi_t)}{H\sqrt{2t}}
  =-\dfrac{m_t}{H^2\sqrt{2t}},&\tau<t\le1,
 \end{cases}
 \label{crit:policy}
\end{equation}
Here $\rho$ is an independent, symmetric Rademacher seed.
For interior parameters, $m_t$ denotes the same observable function, not the
posterior mean under $\mathbb P_h$.  For fixed design parameters the score equation has
bounded, globally Lipschitz coefficients after $\tau$, so the policy is strongly
well posed, including after energy stopping.  Pathwise,
\begin{equation}
 Q_1\le\Lambda:=e+\frac{\log(1/\tau)}{2H^2}.
 \label{crit:path-cap}
\end{equation}
The probe cross term vanishes for every $h$:
\[
 \mathbb E_h\int_0^\tau a_t\,dt=\sqrt{e\tau}\mathbb E\rho=0.
\]
Changing $(h,\rho,\Xi,a)$ to $(-h,-\rho,-\Xi,-a)$ leaves the observation
equation and the seed law unchanged.  Thus $\mathcal L_h(a^{\tau,e})$ is even in $h$.

\begin{lemma}
\label{crit:lem:moments}
For all $e>0$, $0<\tau<1$, and $t\in[\tau,1]$, set $M_j(t)=\mathbb E_\pi m_t^j$ for even $j$ and
$D(t)=H^4e\sqrt{t/\tau}-M_2(t)$.  Then
\begin{align}
 0\le H^4e-M_2(\tau)&\le H^6e^2,&
 M_4(\tau)&\le3H^8e^2, \label{crit:probe-moments}\\
 M_4(t)&\le3H^8e^2(t/\tau)^3,&
 0\le D(t)&\le\frac65H^6e^2\tau^{-3}t^3. \label{crit:joint-moments}
\end{align}
In particular, $Q\le e/\sqrt\tau$, and, with $K_H=12H^4/(35\sqrt2)$,
\begin{equation}
 \begin{aligned}
 \mathcal L_H(a^{\tau,e})&=\mathcal L_{-H}(a^{\tau,e})\\
 &\le\frac{H^2}{2}+\frac e{\sqrt\tau}
       \left\{\frac12-\frac{H^2(1-\tau)}{\sqrt2}\right\}
       +K_He^2\tau^{-3}.
 \end{aligned}
 \label{crit:joint-risk}
\end{equation}
\end{lemma}
\begin{proof}
During the probe, It\^o's formula gives
$M_2'(t)=(e/\tau)\mathbb E_\pi(H^2-m_t^2)^2\le H^4e/\tau$.  Consequently,
\[
 H^4e-M_2(\tau)
 =\frac e\tau\int_0^\tau \mathbb E_\pi(2H^2m_t^2-m_t^4)\,dt
 \le\frac{2H^2e}{\tau}\int_0^\tau\frac{H^4et}{\tau}\,dt
 =H^6e^2.
\]
Also $M_4'(t)\le6H^4(e/\tau)M_2(t)$ during the probe,
which proves the second assertion of \eqref{crit:probe-moments}.
After the probe the exact equations are
\[
 M_2'=\frac{M_2}{2t}-\frac{M_4}{H^2t}+\frac{M_6}{2H^4t},
 \qquad
 M_4'=\frac3t\mathbb E_\pi[m_t^4(1-m_t^2/H^2)^2]\le\frac3tM_4.
\]
Since $m_t^2\le H^2$, $N(t)=M_4(t)/(H^2t)-M_6(t)/(2H^4t)$ lies between zero and
$M_4(t)/(H^2t)$.  Gronwall and variation of constants yield
\[
 \begin{split}
 0\le D(t)
 &=\sqrt{t/\tau}D(\tau)+\sqrt t\int_\tau^tN(s)s^{-1/2}\,ds\\
 &\le H^6e^2\sqrt{t/\tau}
   +\frac{6H^6e^2}{5\tau^3}(t^3-\sqrt t\,\tau^{5/2})
 \le\frac65H^6e^2\tau^{-3}t^3.
 \end{split}
\]
Thus $M_2(t)\le H^4e\sqrt{t/\tau}$, and integration of the feedback
energy gives $Q\le e+e(1-\sqrt\tau)/\sqrt\tau=e/\sqrt\tau$.
The exact endpoint risk identity is
\begin{equation}
 \frac{\mathcal L_H+\mathcal L_{-H}}2-\frac{H^2}{2}
 =\frac e2+\int_\tau^1
   \left(\frac1{4H^4t}-\frac1{H^2\sqrt{2t}}\right)M_2(t)\,dt.
 \label{crit:exact-policy-risk}
\end{equation}
Insert $M_2=H^4e\sqrt{t/\tau}-D$.
The linear contribution is the middle term in \eqref{crit:joint-risk};
the remaining contribution is at most
\[
 \int_\tau^1\frac{D(t)}{H^2\sqrt{2t}}\,dt
 \le\frac{6H^4e^2}{5\sqrt2\tau^3}\int_0^1t^{5/2}\,dt
 =K_He^2\tau^{-3}.
\]
All estimates hold simultaneously in $e$ and $\tau$.  The sixth moment is
controlled by its sign and the bounded posterior, not a Taylor approximation.
\end{proof}

\subsection{Endpoint Maximality with Explicit Constants}
\label{crit:subsec:endpoints}

Choose
\begin{equation}
 \tau=\frac{\delta}{2(1+\delta)},\qquad
 e=\frac{\delta\tau^{5/2}}{8K_H},\qquad 0<\delta\le\frac1{16}.
 \label{crit:design}
\end{equation}
The bracket in \eqref{crit:joint-risk} is $-\delta/4$.  Hence
\begin{equation}
 \mathcal L_H(a^{\tau,e})\le\frac{H^2}{2}
 -\frac{\delta^2\tau^2}{64K_H}
 =\frac{H^2}{2}
 -\frac{35\sqrt2}{1536(1+\delta)^4}\delta^4.
 \label{crit:endpoint-gain}
\end{equation}
To exclude larger interior risks, fix the design $H,\tau,e$ while varying $h$.
The cap \eqref{crit:path-cap} makes
$\exp\{h\Xi_1-h^2Q_1/2\}$ a true likelihood relative to $\mathbb P_0$.
Girsanov and uniqueness identify the law of the same policy under $h$.
For $|h|\le H$, using $e^{hx}\le2\cosh(Hx)$,
\[
 \frac{d\mathbb P_h}{d\mathbb P_\pi}
 =\frac{\exp\{h\Xi_1+(H^2-h^2)Q_1/2\}}{\cosh(H\Xi_1)}
 \le2\exp(H^2e/2)\tau^{-1/4}.
\]
Using $\mathbb E_\pi Q_1\le e/\sqrt\tau$ gives
\begin{equation}
 \sup_{|h|\le H}\mathbb E_hQ_1\le U:=2\exp(H^2e/2)e\tau^{-3/4}.
 \label{crit:interval-energy}
\end{equation}
Let $A=\int_0^1a_t\,dt$ and $S_h=\Xi_1-hQ_1=\int_0^1a_t\,dB_t$.
Then $A^2\le Q_1\le\Lambda$, $\mathbb E_hS_h^2=\mathbb E_hQ_1$, and
$\mathbb E_h\exp(\lambda S_h)\le\exp(\lambda^2\Lambda/2)$.
For a bounded parameter-independent path functional $F$,
\[
 \mathbb E_{h+u}F=\mathbb E_h[F\exp\{uS_h-u^2Q_1/2\}],
 \qquad \partial_h\mathbb E_hF=\mathbb E_h(FS_h).
\]
The exponential moment bound justifies differentiation by dominated convergence,
including for $F=A$ and $F=Q_1$.
Since $\mathcal L_h=h^2/2+\mathbb E_hQ_1/2+h\mathbb E_hA$, Cauchy--Schwarz gives
\begin{equation}
 |\mathcal L_h'-h|
 \le\frac12|\mathbb E_hQ_1S_h|+|\mathbb E_hA|+H|\mathbb E_hAS_h|
 \le\sqrt U+(H+\sqrt\Lambda/2)U=:V.
 \label{crit:risk-derivative}
\end{equation}
Here $\mathbb E_hQ_1^2\le\Lambda \mathbb E_hQ_1$ and $\mathbb E_hA^2\le \mathbb E_hQ_1$.
For all parameters in \eqref{crit:design},
\[
 \begin{gathered}
 \frac45<H<1,\quad \frac\delta3\le\tau\le\frac\delta2,\quad
 e<2\delta\tau^{5/2}\le2^{-31/2}<10^{-4},\\
 U<2^{5/4}\delta^{11/4}\le2^{-39/4}<\frac1{800},
 \qquad \Lambda\le1+\log(3/\delta).
 \end{gathered}
\]
Indeed, $1/(8K_H)=35\sqrt2/[48(1+\delta)^2]<2$.
The logarithmic derivative of $\delta^{11/4}\sqrt{1+\log(3/\delta)}$
with respect to $\log\delta$ is $11/4-1/[2(1+\log(3/\delta))]>0$.
Since $\log48<4$, we obtain
$U\sqrt\Lambda<\sqrt5/800<1/300$ and $\sqrt U<1/25$.
Consequently,
\[
 V<\frac1{25}+\frac1{800}+\frac1{600}<\frac H2,\qquad
 \frac U2+2H\sqrt U<\frac1{1600}+\frac2{25}<\frac{3H^2}{8}.
\]
The risk is strictly increasing on $[H/2,H]$ by
\eqref{crit:risk-derivative}.  In the central interval,
\[
 \sup_{|h|\le H/2}\mathcal L_h
 \le\frac{H^2}{8}+\frac U2+H\sqrt U
 <\frac{H^2}{2}-H\sqrt U\le \mathcal L_H.
\]
Evenness of the risk now proves
$\sup_{|h|\le H}\mathcal L_h(a^{\tau,e})=\mathcal L_H(a^{\tau,e})$.
These bounds locate the worst parameter, without adding error to the endpoint improvement.

\begin{proof}[Completion of the proof of Theorem~\ref{thm:deficit}]
Lemma~\ref{crit:lem:potential} and \eqref{crit:bayes-direction} give the all-algorithm
upper bound.  The policy \eqref{crit:design}, its endpoint maximality, and
\eqref{crit:endpoint-gain} give the lower bound:
\begin{equation}
 \frac{35\sqrt2}{1536(1+\delta)^4}\delta^4
 \le\Delta(H)\le\frac{9}{8H^2(1-2\delta)}\delta^2,
 \qquad 0<\delta\le\frac1{16}.
 \label{crit:deficit-bounds}
\end{equation}
The upper bound alone holds for $0<\delta<1/2$.
These explicit auxiliary estimates support the comparison with
untruncated feedback. The matching unrestricted order is established
by Theorem~\ref{thm:critical-new}.
\end{proof}

\section{Exact Deficit Decomposition and Expected Energy Shape}
\label{crit:sec:identity}

\begin{proposition}
\label{crit:prop:identity}
For every $H>0$ and every admissible algorithm with $Q<\infty$, define
$\mathcal S_H(a)=H^2Q/\sqrt2+\mathbb E_\pi\int_0^1m_ta_t\,dt$.
Then $\mathcal S_H=S_1+S_2$, where
\begin{align}
 S_1&=\frac{H^2}{\sqrt2}\mathbb E_\pi\int_0^1
       \sqrt t\left(a_t+\frac{m_t}{H^2\sqrt{2t}}\right)^2\,dt,
       \label{crit:alignment-part}\\
 S_2&=\frac{H^2}{\sqrt2}\mathbb E_\pi\int_0^1
       (1-\sqrt t)\left(\frac{2m_t^2}{H^2}-\frac{m_t^4}{H^4}\right)
       a_t^2\,dt. \label{crit:nonlinear-part}
\end{align}
Both terms are finite and nonnegative, and
\begin{equation}
 b_H(a)=\frac{\delta Q}{2}-\mathcal S_H(a),\qquad
 S_2\ge\frac1{12\sqrt2H^4}\int_0^1\frac{\mathbb E_\pi m_t^4}{\sqrt t}\,dt.
 \label{crit:fourth-penalty}
\end{equation}
\end{proposition}
\begin{proof}
Boundedness of the posterior and finite expected energy give square integrability, so
$v(t)=\int_0^t\mathbb E_\pi(H^2-m_s^2)^2a_s^2\,ds\le H^4q(t)$.
Tonelli's theorem gives the finite identity
\begin{equation}
 \int_0^1\frac{v(t)}{2\sqrt t}\,dt
 =\int_0^1(1-\sqrt s)v'(s)\,ds\le H^4Q.
 \label{crit:tonelli-weight}
\end{equation}
All terms in the expansion of $S_1$ are integrable, using this bound,
$\mathbb E_\pi\int\sqrt t\,a_t^2\,dt\le Q$, and
$\mathbb E_\pi\int|m_ta_t|\,dt\le H\sqrt Q$.
The cross term is $\mathbb E_\pi\int m_ta_t\,dt$.  Replace the posterior-square term by
\eqref{crit:tonelli-weight}.  For $y=m_t^2/H^2$, the resulting coefficient
of $a_t^2$, after adding $S_2$, is
$\sqrt t+(1-\sqrt t)\{(1-y)^2+2y-y^2\}=1$.
This proves the identity and integrability at zero; values at $t=0$ are immaterial.
Furthermore,
$M_4'(t)=6\mathbb E_\pi m_t^2(H^2-m_t^2)^2a_t^2
\le6H^4\mathbb E_\pi m_t^2a_t^2$.
Since $2y-y^2\ge y$ on $[0,1]$, another application of Tonelli gives
\[
 S_2\ge\frac1{\sqrt2}\int_0^1(1-\sqrt t)\mathbb E_\pi m_t^2a_t^2\,dt
 \ge\frac1{6\sqrt2H^4}\int_0^1(1-\sqrt t)M_4'(t)\,dt
 =\frac1{12\sqrt2H^4}\int_0^1\frac{M_4(t)}{\sqrt t}\,dt.
\]
This lower bound replaces only $S_2$ and can be retained together with $S_1$.
\end{proof}

\begin{corollary}
\label{crit:cor:shape}
If $0<\delta<1/2$ and $b_H(a)>0$, put $p(t)=q(t)/Q$ and $r(t)=\sqrt{2p(t)p'(t)}$.  Then
\begin{equation}
 \|r-1\|_2^2<\frac{2\delta}{1+\delta},\qquad
 \sup_{0\le t\le1}|p(t)^2-t|
 <2\sqrt{\frac{2\delta}{1+\delta}},
 \label{crit:shape}
\end{equation}
and $Q$ satisfies \eqref{crit:positive-energy}.
Writing $g(t)=-\mathbb E_\pi m_ta_t$, one also has
\begin{equation}
 \int_0^1\{\sqrt{v(t)q'(t)}-g(t)\}\,dt<\frac{\delta Q}{2}.
 \label{crit:mean-alignment}
\end{equation}
\end{corollary}
\begin{proof}
Absolute continuity gives $\int_0^1r^2\,dt=1$.
The bound $v\le H^4q$ and Cauchy--Schwarz imply
$0<b_H\le-Q/2+(H^2Q/\sqrt2)\int_0^1r\,dt$.
Thus $\int r>1/(1+\delta)$, which proves the first inequality.
For the second, use $p(t)^2-t=\int_0^t(r^2-1)\,ds$ and $\|r+1\|_2\le2$.
The three nonnegative Cauchy losses are
\[
 \frac{H^2Q}{\sqrt2}-H^2\int_0^1\sqrt{qq'}\,dt,\qquad
 \int_0^1(H^2\sqrt{qq'}-\sqrt{vq'})\,dt,\qquad
 \int_0^1(\sqrt{vq'}-g)\,dt.
\]
Their sum is $\mathcal S_H=\delta Q/2-b_H<\delta Q/2$.
This proves \eqref{crit:mean-alignment}; wherever $vq'>0$ its integrand is
$\tfrac12\sqrt{v/q'}\mathbb E_\pi(a_t+\sqrt{q'/v}\,m_t)^2$.
If $vq'=0$, Cauchy--Schwarz gives $g=0$ and the original integrand is zero.
These bounds concern expected energy and mean alignment, not pathwise concentration.
Positive worst-case improvement implies $b_H>0$, so they also hold for every
policy improving the full interval.
\end{proof}

\section{Fourth-Order Sharpness Within the Original Feedback Family}
\label{crit:sec:restricted}

\subsection{The Original Diffusion and Its Fourth-Moment Change of Measure}
\label{crit:subsec:tilt}

In this section $H\in[H_c,1]$ and the policy is \eqref{crit:policy}.
In logarithmic time $s=\log(t/\tau)$, put $X_s=m_{\tau\exp(s)}/H$.
Conditional on $X_0=x\in(0,1)$, reversing the innovation's sign gives
\begin{equation}
 dX_s=\frac{X_s(1-X_s^2)}{\sqrt2}\,dW_s,\qquad
 Z_s=\ell(X_s),\quad
 \ell(x)=\log x-\tfrac12\log(1-x^2).
 \label{crit:log-diffusion}
\end{equation}
The time-changed Brownian increments are independent of the probe filtration.
Since $\ell'(x)=1/[x(1-x^2)]$, It\^o's formula yields
\begin{equation}
 dZ_s=\frac{dW_s}{\sqrt2}
       +\left(-\frac14+\frac34X_s^2\right)\,ds.
 \label{crit:lamperti}
\end{equation}
After localization, constant diffusion and bounded drift exclude reaching either
endpoint in finite time, since $\ell$ diverges there.  For any finite horizon $S$, define
\[
 M_S=\frac{X_S^4}{x^4}
       \exp\left\{-3\int_0^S(1-X_u^2)^2\,du\right\}.
\]
It\^o's formula identifies $M$ as the stochastic exponential with bounded
integrand $2\sqrt2(1-X_s^2)$, hence a true martingale on each finite horizon.
Under $d\mathbb P_x^{(4)}=M_S\,d\mathbb P_x$, Girsanov's theorem gives
\[
 \begin{aligned}
 dX_s&=\frac{X_s(1-X_s^2)}{\sqrt2}\,dW_s^{(4)}
          +2X_s(1-X_s^2)^2\,ds,\\
 dZ_s&=\frac{dW_s^{(4)}}{\sqrt2}
          +\left(\frac74-\frac54X_s^2\right)\,ds.
 \end{aligned}
\]
In particular, $Z_s\le Z_0+W_s^{(4)}/\sqrt2+7s/4$, and exactly
\begin{equation}
 \mathbb E_xX_S^4=x^4\exp(3S)\mathbb E_x^{(4)}
       \exp\left\{-3\int_0^S(2X_u^2-X_u^4)\,du\right\}.
 \label{crit:exact-tilt}
\end{equation}
This identity retains the original nonlinear diffusion at all amplitudes.

\begin{proposition}
\label{crit:prop:fourth-lower}
There exist $c,C>0$ and $L_0<\infty$, uniform for $H\in[H_c,1]$, such that
$L=\log(1/\tau)\ge L_0$, $1\le f\le\exp(3L)$, and $e=f\tau^{7/2}$ imply
\begin{equation}
 \begin{gathered}
 \inf_{t\in[1/4,1/2]}\mathbb E_\pi m_t^4
 \ge ce^2\tau^{-3}\Psi_L(f),\\
 \Psi_L(f)=\Phi\left(-\frac{\log f+\log L+C}
                            {\sqrt{2(L-\log4)}}\right),
 \end{gathered}
 \label{crit:fourth-lower}
\end{equation}
where $\Phi$ is the standard normal distribution function.  After increasing $L_0$,
\begin{equation}
 f\Psi_L(f)\ge cf^{1/8},\qquad 1\le f\le\exp(3L).
 \label{crit:tail-coercivity}
\end{equation}
\end{proposition}
\begin{proof}
At probe completion, $\Xi_\tau=he+\sqrt e\,N$, with $N$ standard normal under each
endpoint law.  Since $e\le\tau^{1/2}$, $E_0=\{1\le\Xi_\tau/\sqrt e\le2\}$
has probability at least a uniform $c_0>0$.  On this event, for large $L$,
\[
 c_1\sqrt e\le m_\tau\le c_2\sqrt e,\qquad
 Z_0\le\tfrac12\log e+C_0.
\]
These follow from $\tanh u/u\to1$ and compactness of the $H$ interval.
Fix $t\in[1/4,1/2]$, $S=L+\log t$, $b_L=-\tfrac12\log L$, and
$\bar z=\tfrac12\log e+C_0$.  Under each conditional tilted law, consider
\[
 A_L=\left\{\sup_{0\le s\le S}
       (\bar z+W_s^{(4)}/\sqrt2+7s/4)\le b_L\right\}.
\]
On $E_0$ this event implies $X_s^2\le1/L$ for $s\le S$,
so the integral in \eqref{crit:exact-tilt} is at most $6S/L\le6$.
Conditioning on the probe and then integrating gives
\begin{equation}
 \mathbb E_\pi m_t^4\ge ce^2(t/\tau)^3\mathbb P(A_L).
 \label{crit:barrier-moment}
\end{equation}
Here $\mathbb P(A_L)$ is the same Brownian barrier probability for all retained initial states.
Let $\mu=7/4$, $\sigma=1/\sqrt2$, and $d=b_L-\bar z=7L/4-(\log f+\log L)/2-C_0$.
Uniformly over the stated range, $d\ge L/8$ for sufficiently large $L$.
The reflection formula with drift is
\begin{equation}
 \begin{gathered}
 \mathbb P(A_L)=\Phi(-a)-\exp(2\mu d/\sigma^2)\Phi(-b),\\
 a=\frac{\mu S-d}{\sigma\sqrt S},\quad
 b=\frac{\mu S+d}{\sigma\sqrt S}.
 \end{gathered}
 \label{crit:reflection}
\end{equation}
Driftless reflection gives the killed terminal density $p_S(y)-p_S(2d-y)$
on $y<d$, where $p_S$ is the centered normal density with variance $\sigma^2S$.
Multiply by $\exp\{\mu y/\sigma^2-\mu^2S/(2\sigma^2)\}$ and integrate over $y<d$
to obtain \eqref{crit:reflection}.  We must control its subtraction.
Write $R(u)=\Phi(-u)/\phi(u)$, where $\phi$ is the normal density.
For $u>0$, integration by parts gives $u/(1+u^2)\le R(u)\le1/u$; also
$R(u)=\int_0^\infty e^{-uv-v^2/2}\,dv$ is decreasing.
For $a\ge0$, the reflected-to-first-term ratio is $R(b)/R(a)$, by density cancellation.
As $d\ge L/8$ and $\mu S\le7L/4$, one has $a/b\le13/15$ and $b\ge c\sqrt L$.
For $a\ge4$ the ratio is at most $(17/16)(13/15)<1$.
For $0\le a\le4$ it is at most $1/[bR(4)]$, tending uniformly to zero.
If $a<0$, decrease $d$ to $\mu S$; monotonicity and the same formula
give $\mathbb P(A_L)\ge1/2-O(S^{-1/2})\ge1/4$ for large $L$.
Thus in all cases $\mathbb P(A_L)\ge c\Phi(-a)$.
Finally, $a\le(\log f+\log L+C)/\sqrt{2(L-\log4)}$ after enlarging $C$.
Together with $t^3\ge1/64$, this proves \eqref{crit:fourth-lower}.
For \eqref{crit:tail-coercivity}, write $y=\log f$ and $k=\log L+C$.
For $0\le y\le\sqrt L$, the normal argument stays bounded, so
$f\Psi_L(f)\ge cf\ge cf^{1/8}$.
For $\sqrt L\le y\le3L$, the Mills lower bound gives
\[
 \log(f\Psi_L(f))\ge
 y-\frac{(y+k)^2}{4(L-\log4)}
 -\log\left(1+\frac{y+k}{\sqrt{2(L-\log4)}}\right)-C'.
\]
The quadratic subtraction divided by $y$ is at most
\[
 \frac{3L}{4(L-\log4)}
 +\frac{k}{2(L-\log4)}
 +\frac{k^2}{4(L-\log4)\sqrt L}=\frac34+o(1).
\]
The logarithmic subtraction divided by $y$ is bounded by
$L^{-1/2}\log(1+(3L+k)/\sqrt{2(L-\log4)})=o(1)$.
Uniformly for large $L$, these ratios sum to at most $7/8$, giving $y/8-C'$.
\end{proof}

\subsection{Optimization over Every Probe Energy}
\label{crit:subsec:family}

\begin{proof}[Proof of Theorem~\ref{thm:restricted}]
For $0<\delta<1$, write
\[
 \begin{gathered}
 C_{\mathrm{tanh}}(H)=
 \inf_{\substack{\delta/3\le\tau\le2\delta/3\\ e>0}}
       \sup_{|h|\le H}\mathcal L_h(a^{\tau,e}),\\
 \Delta_{\mathrm{tanh}}(H)=H^2/2-C_{\mathrm{tanh}}(H).
 \end{gathered}
\]
For members with $b_H(a)\le0$ the desired upper bound is immediate.
For $b_H(a)>0$, \eqref{crit:positive-energy} gives $e\le Q=O(\delta)$.
Hence $\bar Q=e/\sqrt\tau=O(\sqrt\delta)<1$, uniformly in the family.
Put $f=e\tau^{-7/2}=\bar Q/\tau^3$ and $L=\log(1/\tau)$.
Then $f\le\tau^{-3}$.  If $f<1$, nonnegativity of the exact losses gives
$b_H\le\delta\bar Q/2=\delta\tau^3f/2=O(\delta^4)$.
If $f\ge1$, integrate \eqref{crit:fourth-lower} over $[1/4,1/2]$
in \eqref{crit:fourth-penalty}, obtaining
\[
 b_H(a)\le\frac{\delta\tau^3f}{2}-c\tau^4f^2\Psi_L(f).
\]
Positivity forces $f\Psi_L(f)\le C\delta/\tau\le C'$.
By \eqref{crit:tail-coercivity}, $f$ is bounded independently of
$\delta,\tau,e$, and again $b_H(a)\le C\delta^4$.
Since $\sup_{|h|\le H}\mathcal L_h(a)\ge H^2/2-b_H(a)$, this proves
$\Delta_{\mathrm{tanh}}(H)\le C\delta^4$.
For the reverse bound, \eqref{crit:design} belongs to the stated family
and has full-interval improvement \eqref{crit:endpoint-gain}.
Thus $c\delta^4\le\Delta_{\mathrm{tanh}}(H)\le C\delta^4$ for small $\delta>0$.
The upper bound also works for any fixed $0<c_-\le\tau/\delta\le c_+<\infty$.
It excludes unbounded logarithmic improvement by retuning the probe energy
in this family, but imposes no such bound on other feedbacks.
\end{proof}

\section{Failure of Permanent-Stopping Truncation}
\label{crit:sec:stopping}

\begin{proposition}
\label{crit:prop:stopping}
Fix $c_\tau>0$, $\eta\in(0,1/2)$, $\kappa\ge0$, and $\beta>0$.
Set $\tau=c_\tau\delta$, $e/\sqrt\tau=\delta^\beta$, and $r_\delta=H\eta\delta^\kappa$.
Run the probe and then the original feedback until
$T=\inf\{t\ge\tau:|m_t|\ge r_\delta\}\wedge1$, and set $a_t=0$
permanently after $T$.  If the probe ends outside the threshold, set $T=\tau$.  Suppose
\begin{equation}
 \beta+\tfrac12>2\kappa,\qquad 0<\beta<3+2\kappa.
 \label{crit:stopping-range}
\end{equation}
For a constant $c>0$ depending only on the fixed parameters and all sufficiently small $\delta>0$,
\begin{equation}
 b_H(a)\le\frac12\delta^{\beta+1}
 -\frac{c}{\sqrt{\log(1/\delta)}}
       \delta^{\,2\kappa+(\beta+1-2\kappa)^2/4}<0.
 \label{crit:stopping-failure}
\end{equation}
Consequently $\sup_{|h|\le H}\mathcal L_h(a)>H^2/2$.
\end{proposition}
\begin{proof}
After stopping, the posterior remains constant.  Its second moment obeys
\[
 M_2'(t)=\frac{\mathbb E_\pi[
 \mathbf{1}_{\{t<T\}}m_t^2(1-m_t^2/H^2)^2]}{2t}
 \le\frac{M_2(t)}{2t}.
\]
The probe bound and Gronwall therefore imply
$\mathbb E_\pi m_t^2\le H^4e\sqrt{t/\tau}$ and $Q\le e/\sqrt\tau=\delta^\beta$.
On $\{T\le1/4\}$, the post-stopping contribution to \eqref{crit:alignment-part} is
\[
 S_1\ge\frac1{\sqrt2H^2}
 \mathbb E_\pi[m_T^2(1-\sqrt T)\mathbf{1}_{\{T\le1/4\}}]
 \ge\frac{r_\delta^2}{2\sqrt2H^2}\mathbb P_\pi(T\le1/4).
\]
This covers immediate stopping after the probe, when $|m_T|\ge r_\delta$.
Use the positive-probability probe event $E_0$ from Proposition~\ref{crit:prop:fourth-lower}.
Condition first on the probe filtration $\mathcal F_\tau$; subsequent innovation
Brownian increments are independent of this entire filtration.
There $X_0=m_\tau/H$ is between constant multiples of $\sqrt e$ and, by the
first condition in \eqref{crit:stopping-range}, below $\eta\delta^\kappa$ for small $\delta$.
Until the threshold hit, \eqref{crit:lamperti} gives $Z_s\ge Z_0+W_s/\sqrt2-s/4$.
Let $D=\log(1/\delta)$ and $S=\log((1/4)/\tau)=D+O(1)$.
If this Brownian lower bound at $S$ exceeds $\ell(\eta\delta^\kappa)$,
the original process must have hit the threshold by time $1/4$.
This uses the equation only on paths not yet stopped; Brownian motion can be continued after $T$.
Uniformly on $E_0$, since $e=\sqrt{c_\tau}\delta^{\beta+1/2}$,
\[
 \ell(\eta\delta^\kappa)-Z_0+S/4
 \le\frac{\beta+1-2\kappa}{2}D+C.
\]
Here $\ell(\eta\delta^\kappa)=\kappa\log\delta+O(1)$,
including $\kappa=0$, and $\beta+1-2\kappa>1/2$.
The normal lower tail bound with variance $S/2$ gives
\[
 \mathbb P_\pi(T\le1/4)\ge
 cD^{-1/2}\delta^{(\beta+1-2\kappa)^2/4}.
\]
Squaring the normal argument contributes $(\beta+1-2\kappa)^2D/4+O(1)$ to its
density exponent; the Mills factor is bounded below by a constant times $D^{-1/2}$.
Combine this estimate with $b_H=\delta Q/2-S_1-S_2$ and $S_2\ge0$
to obtain the first inequality in \eqref{crit:stopping-failure}.
The loss exponent minus the gain exponent is
\[
 2\kappa+\frac{(\beta+1-2\kappa)^2}{4}-(\beta+1)
 =\frac{(\beta-1-2\kappa)^2-4}{4}<0.
\]
The two strict conditions in \eqref{crit:stopping-range} imply
$-3/2<\beta-1-2\kappa<2$, proving this strict inequality.
Thus the negative term eventually dominates despite its logarithmic denominator.
The worst-case risk is at least the endpoint average $H^2/2-b_H(a)>H^2/2$.
\end{proof}

This proposition concerns permanent cessation, not soft truncation, switching to
a nonzero feedback, or restarting exploration.  At $\beta=3+2\kappa$ the estimate
does not decide the sign of the improvement.

\section{A Uniform Deficit Law and the Critical Upper Bound}
\label{sharp:sec:proof}

We prove Theorem~\ref{thm:energy-new} and the upper bound in
Theorem~\ref{thm:critical-new}. All expectations in this appendix are
under the equal-endpoint mixture, with its complete record filtration.
The proof uses a test function evaluated on each original algorithm;
it requires neither a Markov reduction nor existence of an optimal rule.

\subsection{Normalisation and finite-energy identities}
\label{sharp:subsec:normalise}

For an original rule of finite endpoint risk, put
\[
 X_t=m_t/H,\qquad u_t=Ha_t,\qquad
 q(t)=\mathbb E_\pi\int_0^t u_r^2\,dr,\qquad q=q(1)=H^2Q.
\]
Thus $q$ denotes normalised expected energy, whereas $Q$ keeps its
meaning from the main text. The complete-record posterior equation
\eqref{crit:posterior} gives
\begin{equation}
 dX_t=(1-X_t^2)u_t\,dI_t,\qquad X_0=0,\qquad |X_t|\le1.
 \label{sharp:diffusion}
\end{equation}
The innovation $I$ is Brownian motion relative to the full mixture
filtration. The argument below in fact applies to any predictable $u$
and adapted solution of \eqref{sharp:diffusion} in a Brownian filtration,
provided $q<\infty$. Randomisation is included in that filtration.

In these variables Proposition~\ref{crit:prop:identity} becomes
\begin{align}
 S_1&=\frac1{\sqrt2}\mathbb E_\pi\int_0^1\sqrt t
       \left(u_t+\frac{X_t}{\sqrt{2t}}\right)^2dt,
       \label{sharp:S1}\\
 S_2&=\frac1{\sqrt2}\mathbb E_\pi\int_0^1(1-\sqrt t)
                (2X_t^2-X_t^4)u_t^2dt,\qquad
 S=S_1+S_2=\mathcal S_H(a).
       \label{sharp:S2}
\end{align}
Write $v(t)=\mathbb E_\pi X_t^2$ and $M(t)=\mathbb E_\pi X_t^4$.
It\^o's formula gives, for almost every $t$,
\begin{equation}
 v'(t)=\mathbb E_\pi(1-X_t^2)^2u_t^2,\qquad
 M'(t)=6\mathbb E_\pi X_t^2(1-X_t^2)^2u_t^2.
 \label{sharp:moments}
\end{equation}
The functions $v,M$ are absolutely continuous and nondecreasing,
and $v(t)\le q(t)$. Indeed the stochastic integrands for $X^2,X^4$
are bounded multiples of $u$, so their integrals are square integrable.
No fourth action moment is needed. Tonelli's theorem gives
\[
 \int_0^1\frac{v(t)}{2\sqrt t}\,dt
 =\int_0^1(1-\sqrt t)v'(t)\,dt\le q.
\]
Also $\mathbb E_\pi\int_0^1|X_tu_t|dt\le\sqrt q$. Hence all terms in
the expansion of \eqref{sharp:S1} are integrable at time zero.
Expanding and using $(1-X^2)^2+2X^2-X^4=1$ proves
\begin{equation}
 S=\frac q{\sqrt2}+\mathbb E_\pi\int_0^1X_tu_t\,dt.
 \label{sharp:identity}
\end{equation}
For the original experiment, the corresponding improvement is
\begin{equation}
 b_H(a)=\frac{\delta}{2H^2}q-S,\qquad \delta=\sqrt2H^2-1.
 \label{sharp:improvement}
\end{equation}

\begin{lemma}[Expected-energy shape and fourth moment]
\label{sharp:lem:shape}
For $q>0$, one has
\begin{align}
 |q(t)-q\sqrt t|&\le2\sqrt{Sq},\qquad 0\le t\le1,
       \label{sharp:shape}\\
 0\le q(t)-v(t)&\le\frac{\sqrt2}{1-\sqrt t}S_2,\qquad t<1.
       \label{sharp:gap}
\end{align}
With $T=1/2$ and $K_0=12(1+\sqrt2)$,
\begin{equation}
 M(T)\le K_0S_2.
 \label{sharp:fourth}
\end{equation}
In particular, $q>0$ implies $S>0$.
\end{lemma}
\begin{proof}
Let $r(t)=\sqrt{2q(t)q'(t)}/q$, defined almost everywhere.
The chain rule gives $\int_0^1r^2dt=1$ and
$q(t)/q=\|r\|_{L^2(0,t)}$, including flat intervals of $q(t)$.
By \eqref{sharp:identity}, $v(t)\le q(t)$, and Cauchy--Schwarz,
\[
 S\ge\frac q{\sqrt2}-\int_0^1\sqrt{v(t)q'(t)}\,dt
 \ge\frac q{\sqrt2}\left(1-\int_0^1r(t)\,dt\right).
\]
Therefore $\|r-1\|_2^2\le2\sqrt2 S/q$; the reverse triangle
inequality yields \eqref{sharp:shape}. The exact identity
\[
 q(t)-v(t)=\mathbb E_\pi\int_0^t(2X_r^2-X_r^4)u_r^2\,dr
\]
gives \eqref{sharp:gap}. Since
$M'\le6\mathbb E_\pi X_t^2u_t^2
\le6\mathbb E_\pi(2X_t^2-X_t^4)u_t^2$, Tonelli gives
\[
 S_2\ge\frac1{6\sqrt2}\int_0^1(1-\sqrt t)M'(t)\,dt
 =\frac1{12\sqrt2}\int_0^1\frac{M(t)}{\sqrt t}\,dt.
\]
Monotonicity of $M$ now proves \eqref{sharp:fourth}, with the displayed
$K_0$. If $S=0$ and $q>0$, \eqref{sharp:shape}--\eqref{sharp:gap}
give $v(T)=q\sqrt T>0$, whereas \eqref{sharp:fourth} gives $M(T)=0$.
This contradicts $v(T)^2\le M(T)$.
\end{proof}

\subsection{A profile with compensated threshold drift}
\label{sharp:subsec:profile}

Fix $d_*=1/100$ and the continuous positive function
\begin{equation}
 \rho(z)=\frac{4d_*}{5}
 \begin{cases}
  e^{2z}(1+z^2),&z\le0,\\
  e^{-2z},&z\ge0,
 \end{cases}
 \qquad \int_{\mathbb R}\rho(z)\,dz=d_*.
 \label{sharp:rho}
\end{equation}
For any $R_0\ge1$, set $\nu=d_*/\sqrt{R_0}$ and define
\begin{equation}
 w(z)=\int_z^\infty
       \frac{e^{2\nu(y-z)}-1}{2\nu}\rho(y)\,dy,\qquad
 h(R,z)=\frac{w(z)}{\sqrt{R+w(z)^2}},\quad 1\le R\le R_0.
 \label{sharp:profile}
\end{equation}
The right tail is integrable since $\nu<1$. Differentiation on compact
sets gives $w>0$, $w\in C^2(\mathbb R)$, and
\begin{equation}
 w''+2\nu w'=\rho,\qquad
 -w'=2\nu w+\int_z^\infty\rho(y)\,dy\le2\nu w+d_*.
 \label{sharp:ode}
\end{equation}
For each fixed $R_0$, the parameter $\nu$ is held fixed when taking
partial derivatives in $R$.

\begin{lemma}[Uniform profile estimates]
\label{sharp:lem:profile}
Put $\mathcal D=(R+w^2)^{1/2}$,
$A=-h_R+\nu h_z/2+h_{zz}/4$, and $B=12h+7h_z+h_{zz}$.
Uniformly over $R_0\ge1$, $1\le R\le R_0$, and $z\in\mathbb R$,
\begin{align}
 A&\ge\frac{w+R\rho}{4\mathcal D^3},
       \label{sharp:residual}\\
 e^{2z}\sqrt R\,B^2&\le C_1\frac{w+R\rho}{\mathcal D^3},
       \label{sharp:square}\\
 e^{2z}\sqrt R\,|B|&\le C_2,
       \label{sharp:curvature}
\end{align}
where one may take $C_1=1000$ and $C_2=2$. There is also a fixed
$c_h>0$ such that for every $r>0$,
\begin{equation}
 h(R,\tfrac12\log r)\ge\frac{c_h}{\sqrt R(1+r)}.
 \label{sharp:kernel-lower}
\end{equation}
\end{lemma}
\begin{proof}
Direct differentiation gives
\[
 h_R=-\frac{w}{2\mathcal D^3},\quad
 h_z=\frac{Rw'}{\mathcal D^3},\quad
 h_{zz}=\frac{Rw''}{\mathcal D^3}
              -\frac{3Rw(w')^2}{\mathcal D^5}.
\]
Using \eqref{sharp:ode},
\[
 A=\frac{w}{2\mathcal D^3}+\frac{R\rho}{4\mathcal D^3}
          -\frac{3Rw(w')^2}{4\mathcal D^5}.
\]
Since $\nu^2R\le d_*^2$,
\[
 \frac{R(w')^2}{R+w^2}\le2d_*^2+8\nu^2R\le10d_*^2<1/3,
\]
which proves \eqref{sharp:residual}.

For $z\ge0$, integration is explicit:
\begin{equation}
 w=b_\nu e^{-2z},\quad b_\nu=\frac{d_*}{5(1-\nu)}
       \in[d_*/5,d_*/4],\quad
 w'=-2w,\quad w''=4w,\quad \rho=4(1-\nu)w.
 \label{sharp:right-tail}
\end{equation}
Thus $|B|\le42w/\sqrt R$ and
$(w+R\rho)/\mathcal D^3\ge3w/\sqrt R$. The latter follows from
\[
 \frac{(w+R\rho)/\mathcal D^3}{w/\sqrt R}
 =\frac{R^{-1}+4(1-\nu)}{(1+w^2/R)^{3/2}}\ge3,
\]
using $R\ge1$, $\nu\le0.01$, and $w\le0.0025$.
The ratios needed for \eqref{sharp:square} and
\eqref{sharp:curvature} are consequently at most $588b_\nu\le1.47$
and $42b_\nu\le0.105$.

For $z=-r\le0$, $w\ge d_*/5$ and $\rho\le4d_*/5$.
Equation~\eqref{sharp:ode} implies
\begin{equation}
 |w'|\le6w,\qquad 0<w''\le5w,\qquad
 w(-r)\le d_*(1+r)e^{2\nu r}.
 \label{sharp:left-tail}
\end{equation}
For the last estimate, integrate
$\frac{d}{dr}w(-r)\le2\nu w(-r)+d_*$ and use
$w(0)=b_\nu\le d_*$. The derivative formulas above give
$|B|\le(12+42+5+108)w/\mathcal D=167w/\mathcal D$.
It follows that
\[
 e^{-2r}\sqrt R\,|B|\le167d_*(1+r)e^{-(2-2\nu)r}
 \le167d_*<2.
\]
For the squared estimate, its left side divided by
$(w+R\rho)/\mathcal D^3$ is at most
\begin{equation}
 167^2e^{-2r}\frac{Rw^2+\sqrt R\,w^3}{w+R\rho},
 \label{sharp:tail-ratio}
\end{equation}
where we used $\mathcal D\le\sqrt R+w$.
If $0\le r\le\sqrt R$, then $\nu r\le d_*$ and
$w\le d_*(1+r)e^{2d_*}\le2d_*e^{2d_*}\sqrt R$. Hence
\[
 e^{-2r}\frac{Rw^2}{w+R\rho}
 \le\frac{w^2}{(4d_*/5)(1+r^2)}
 \le(5d_*/2)e^{4d_*}.
\]
The other summand is this expression multiplied by $w/\sqrt R$.
Thus \eqref{sharp:tail-ratio} is at most
$167^2(5d_*/2)e^{4d_*}(1+2d_*e^{2d_*})<748$.
If $r\ge\sqrt R\ge1$, the denominator is at least $w$ and $R\le r^2$.
The two summands before the factor $167^2$ are bounded by
\[
 d_*r^2(1+r)e^{-(2-2d_*)r},\qquad
 d_*^2r(1+r)^2e^{-(2-4d_*)r}.
\]
Their sum is at most $(2d_*+4d_*^2)r^3e^{-3r/2}$.
Since $\sup_{r\ge0}r^3e^{-3r/2}=(2/e)^3<1$,
\eqref{sharp:tail-ratio} is smaller than $569$ in this region.
This proves \eqref{sharp:square} and \eqref{sharp:curvature} on the
whole line, uniformly even when $R_0\to\infty$.

Finally, if $0<r\le1$, monotonicity of $w$ gives
$w(\frac12\log r)\ge d_*/5$. The function
$w/\sqrt{R+w^2}$ increases with $w$, so it is at least
$d_*/(5\sqrt{2R})$ there. If $r\ge1$,
\eqref{sharp:right-tail} gives $w=b_\nu/r<1$ and
$h\ge d_*/(5r\sqrt{2R})$. Thus \eqref{sharp:kernel-lower}
holds with $c_h=d_*/(5\sqrt2)$.
\end{proof}

\subsection{A test function valid for every action}
\label{sharp:subsec:certificate}

Fix $0<\tau\le1/4$ and $k>0$. Recall $T=1/2$ and define
\begin{align}
 R_0&=1+\log(T/\tau),\quad R(t)=1+\log(T/t),\quad
 \nu=d_*/\sqrt{R_0},\quad k_t=k(t/\tau)^{-\nu},\nonumber\\
 z(t,x)&=\log\frac{|x|}{\sqrt{k_t}\,t^{7/4}},\quad
 K(t)=\frac{k_t}{\sqrt{R(t)}},\quad
 \Phi(t,x)=\varepsilon\frac{x^4}{t^3}h(R(t),z(t,x)),
 \label{sharp:test}
\end{align}
where $\Phi(t,0)=0$ and $\varepsilon=10^{-3}$ is fixed.
For $x\ne0$, differentiation yields
\begin{equation}
 \Phi_{xx}=\varepsilon x^2t^{-3}B,\qquad
 \Phi_t+\frac{x^2}{4t}\Phi_{xx}
       =\varepsilon x^4t^{-4}A.
 \label{sharp:derivatives}
\end{equation}
In particular, $tR'=-1$ and $tz_t=-7/4+\nu/2$; the
$\nu h_z/2$ term is included in $A$.
The derivative bounds in Lemma~\ref{sharp:lem:profile} show that
$h_z,h_{zz},h_R$ are bounded for $R\ge1$.
For fixed $\tau,k$, as $x\to0$ the derivatives
$\Phi_x,\Phi_{xx},\Phi_t$ are respectively $O(x^3),O(x^2),O(x^4)$,
uniformly in $t\in[\tau,T]$. They extend continuously by zero.
Thus $\Phi\in C^{1,2}([\tau,T]\times\mathbb R)$, including the
point where the two formulas for $\rho$ meet.

Because $x^2=k_tt^{7/2}e^{2z}$, \eqref{sharp:curvature} gives
$|\Phi_{xx}|\le K(t)\sqrt t$. For every real action $u$, write
$u_0=-x/\sqrt{2t}$ and $e_u=u-u_0$. Completing the square gives
\begin{equation}
\begin{aligned}
 &\Phi_t+\tfrac12u^2\Phi_{xx}
       +K\sqrt t\left(u+\frac{x}{\sqrt{2t}}\right)^2\\
 &\quad=\varepsilon x^4t^{-4}A+u_0\Phi_{xx}e_u
                    +(K\sqrt t+\Phi_{xx}/2)e_u^2\\
 &\quad\ge\varepsilon x^4t^{-4}A
                    -\frac{x^2\Phi_{xx}^2}{4Kt^{3/2}}
 \ge0.
\end{aligned}
 \label{sharp:all-actions}
\end{equation}
Indeed the quadratic coefficient is at least $K\sqrt t/2$.
The negative term divided by $\varepsilon x^4t^{-4}$ is
$\varepsilon e^{2z}\sqrt R B^2/4\le\varepsilon C_1A\le A$,
by \eqref{sharp:residual}--\eqref{sharp:square}.
At $x=0$ the left side equals $K\sqrt t\,u^2\ge0$.

For the actual generator in \eqref{sharp:diffusion}, the difference
from $\Phi_t+\frac12u^2\Phi_{xx}$ is
\[
 -\tfrac12(2x^2-x^4)u^2\Phi_{xx}
 \ge-\tfrac12K\sqrt t(2x^2-x^4)u^2,\qquad |x|\le1.
\]
For $t\le T$, the two negative terms are paid by
\eqref{sharp:S1}--\eqref{sharp:S2}, since
$\sqrt2\le2$ and $\sqrt t/[\sqrt2(1-\sqrt t)]<2$.
Consequently It\^o's formula gives
\begin{equation}
 \mathbb E_\pi\Phi(\tau,X_\tau)
 \le\mathbb E_\pi\Phi(T,X_T)
              +2\sup_{\tau\le t\le T}K(t)\,S.
 \label{sharp:ito}
\end{equation}
To justify this for the full class, fix $\tau,k$ first.
The required derivatives are bounded on $[\tau,T]\times[-1,1]$,
and the stochastic integral has integrand
$\Phi_x(t,X_t)(1-X_t^2)u_t$. Its second moment is bounded by
a fixed multiple of $q<\infty$; all drift terms are integrable.
Thus expectation can be taken directly, or after energy localisation
and removal of that localisation with the same integrable bounds.
The argument does not assume a uniform amplitude bound, a fourth
action moment, or Markov feedback. If the action stops, the posterior
remains in \eqref{sharp:S1}; that loss is not deleted.

The weight in \eqref{sharp:ito} satisfies a uniform estimate.
Writing $a=\log(t/\tau)$,
$\log(K/k)=-\nu a-\frac12\log(R_0-a)$ is convex, so its maximum
is at an endpoint. With $y=\sqrt{R_0}\ge1$,
\[
 \sqrt{R_0}e^{-\nu(R_0-1)}
 =ye^{-d_*(y-y^{-1})}
 \le e^{d_*}ye^{-d_*y}\le\frac{e^{d_*}}{ed_*}<101.
\]
Hence, with the absolute constant $C_4=101$,
\begin{equation}
 \sup_{\tau\le t\le T}K(t)\le\frac{C_4k}{\sqrt{R_0}}.
 \label{sharp:weight}
\end{equation}

\subsection{Initial mass and absorption}
\label{sharp:subsec:absorption}

\begin{lemma}[Relative-deficit bound]
\label{sharp:lem:relative}
There are absolute constants $C,s_0>0$ such that any process
\eqref{sharp:diffusion} with finite $q>0$ and $s=S/q\le s_0$ satisfies
\begin{equation}
 q\le Cs^3\sqrt{\log(e/s)}.
 \label{sharp:relative}
\end{equation}
\end{lemma}
\begin{proof}
By Lemma~\ref{sharp:lem:shape}, $s>0$. Consider
$\tau=Ds\le1/4$, where $D\ge128$ will be fixed after the
test-function constants. Equations~\eqref{sharp:shape}--\eqref{sharp:gap}
imply
\[
 v(\tau)\ge q\sqrt\tau(1-2/\sqrt D-2/D)\ge q\sqrt\tau/2.
\]
Here the gap term divided by $q\sqrt\tau$ is at most
$\sqrt2\sqrt\tau/[D(1-\sqrt\tau)]\le2/D$.
Set $k=q/\tau^3$ and $r=X_\tau^2/(q\sqrt\tau)$. Then
$\mathbb E_\pi r\ge1/2$ and $z(\tau,X_\tau)=\frac12\log r$
away from $X_\tau=0$. By \eqref{sharp:kernel-lower},
\[
 \mathbb E_\pi\Phi(\tau,X_\tau)
 \ge\frac{\varepsilon c_hq^2}{\tau^2\sqrt{R_0}}
           \mathbb E_\pi\frac{r^2}{1+r}
 \ge\frac{\varepsilon c_hq^2}{6\tau^2\sqrt{R_0}}.
\]
The second step is Jensen's inequality for the increasing convex
function $r^2/(1+r)$; it is bounded by $r$, and its value at $1/2$
is $1/6$. The formulas extend by continuity at $r=0$.
At the other endpoint, $0<h\le1$ and \eqref{sharp:fourth} give
\[
 \mathbb E_\pi\Phi(T,X_T)\le\varepsilon T^{-3}M(T)
                  \le\varepsilon T^{-3}K_0S.
\]
Combining with \eqref{sharp:ito} and \eqref{sharp:weight},
\begin{equation}
 \frac{q^2}{\tau^2}
 \le A_1\sqrt{R_0}S+A_2\frac q{\tau^3}S,
 \quad
 A_1=\frac{6T^{-3}K_0}{c_h},\qquad
 A_2=\frac{12C_4}{\varepsilon c_h}.
 \label{sharp:absorb}
\end{equation}
All these constants are independent of $\tau,k,q,s$ and the rule.
Now fix $D=\max\{128,2A_2\}$ and $s_0=1/(4D)$.
For $0<s\le s_0$, the second term in \eqref{sharp:absorb}
divided by its left side is $A_2s/\tau=A_2/D\le1/2$.
Absorption therefore gives
\[
 q\le2A_1D^2s^3\sqrt{1+\log(T/(Ds))}
   \le2A_1D^2s^3\sqrt{\log(e/s)}.
\]
This proves \eqref{sharp:relative}, without assuming $q\le1$.
The chosen $\tau,k$ are deterministic parameters in the analysis
of a fixed rule, not actions or information supplied to the controller.
\end{proof}

\subsection{Uniform energy law and matching minimax order}
\label{sharp:subsec:conclusion}

\begin{proof}[Proof of Theorem~\ref{thm:energy-new}]
Enlarge the constant in \eqref{sharp:relative} to $C\ge1$ and
take $s_0\le1$. First let $0<q\le1$.
If $s\ge s_0$, then
$S=sq\ge s_0q^{4/3}/[\log(e/q)]^{1/6}$.
If $s<s_0$, use $s\log(e/s)\le1$ in \eqref{sharp:relative} to get
\[
 q\le Cs^{5/2},\qquad
 \log(e/s)\le1+\tfrac25\log C+\tfrac25\log(1/q)
              \le C_{\log}\log(e/q),
 \quad C_{\log}=1+\tfrac25\log C.
\]
Substituting this upper bound for the logarithm back into
\eqref{sharp:relative} and taking cube roots gives
\begin{equation}
 S\ge c_q\frac{q^{4/3}}{[\log(e/q)]^{1/6}},
 \qquad c_q=\min\{s_0,(C\sqrt{C_{\log}})^{-1/3}\}>0.
 \label{sharp:inversion}
\end{equation}
For $H\in[H_c,1]$ and $0<Q\le1$,
\[
 Q/\sqrt2\le q=H^2Q\le Q\le1,\qquad
 \log(e/q)\le(1+\log\sqrt2)\log(e/Q).
\]
Since $S=\mathcal S_H(a)$, \eqref{sharp:inversion} proves
\eqref{eq:energy-new}, uniformly over the full stated radius interval.

For the second assertion, if $b_H(a)>0$ then \eqref{sharp:improvement}
gives $q>0$ and $0<s<\delta/(2H^2)\le\delta$.
Choose $\delta_1\le\min\{s_0,1/16\}$, so $H\in[H_c,1]$.
The function $s^3\sqrt{\log(e/s)}$ is increasing on $(0,1]$, with
derivative
$s^2\sqrt{\log(e/s)}[3-1/(2\log(e/s))]>0$.
Thus \eqref{sharp:relative} implies
$q\le C\delta^3\sqrt{\log(e/\delta)}$.
Dividing by $H^2\ge1/\sqrt2$ proves
\eqref{eq:positive-energy-sharp}. This argument does not use
$Q\le1$ or assume the desired critical bound in advance.
\end{proof}

\begin{proof}[Upper bound in Theorem~\ref{thm:critical-new}]
For any original rule with finite endpoint risk and $b_H(a)>0$,
the preceding relative-deficit argument gives
\[
 b_H(a)\le\frac{\delta}{2H^2}q
       \le C\delta^4\sqrt{\log(e/\delta)}.
\]
Rules with nonpositive improvement satisfy the same upper bound;
infinite endpoint risk cannot improve the zero-action baseline.
For each original rule, worst-case risk is at least its endpoint
average. Taking infima and subtracting from $H^2/2$ yields only
the needed direction
\[
 \Delta(H)\le\sup_a b_H(a)
       \le C\delta^4\sqrt{\log(e/\delta)}.
\]
This covers the complete admissible class, without assuming that
the endpoint prior is least favorable. Theorem~\ref{thm:moving},
proved in Appendix~\ref{new:sec:moving}, supplies the lower bound
over the entire parameter interval for $0<\delta\le1/16$.
Choose $\delta_0\le\min\{\delta_1,1/16\}$ and use
$\log(e/\delta)\asymp\log(1/\delta)$ on this interval.
This completes the matching critical law.
\end{proof}

\section{Moving Soft-Threshold Feedback}
\label{new:sec:moving}

This section proves Theorem~\ref{thm:moving}, the constructive lower bound
in Theorem~\ref{thm:critical-new}, and Corollary~\ref{cor:no-four-thirds}.
We first establish a loss-clock estimate and joint energy and loss bounds
for the entire parameter range used below.  The final step separately
locates the maximum risk over the full parameter interval.

\subsection{The Policy, Admissibility, and Probe Costs}
\label{move:subsec:setup}

Consider the scalar experiment
\[
 dY_t=ha_t\,dt+dB_t,\qquad |h|\le H,\qquad
 \mathcal L_h(a)=\frac12\mathbb E_h\int_0^1(a_t+h)^2\,dt.
\]
Write $H_c=2^{-1/4}$, $\delta=\sqrt2H^2-1$, and
$\pi_H=(\delta_{-H}+\delta_H)/2$.  The minimax value and its improvement
over zero action are
\[
 C_1^1(H)=\inf_a\sup_{|h|\le H}\mathcal L_h(a),\qquad
 \Delta(H)=H^2/2-C_1^1(H),
\]
where the infimum is over the original admissible class.
Unless a parameter or a conditional
initial state is specified, expectations are under the endpoint mixture
$\mathbb P_\pi=(\mathbb P_H+\mathbb P_{-H})/2$.
Throughout the joint estimates assume
\begin{equation}
 \begin{gathered}
 H\in[H_c,1],\qquad 0<\tau\le\frac12,\qquad A\ge1,\qquad f>0,\\
 e=f\tau^{7/2},\qquad k=Af,\qquad L=\log(1/\tau).
 \end{gathered}
 \label{move:parameters}
\end{equation}
Let $\rho$ be an independent uniform sign, included in the initial
observation filtration.  The general form of \eqref{eq:moving-policy} is
\begin{equation}
 \begin{gathered}
 \Xi_t=\int_0^t a_u\,dY_u,\qquad
 m_t=H\tanh(H\Xi_t),\qquad X_t=m_t/H,\\
 g_t(x)=\frac1{1+x^2/(kt^{7/2})},\\
 a_t=
 \begin{cases}
  \rho\sqrt{e/\tau},&0\le t\le\tau,\\[1mm]
  -\dfrac{m_t}{H^2\sqrt{2t}}\,g_t(X_t),&\tau<t\le1.
 \end{cases}
 \end{gathered}
 \label{move:policy}
\end{equation}
The design depends on $H$, not on the unknown $h$.  In particular,
$m_t$ denotes the same observable functional for every $h$; its
posterior interpretation is used only under $\mathbb P_\pi$.
The soft threshold for $|X_t|$ is $\sqrt{k}\,t^{7/4}$.
The feedback remains nonzero whenever $m_t\ne0$, including beyond that
threshold.

For fixed design parameters, write the post-probe feedback as
$a_t=b(t,\Xi_t)$.  On $t\in[\tau,1]$, both $b$ and its derivative in
the score are bounded.  Indeed, for $\lambda\ge0$,
\[
 \left|\frac{d}{dx}\frac{x}{1+\lambda x^2}\right|
 =\frac{|1-\lambda x^2|}{(1+\lambda x^2)^2}\le1.
\]
The score equation
$d\Xi_t=h b(t,\Xi_t)^2\,dt+b(t,\Xi_t)\,dB_t$
therefore has a unique nonexplosive strong solution after the explicit
probe.  The feedback is predictable, with the displayed convention at
$t=\tau$.  Stopping at any fixed energy cap and setting the action
permanently to zero preserves uniqueness: the solution agrees with the
original solution up to the cap and its score is constant afterwards.
Thus this is a rule in the original admissible class, including the
common-record and energy-stopped well-posedness requirements.
Its path energy and expected energy satisfy
\begin{equation}
 \mathcal Q:=\int_0^1a_t^2\,dt
 \le\Lambda:=e+\frac{L}{2H^2},\qquad
 Q:=\mathbb E_\pi\mathcal Q,\qquad
 \overline Q:=\frac e{\sqrt\tau}=f\tau^3.
 \label{move:cap}
\end{equation}
The deterministic cap is valid under every parameter law, since
$|X_t|g_t(X_t)\le1$.
Changing $(h,\rho,\Xi,a)$ to $(-h,-\rho,-\Xi,-a)$ preserves the
observation path and its drift $ha$, the feedback rule, and the loss.
Uniqueness and symmetry of the seed imply, for this same rule,
\begin{equation}
 \mathcal L_h(a)=\mathcal L_{-h}(a).
 \label{move:evenness}
\end{equation}

The full-record posterior identity \eqref{crit:posterior} applies and gives
\begin{equation}
 dX_t=H(1-X_t^2)a_t\,dI_t,\qquad X_0=0,
 \label{move:posterior}
\end{equation}
where $I$ is Brownian motion relative to the full mixture filtration,
including the seed.  By Proposition~\ref{crit:prop:identity}, put
\begin{equation}
 \begin{aligned}
 b_H(a)&=\frac{H^2}{2}
       -\frac{\mathcal L_H(a)+\mathcal L_{-H}(a)}2
       =\frac{\delta Q}{2}-S_1-S_2,\\
 S_1&=\frac{H^2}{\sqrt2}\mathbb E_\pi\int_0^1
       \sqrt t\left(a_t+\frac{m_t}{H^2\sqrt{2t}}\right)^2\,dt,\\
 S_2&=\frac{H^2}{\sqrt2}\mathbb E_\pi\int_0^1
       (1-\sqrt t)(2X_t^2-X_t^4)a_t^2\,dt,\qquad
 \mathcal S_H(a)=S_1+S_2.
 \end{aligned}
 \label{move:loss-identity}
\end{equation}
These are exactly the functionals in
\eqref{crit:alignment-part}--\eqref{crit:nonlinear-part}, not losses
of a modified experiment.  The cap ensures their finiteness.
On the feedback interval their contributions are
\begin{align}
 S_{1,\mathrm{fb}}
 &=\frac1{2\sqrt2}\int_\tau^1
       \frac{\mathbb E_\pi[X_t^2(1-g_t(X_t))^2]}{\sqrt t}\,dt,
       \label{move:alignment-feedback}\\
 S_{2,\mathrm{fb}}
 &=\frac1{2\sqrt2}\int_\tau^1
       \frac{1-\sqrt t}{t}
       \mathbb E_\pi[(2X_t^4-X_t^6)g_t(X_t)^2]\,dt.
       \label{move:nonlinear-feedback}
\end{align}

For clarity, the probe moment bounds \eqref{crit:probe-moments}
must be divided by $H^2$ and $H^4$ for the second and fourth moments
of $X=m/H$, respectively.  They give
\begin{equation}
 \begin{gathered}
 H^2e-H^4e^2\le\mathbb E_\pi X_\tau^2\le H^2e,\qquad
 \mathbb E_\pi X_\tau^4\le3H^4e^2,\\
 \mathbb E_\pi X_t^2\le H^2et/\tau,\qquad 0\le t\le\tau.
 \end{gathered}
 \label{move:probe-moments}
\end{equation}
These estimates also follow directly from \eqref{move:posterior}.
With $v(t)=\mathbb E_\pi X_t^2$ and $M(t)=\mathbb E_\pi X_t^4$,
the probe equations are
\[
 v'(t)=\frac{H^2e}{\tau}\mathbb E_\pi(1-X_t^2)^2
       \le\frac{H^2e}{\tau},\qquad
 M'(t)\le\frac{6H^2e}{\tau}v(t).
\]
Moreover,
$H^2e-v(\tau)=(H^2e/\tau)\int_0^\tau(2v(t)-M(t))\,dt
\le H^4e^2$, proving all of \eqref{move:probe-moments}.
Since $\rho$ is in the initial filtration,
$\mathbb E_\pi(\rho m_t)=\mathbb E_\pi(\rho h)=0$.
Thus the probe cross term in $S_1$ vanishes exactly, and
\begin{equation}
 \begin{aligned}
 S_{1,\mathrm{pr}}
 &=\frac{H^2e}{\sqrt2\tau}\int_0^\tau\sqrt t\,dt
   +\frac1{2\sqrt2}\int_0^\tau\frac{v(t)}{\sqrt t}\,dt
   \le\frac{H^2}{\sqrt2}e\sqrt\tau,\\
 S_{2,\mathrm{pr}}
 &\le\frac{\sqrt2H^2e}{\tau}\int_0^\tau v(t)\,dt
   \le\frac{H^4}{\sqrt2}e^2.
 \end{aligned}
 \label{move:probe-costs}
\end{equation}
Both probe costs will be included in the joint estimates.

\subsection{A Loss-Clock Occupation Estimate}
\label{move:subsec:occupation}

The following estimate uses standard Brownian tools.  Half-line
occupation penalisation has classical inverse-square-root decay
\cite{RVY}; the proof below gives the adapted-clock reduction explicitly.
The arcsine distributional input is also classical \cite{SS}.

\begin{lemma}[Adapted-clock occupation estimate]
\label{move:lem:occupation}
Let $(\Omega,\mathcal F,(\mathcal F_s)_{s\ge0},\mathbb P)$ satisfy
the usual filtration conditions, and let $W$ be Brownian motion
relative to this entire filtration.  Let $\alpha$ be predictable
or progressively measurable, with values in $[0,1]$.  For a
deterministic $z\in\mathbb R$, define
\begin{equation}
 K_s=\int_0^s(1-\alpha_u^2)\,du,\qquad
 Z_s=z+\frac1{\sqrt2}\int_0^s\alpha_u\,dW_u-\frac74K_s.
 \label{move:loss-clock}
\end{equation}
Assume
$\mathbf1_{\{Z_u\ge0\}}\alpha_u^2
\le\frac14\mathbf1_{\{Z_u\ge0\}}$
for $du\otimes d\mathbb P$-almost every $(u,\omega)$.
Then, for every finite $s\ge0$, with $z_-=\max\{0,-z\}$,
\begin{equation}
 \mathbb E e^{-3K_s}\le\frac{8(1+z_-)}{\sqrt{1+s}}.
 \label{move:occupation-bound}
\end{equation}
No Markov assumption or independence between a clock and its
time-changed Brownian motion is required.
\end{lemma}
\begin{proof}
Put $M_u=\int_0^u\alpha_v\,dW_v$ and
$\Theta_u=\langle M\rangle_u=u-K_u$.
Fix a finite terminal time $s$.  To cover finite total quadratic
variation, enlarge the space by an independent Brownian motion
$\widetilde W$ and extend the martingale as
\[
 \widehat M_u=M_{u\wedge s}+\widetilde W_{(u-s)_+}.
\]
With the filtration obtained from $\mathcal F_{u\wedge s}$ and
the independent increments revealed up to $(u-s)_+$, this is a
continuous martingale whose quadratic variation tends to infinity.
The Dambis--Dubins--Schwarz time change gives a standard Brownian
motion $B$ on the enlarged space such that, simultaneously for
$0\le u\le s$,
\begin{equation}
 M_u=B_{\Theta_u},\qquad
 Z_u=z+B_{\Theta_u}/\sqrt2-\frac74K_u.
 \label{move:dds}
\end{equation}
The extension changes neither paths nor probabilities up to $s$;
the coefficient condition is not needed beyond $s$.

Suppose $s\ge1$ and fix $0\le b\le s/2$.  Set
$T=s-b$ and $d=(-z+7b/4)_+$.  On $\{K_s\le b\}$,
$\Theta_s\ge T$ and $K_u\le b$ for all $u\le s$.
Consequently $B_{\Theta_u}/\sqrt2>d$ implies $Z_u>0$.
The change of variables for a continuous nondecreasing clock gives
$\int_0^{\Theta_s}F(v)\,dv
=\int_0^sF(\Theta_u)\,d\Theta_u$
for the nonnegative Borel functions used here.  Thus, on this event,
\begin{equation}
 \begin{aligned}
 O_T^d
 &:=\int_0^T\mathbf1_{\{B_v/\sqrt2>d\}}\,dv\\
 &\le\int_0^s
       \mathbf1_{\{B_{\Theta_u}/\sqrt2>d\}}\alpha_u^2\,du\\
 &\le\int_0^s\mathbf1_{\{Z_u>0\}}\alpha_u^2\,du\\
 &\le\frac13\int_0^s
       \mathbf1_{\{Z_u>0\}}(1-\alpha_u^2)\,du
 \le b/3.
 \end{aligned}
 \label{move:clock-inclusion-calculation}
\end{equation}
The almost-everywhere coefficient assumption suffices by Fubini's
theorem.  Flat portions of the clock have zero $d\Theta$ mass.
In particular, up to a null set this proves the event inclusion
\begin{equation}
 \{K_s\le b\}
 \subseteq
 \{O_{s-b}^{(-z+7b/4)_+}\le b/3\}.
 \label{move:event-inclusion}
\end{equation}
We take the unconditional Brownian probability of the event on the
right, without conditioning on the clock.

Here is the needed Brownian occupation calculation, including its
distributional input.  For a Brownian motion starting at zero,
\begin{equation}
 \mathbb P(O_\ell^0\le r)
 =\frac2\pi\arcsin\sqrt{r/\ell},
 \qquad 0\le r\le\ell,\quad \ell>0.
 \label{move:arcsine}
\end{equation}
To verify this formula, let $p>0$ and $q\ge0$.  A bounded, continuously
differentiable solution of
\[
 \frac12u''(x)-(p+q\mathbf1_{\{x>0\}})u(x)=-1,\qquad x\ne0,
\]
has the forms
$u(x)=(p+q)^{-1}+c_+\exp(-\sqrt{2(p+q)}x)$ for $x\ge0$ and
$u(x)=p^{-1}+c_-\exp(\sqrt{2p}x)$ for $x\le0$.
Continuity of $u$ and $u'$ at zero gives
$u(0)=1/\sqrt{p(p+q)}$.
It\^o's formula for
$\exp(-pt-qO_t^0)u(B_t)$, followed by expectation and passage
to infinite time, proves
\[
 \int_0^\infty e^{-p\ell}\mathbb E e^{-qO_\ell^0}\,d\ell
 =\frac1{\sqrt{p(p+q)}}.
\]
The matching first derivatives remove the local-time term, and bounded
$u,u'$ justify the stochastic integral and the terminal limit.
The candidate density has the same transform:
\[
 \int_0^\infty e^{-p\ell}\int_0^\ell
   \frac{e^{-qr}}{\pi\sqrt{r(\ell-r)}}\,dr\,d\ell
 =\frac1{\sqrt{p(p+q)}}.
\]
Indeed, substituting $\ell=r+w$ factors this into two Gaussian
integrals.  Uniqueness of Laplace transforms, first in $\ell$
and then in $q$, identifies the distribution; continuity in $\ell>0$
removes any exceptional times.  Integrating the density proves
\eqref{move:arcsine}.  The factor $1/\sqrt2$ in $O^0$ does not
change its positive-occupation time.

For $d\ge0$, let $\sigma_d=\inf\{v:B_v/\sqrt2=d\}$, with
$\sigma_0=0$.  The reflection principle, with the Brownian level
$\sqrt2d$, gives
\[
 \mathbb P(\sigma_d>T/2)
 =2\Phi(2d/\sqrt T)-1\le\frac{2d}{\sqrt T},
\]
where $\Phi$ is the standard normal distribution function.
On $\{\sigma_d\le T/2\}$ the remaining length
$\ell=T-\sigma_d$ is at least $T/2$.
The strong Markov property at $\sigma_d$ and
\eqref{move:arcsine}, using
$(2/\pi)\arcsin x\le x$ for $x\in[0,1]$, show that
for $0\le r\le T/3$,
\begin{equation}
 \mathbb P(O_T^d\le r)
 \le\frac{2d+\sqrt{2r}}{\sqrt T}
 \le\frac{2(d+\sqrt r)}{\sqrt T}.
 \label{move:brownian-occupation}
\end{equation}
This includes $d=0$ and $r=0$.
With $r=b/3$ and $T=s-b$, the required $r\le T/3$ holds.
Equations~\eqref{move:event-inclusion} and
\eqref{move:brownian-occupation} imply
\[
 \mathbb P(K_s\le b)
 \le\frac{2\sqrt2}{\sqrt s}
       \left(z_-+\frac74b+\sqrt{b/3}\right),
 \qquad 0\le b\le s/2.
\]
Finally Tonelli's theorem gives, even when $K_s$ has atoms,
\[
 \mathbb E e^{-3K_s}
 =\int_0^\infty3e^{-3b}\mathbb P(K_s\le b)\,db
 \le\frac{2\sqrt2}{\sqrt s}
       \left(z_-+\frac7{12}+\frac13\right)+e^{-3s/2}.
\]
Here the exponential density $3e^{-3b}$ has mean $1/3$,
and Jensen's inequality gives mean at most $1/3$ for $\sqrt{b/3}$.
For $s\ge1$, multiplication by $\sqrt{1+s}$ bounds the last
display by $4z_-+11/3+1<8(1+z_-)$.
For $0\le s\le1$, the bound $\mathbb E e^{-3K_s}\le1$
proves \eqref{move:occupation-bound} directly.
\end{proof}

\subsection{Joint Energy and Loss Bounds}
\label{move:subsec:joint}

\begin{proposition}[Joint bounds for the moving feedback]
\label{move:prop:joint}
For every choice in \eqref{move:parameters}, the policy
\eqref{move:policy} satisfies
\begin{equation}
 \begin{aligned}
 f\tau^3\left[1-H^2\left(e+\frac65f\tau+\frac3A\right)\right]
 &\le Q\le f\tau^3,\\
 S_1&\le
 \left(\frac{H^2}{\sqrt2}+\frac{H^4}{4\sqrt2 A}\right)f\tau^4
 \le f\tau^4,\\
 S_2&\le K_0(1+\log A)
       \frac{f^2\tau^4}{\sqrt{1+L}},\qquad K_0=32.
 \end{aligned}
 \label{move:joint-bounds}
\end{equation}
Both $S_1$ and $S_2$ include their probe contributions.
The first inequality remains valid when its left-hand side is negative.
In particular, if
\begin{equation}
 H^2\left(e+\frac65f\tau+\frac3A\right)\le\frac12,
 \label{move:energy-condition}
\end{equation}
then
\begin{equation}
 b_H(a)\ge f\tau^4
 \left\{\frac{\delta}{4\tau}-1
       -32(1+\log A)\frac f{\sqrt{1+L}}\right\}.
 \label{move:joint-improvement}
\end{equation}
\end{proposition}
\begin{proof}
For $t>\tau$, put
\[
 \alpha_t=(1-X_t^2)g_t(X_t),\qquad D_t=1-\alpha_t^2,\qquad
 v(t)=\mathbb E_\pi X_t^2,\qquad M(t)=\mathbb E_\pi X_t^4.
\]
It\^o's formula gives the exact equations
\begin{equation}
 v'(t)=\frac1{2t}\mathbb E_\pi(X_t^2\alpha_t^2),\qquad
 M'(t)=\frac3t\mathbb E_\pi(X_t^4\alpha_t^2).
 \label{move:moment-equations}
\end{equation}
Since $0\le\alpha_t\le1$, the probe bounds imply
\begin{equation}
 v(t)\le H^2e\sqrt{t/\tau},\qquad
 M(t)\le3H^4e^2(t/\tau)^3,\qquad \tau\le t\le1.
 \label{move:rough-moments}
\end{equation}
In particular, including the probe energy,
\[
 Q=e+\frac1{2H^2}\int_\tau^1
       \frac{\mathbb E_\pi[X_t^2g_t(X_t)^2]}t\,dt
 \le e+\frac e{\sqrt\tau}(1-\sqrt\tau)=\overline Q.
\]

\emph{Actual energy.}
Let $u_t^2=X_t^2/(kt^{7/2})$.  Globally,
\[
 D_t\le2X_t^2+2u_t^2,\qquad
 \left(\frac{v(t)}{\sqrt t}\right)'
 =-\frac{\mathbb E_\pi(X_t^2D_t)}{2t^{3/2}}
 \ge-t^{-3/2}M(t)-k^{-1}t^{-5}M(t).
\]
Integrating and using \eqref{move:probe-moments} and
\eqref{move:rough-moments} yields
\begin{equation}
 \begin{aligned}
 v(1)
 &\ge \frac{v(\tau)}{\sqrt\tau}
       -\int_\tau^1t^{-3/2}M(t)\,dt
       -\frac1k\int_\tau^1t^{-5}M(t)\,dt\\
 &\ge H^2\overline Q-\frac{H^4e^2}{\sqrt\tau}
       -\frac65H^4e^2\tau^{-3}
       -\frac{3H^4e^2}{k\tau^4}.
 \end{aligned}
 \label{move:second-moment-lower}
\end{equation}
The posterior equation on the entire time interval also gives
\[
 v(1)=H^2\mathbb E_\pi\int_0^1
               (1-X_t^2)^2a_t^2\,dt\le H^2Q.
\]
Divide \eqref{move:second-moment-lower} by $H^2$ and use
$e=f\tau^{7/2}$ and $k=Af$.  This proves the lower energy
bound in \eqref{move:joint-bounds}.

\emph{Alignment loss.}
For $g=(1+u^2)^{-1}$, the following deterministic inequalities hold
on the whole state space, with continuous extension at $u=0$:
\begin{equation}
 \frac{(1-g)^2}{u^2}
 =\frac{u^2}{(1+u^2)^2}
 \le\frac{1-g^2}{2}\le\frac D2,
 \qquad D=1-(1-X^2)^2g^2.
 \label{move:alignment-inequality}
\end{equation}
The normalized fourth moment $N(t)=t^{-3}M(t)$ is nonincreasing,
since \eqref{move:moment-equations} gives
\begin{equation}
 N'(t)=-3t^{-4}\mathbb E_\pi(X_t^4D_t).
 \label{move:normalized-fourth}
\end{equation}
Using $X_t^2=kt^{7/2}u_t^2$ in
\eqref{move:alignment-feedback}, including its time weight, we obtain
\begin{equation}
 \begin{aligned}
 S_{1,\mathrm{fb}}
 &=\frac1{2\sqrt2 k}\int_\tau^1
     t^{-4}\mathbb E_\pi\left[
       X_t^4\frac{(1-g_t(X_t))^2}{u_t^2}\right]\,dt\\
 &\le\frac{N(\tau)-N(1)}{12\sqrt2 k}
 \le\frac{H^4e^2}{4\sqrt2 k\tau^3}.
 \end{aligned}
 \label{move:alignment-paid}
\end{equation}
Adding \eqref{move:probe-costs} proves the stated $S_1$ bound:
the coefficient is at most $5/(4\sqrt2)<1$ for $H\le1$, $A\ge1$.

\emph{Fourth-moment change of measure.}
In logarithmic time $s=\log(t/\tau)$ set
$\widetilde X_s=X_{\tau\exp(s)}$.  Reversing the sign of the
time-changed innovation gives, exactly,
\begin{equation}
 d\widetilde X_s=\frac{\widetilde X_s\alpha_s}{\sqrt2}\,dW_s,
 \qquad
 \alpha_s=(1-\widetilde X_s^2)
                g_{\tau\exp(s)}(\widetilde X_s),\qquad 0\le s\le L.
 \label{move:log-diffusion}
\end{equation}
Explicitly, $W_s=-\int_\tau^{\tau\exp(s)}t^{-1/2}\,dI_t$ is Brownian
motion in the full time-changed mixture filtration.
Its future increments are independent of the entire
probe filtration at the deterministic time $\tau$.
Because the coefficients depend only on time and the current state
and strong uniqueness holds, we may first condition on the probe
and then work with any initial state $\widetilde X_0=x\in(-1,1)$.
If $x=0$, the process stays zero.  If $x\ne0$, boundedness of
$\alpha$ and the stochastic-exponential representation of
\eqref{move:log-diffusion} prevent hitting zero in finite time.
Also $|\widetilde X_s|<1$, since the score remains finite.
Thus $\log|\widetilde X_s|$ is legitimate for either sign of $x$.

For such an $x$, define on each finite time interval
\begin{equation}
 \begin{aligned}
 \mathcal M_s
 &:=\frac{\widetilde X_s^4}{x^4}
       \exp\left(-3\int_0^s\alpha_u^2\,du\right)\\
 &=\exp\left(2\sqrt2\int_0^s\alpha_u\,dW_u
                       -4\int_0^s\alpha_u^2\,du\right).
 \end{aligned}
 \label{move:fourth-density}
\end{equation}
Its stochastic-exponential integrand is bounded by $2\sqrt2$,
so Novikov's criterion makes it a true martingale on every finite
interval.  Under the corresponding law $\mathbb P_x^{(4)}$,
$W_s^{(4)}=W_s-\int_0^s2\sqrt2\alpha_u\,du$ is Brownian motion
relative to that same filtration, and
\begin{equation}
 \mathbb E_x\widetilde X_s^4
 =x^4e^{3s}\mathbb E_x^{(4)}e^{-3K_s},\qquad
 K_s=\int_0^s(1-\alpha_u^2)\,du.
 \label{move:exact-fourth-tilt}
\end{equation}
No limiting change of measure as $L\to\infty$ is used.

Define the log distance to the moving threshold by
\begin{equation}
 Z_s=\log\frac{|\widetilde X_s|}
                    {\sqrt{k}(\tau e^s)^{7/4}},\qquad
 z=\log\frac{|x|}{\sqrt{k}\tau^{7/4}}.
 \label{move:threshold-coordinate}
\end{equation}
Under the tilted law, the drift of $\log|\widetilde X_s|$ is
$7\alpha_s^2/4$.  Therefore
\[
 Z_s=z+\frac1{\sqrt2}\int_0^s\alpha_u\,dW_u^{(4)}
               -\frac74\int_0^s(1-\alpha_u^2)\,du.
\]
Moreover,
$g_{\tau e^s}(\widetilde X_s)=(1+e^{2Z_s})^{-1}$.
On $\{Z_s\ge0\}$ this is at most $1/2$, and hence
$\alpha_s^2\le1/4$.  The coefficient is continuous and adapted,
so all hypotheses of Lemma~\ref{move:lem:occupation} hold.
Combining that lemma with \eqref{move:exact-fourth-tilt} gives
\begin{equation}
 \mathbb E_x\widetilde X_s^4
 \le\frac{8x^4e^{3s}}{\sqrt{1+s}}
       \left(1+\log_+\frac{\sqrt{k}\tau^{7/4}}{|x|}\right),
 \label{move:conditional-fourth}
\end{equation}
where $\log_+r=\max\{0,\log r\}$.
All posterior saturation factors have been retained.

\emph{Integration over the probe.}
Since $\sqrt{k}\tau^{7/4}=\sqrt{Ae}$ and $A\ge1$, for every
real $x$, with the product extended continuously at zero,
\begin{equation}
 x^4\log_+\frac{\sqrt{Ae}}{|x|}
 \le\frac{\log A}{2}x^4+\frac{e^2}{4\exp(1)}.
 \label{move:initial-log}
\end{equation}
Indeed, put $r=|x|/\sqrt e$ and use
$\log_+(\sqrt A/r)\le(\log A)/2+\log_+(1/r)$ and
$\sup_{0<r\le1}r^4\log(1/r)=1/(4\exp(1))$.
Apply \eqref{move:conditional-fourth} conditional on the probe,
then use \eqref{move:initial-log} and
$\mathbb E_\pi X_\tau^4\le3H^4e^2$.  As $H\le1$,
$8[3(1+\frac12\log A)+1/(4\exp(1))]\le32(1+\log A)$.
We obtain the uniform estimate
\begin{equation}
 M(t)\le
 \frac{32(1+\log A)e^2(t/\tau)^3}
      {\sqrt{1+\log(t/\tau)}},\qquad \tau\le t\le1.
 \label{move:fourth-bound}
\end{equation}
This integration does not assume that a conditional random
information clock is a Gaussian variance.

\emph{Nonlinear loss.}
Equation~\eqref{move:nonlinear-feedback} implies
$S_{2,\mathrm{fb}}\le2^{-1/2}\int_\tau^1M(t)t^{-1}\,dt$.
Splitting at $\sqrt\tau$ gives
\begin{equation}
 \begin{aligned}
 \int_\tau^1\frac{t^2}{\sqrt{1+\log(t/\tau)}}\,dt
 &\le\frac{\tau^{3/2}}3+\frac1{3\sqrt{1+L/2}}\\
 &\le\frac{1+\sqrt2}{3\sqrt{1+L}}
 \le\frac1{\sqrt{1+L}}.
 \end{aligned}
 \label{move:time-integral}
\end{equation}
Here $\tau^{3/2}\sqrt{1+L}\le1$; likewise
$\tau^3\sqrt{1+L}\le1$.
Consequently \eqref{move:probe-costs} and
\eqref{move:fourth-bound} imply
\[
 S_2\le\frac{H^4e^2}{\sqrt2}
       +\frac{32(1+\log A)e^2\tau^{-3}}{\sqrt{2(1+L)}}
 \le32(1+\log A)\frac{f^2\tau^4}{\sqrt{1+L}}.
\]
For the last inequality, use $e^2\tau^{-3}=f^2\tau^4$,
$H^4\le1$, $1+\log A\ge1$, and $33/\sqrt2<32$.
This includes the probe nonlinear loss and proves $K_0=32$.
Finally \eqref{move:energy-condition} gives $Q\ge f\tau^3/2$.
Since $\delta\ge0$ for $H\ge H_c$, substitution into
\eqref{move:loss-identity} proves \eqref{move:joint-improvement}.
\end{proof}

\subsection{The Explicit Design and the Full Parameter Interval}
\label{move:subsec:interval}

\begin{proof}[Proof of Theorem~\ref{thm:moving}]
Take exactly the design in \eqref{eq:moving-design}:
\begin{equation}
 \begin{gathered}
 0<\delta\le\frac1{16},\qquad H^2=\frac{1+\delta}{\sqrt2},
 \qquad A=24,\qquad
 \kappa=\frac1{128(1+\log24)},\\
 \tau=\delta/8,\qquad L=\log(1/\tau),\qquad
 f=\kappa\sqrt{1+L},\qquad e=f\tau^{7/2},\qquad k=24f.
 \end{gathered}
 \label{move:design}
\end{equation}
This lies in \eqref{move:parameters}, and $4/5<H<1$.
Set $L_0=\log128$.  Then $\tau\le1/128$, $L\ge L_0$,
$1+L_0<6$, and $\kappa\le1$.  The functions
$\sqrt{1+L}e^{-L}$ and $\sqrt{1+L}e^{-7L/2}$
are decreasing on $[L_0,\infty)$.  Hence, over the entire stated
$\delta$ interval,
\begin{equation}
 \begin{gathered}
 f\tau<\frac{\sqrt6}{128}<\frac1{50},\qquad
 e<\frac{\sqrt6}{128^{7/2}}<10^{-6},\\
 H^2\left(e+\frac65f\tau+\frac3A\right)
 <10^{-6}+\frac3{125}+\frac18<\frac3{20}<\frac12.
 \end{gathered}
 \label{move:explicit-energy-check}
\end{equation}
Thus Proposition~\ref{move:prop:joint} gives $Q\ge f\tau^3/2$.
Also $\delta/(4\tau)=2$ and
$32(1+\log24)f/\sqrt{1+L}=1/4$.
In particular, \eqref{move:joint-improvement} gives the convenient bound
\begin{equation}
 b_H(a)\ge\frac12f\tau^4
 =\frac{\kappa}{2\cdot8^4}
       \delta^4\sqrt{1+\log(8/\delta)}.
 \label{move:endpoint-improvement}
\end{equation}
It remains to prove that this endpoint average controls the maximum
risk; no least-favorable-prior assertion is being used.

Fix the design $H,\tau,e,k$ while varying the unknown parameter $h$.
The same observable rule, deterministic energy cap, and evenness
established in \eqref{move:policy}--\eqref{move:evenness} allow
the full-record likelihood argument underlying
\eqref{crit:risk-derivative} to be applied here.  We give the
argument and its constants for this design.
The cap makes
$\exp(h\Xi_1-h^2\mathcal Q/2)$ a true likelihood relative to
$\mathbb P_0$; Girsanov's theorem and uniqueness identify the law
of the same rule under $h$.  Consequently, for $|h|\le H$,
\begin{equation}
 \frac{d\mathbb P_h}{d\mathbb P_\pi}
 =\frac{\exp\{h\Xi_1+(H^2-h^2)\mathcal Q/2\}}
        {\cosh(H\Xi_1)}
 \le2e^{H^2e/2}\tau^{-1/4}.
 \label{move:interval-likelihood}
\end{equation}
We used $e^{hx}/\cosh(Hx)\le2$ and the cap
$\mathcal Q\le e+L/(2H^2)$.
The already proved upper bound $Q\le e/\sqrt\tau$ therefore gives
\begin{equation}
 \sup_{|h|\le H}\mathbb E_h\mathcal Q
 \le U:=2e^{H^2e/2}e\tau^{-3/4}
       =2e^{H^2e/2}f\tau^{11/4}.
 \label{move:interval-energy}
\end{equation}
Here no posterior interpretation under interior parameter laws is needed.

Let $\mathcal A=\int_0^1a_t\,dt$ and
$R_h=\Xi_1-h\mathcal Q=\int_0^1a_t\,dB_t$ under $\mathbb P_h$.
Then
\[
 \mathcal A^2\le\mathcal Q\le\Lambda,\qquad
 \mathbb E_hR_h^2=\mathbb E_h\mathcal Q,\qquad
 \mathbb E_he^{\lambda R_h}\le e^{\lambda^2\Lambda/2}
 \quad(\lambda\in\mathbb R).
\]
For a bounded parameter-independent record functional $F$,
\[
 \mathbb E_{h+u}F
 =\mathbb E_h[F\exp\{uR_h-u^2\mathcal Q/2\}],\qquad
 \partial_h\mathbb E_hF=\mathbb E_h(FR_h).
\]
The exponential moment bound justifies dominated differentiation,
in particular for $F=\mathcal Q$ and $F=\mathcal A$.
The derivative holds with the design fixed, not with the design
varied as a function of $h$.
Expanding the loss as
$\mathcal L_h=h^2/2+\mathbb E_h\mathcal Q/2+h\mathbb E_h\mathcal A$
and applying Cauchy--Schwarz yields
\begin{equation}
 \begin{aligned}
 |\mathcal L_h'-h|
 &\le\frac12|\mathbb E_h(\mathcal Q R_h)|
       +|\mathbb E_h\mathcal A|
       +H|\mathbb E_h(\mathcal A R_h)|\\
 &\le\sqrt U+(H+\tfrac12\sqrt\Lambda)U=:V.
 \end{aligned}
 \label{move:risk-derivative}
\end{equation}
Indeed, $\mathbb E_h\mathcal Q^2\le\Lambda\mathbb E_h\mathcal Q$
and $\mathbb E_h\mathcal A^2\le\mathbb E_h\mathcal Q$.

All constants needed to locate the maximum are uniform on
$0<\delta\le1/16$.  Since $H^2\ge1/\sqrt2$ and
\eqref{move:explicit-energy-check} holds,
\begin{equation}
 \begin{gathered}
 \Lambda\le1+L,\qquad
 U\le3\sqrt{1+L}e^{-11L/4}<10^{-4},\\
 U\sqrt\Lambda\le3(1+L)e^{-11L/4}<10^{-4}.
 \end{gathered}
 \label{move:uniform-smallness}
\end{equation}
The two functions of $L$ decrease on $[L_0,\infty)$.
At $L_0$, their values are smaller than
$3\sqrt6/2^{19}$ and $18/2^{19}$, respectively, each below $10^{-4}$.
It follows that
\[
 V<\frac1{100}+\frac1{10000}+\frac1{20000}
   <\frac1{50}<\frac H2.
\]
Thus the risk is strictly increasing on $[H/2,H]$.
For the remaining central interval, the loss expansion and
\eqref{move:interval-energy} give
\[
 \sup_{|h|\le H/2}\mathcal L_h(a)
 \le\frac{H^2}{8}+\frac U2+H\sqrt U,\qquad
 \mathcal L_H(a)\ge\frac{H^2}{2}-H\sqrt U.
\]
Moreover
$U/2+2H\sqrt U<1/20000+1/50<3H^2/8$, since $H>4/5$.
The central risks are therefore strictly smaller than the endpoint
risk.  Together with evenness, this proves
\begin{equation}
 \sup_{|h|\le H}\mathcal L_h(a)
 =\mathcal L_H(a)=\mathcal L_{-H}(a)
 =\frac{H^2}{2}-b_H(a).
 \label{move:endpoint-maximality}
\end{equation}
The rough bounds involving $\sqrt U$ have only located the maximum;
they have not been subtracted from the endpoint improvement.
Finally, \eqref{move:endpoint-improvement} implies
\begin{equation}
 \begin{gathered}
 \sup_{|h|\le H}\mathcal L_h(a)
 \le\frac{H^2}{2}-c\delta^4\sqrt{\log(1/\delta)},\\
 c=\frac{\kappa}{2\cdot8^4}
   =\frac1{2^{20}(1+\log24)}.
 \end{gathered}
 \label{move:full-interval-guarantee}
\end{equation}
This proves Theorem~\ref{thm:moving}.  Taking the infimum over
admissible policies also proves
$\Delta(H)\ge c\delta^4\sqrt{\log(1/\delta)}$, the constructive
side of Theorem~\ref{thm:critical-new}.
\end{proof}

\subsection{Sharpness of the Energy Order and Logarithmic Deficit}
\label{move:subsec:counterexample}

\begin{proof}[Proof of Corollary~\ref{cor:no-four-thirds}]
For the explicit policies \eqref{move:design},
Proposition~\ref{move:prop:joint} and
\eqref{move:explicit-energy-check} give
\begin{equation}
 \frac12f\tau^3\le Q\le f\tau^3,\qquad
 0\le\mathcal S_H(a)=S_1+S_2\le\frac54f\tau^4.
 \label{move:counterexample-costs}
\end{equation}
As $\delta\downarrow0$, $f=\kappa\sqrt{1+\log(8/\delta)}$
tends to infinity, whereas $f\tau^3$ tends to zero.
Thus the actual energy $Q$ is strictly positive and tends to zero,
and
\begin{equation}
 0\le\frac{\mathcal S_H(a)}{Q^{4/3}}
 \le\frac54\,2^{4/3}f^{-1/3}
 =\frac54\,2^{4/3}\kappa^{-1/3}
       (1+\log(8/\delta))^{-1/6}\longrightarrow0.
 \label{move:counterexample-ratio}
\end{equation}
For example, take $\delta_n=2^{-n-4}$ for $n\ge1$,
$H_n=((1+\delta_n)/\sqrt2)^{1/2}$, and use
\eqref{move:design} at each $H_n$.  Then $H_n\downarrow H_c$,
$0<Q_n\to0$, and the ratio in \eqref{move:counterexample-ratio}
tends to zero.  Given any $c_*>0$, $q_0>0$, and $\eta>0$,
this sequence eventually has
$H_c<H_n<H_c+\eta$, $0<Q_n\le q_0$, and
$\mathcal S_{H_n}(a^{(n)})<c_*Q_n^{4/3}$.
Hence no positive four-thirds bound can hold uniformly over that
critical neighborhood and all admissible algorithms of sufficiently
small positive energy.  These counterexamples are the same
implementable policies with the full-interval improvement already
proved above.

To obtain the matching assertions, \eqref{move:counterexample-costs}
and $\tau=\delta/8$ give
$Q\asymp\delta^3\sqrt{\log(1/\delta)}$.
In particular,
\[
 \log(e/Q)=3\log(1/\delta)
          -\tfrac12\log\log(1/\delta)+O(1)
          \asymp\log(1/\delta).
\]
Thus \eqref{move:counterexample-ratio} also gives
$\mathcal S_H(a)\lesssim Q^{4/3}/[\log(e/Q)]^{1/6}$ along these
policies. For all sufficiently small $\delta$, $0<Q\le1$ and
$H\in[H_c,1]$, so Theorem~\ref{thm:energy-new} supplies the reverse bound.
This proves the two asserted orders and shows that multiplying the
uniform deficit lower bound by any diverging function of $1/Q$ is
impossible along this same sequence.
\end{proof}

\section{Exact finite-episode no-learning regions and task corrections}
\label{fin:sec:proof}

We prove Theorems~\ref{thm:finite-zero} and~\ref{thm:finite-positive},
Corollary~\ref{cor:finite-expansion}, and
Proposition~\ref{prop:finite-nonuniversal}.
Throughout, \(d,N\ge1\), \(g\) is nonconstant, convex, and \(C^4\),
\(\|g'\|_\infty\le L\), \(0\le g''\le M\), and its third and fourth
derivatives are bounded. We retain the full-record, parameter-independent
randomization and energy-stopped uniqueness conventions of the model.
No symmetry or boundedness of \(g\) is imposed. Every parameter used below
must lie in the model's parameter domain.
Write
\[
 \begin{gathered}
 F(t,x)=P_{1-t}g(x),\qquad S(t,x)=F_x(t,x),\qquad
 \mu=\mathbb E g(Z),\quad \xi=g(Z)-\mu,\quad c=\mathbb E\xi^2>0,\\
 \kappa_3=\mathbb E\xi^3,\qquad
 \kappa_4=\mathbb E\xi^4-3c^2,\qquad
 \psi_g(u)=\log\mathbb E e^{-u\xi},\quad Z\sim N(0,1).
 \end{gathered}
\]
The heat equation, It\^o's formula, and isometry give
\begin{equation}\label{fin:heat}
 |S|\le L,\quad |S_x|\le M,\qquad
 g(W_1)-\mu=\int_0^1S(t,W_t)dW_t,\qquad
 \mathbb E\int_0^1S(t,W_t)^2dt=c.
\end{equation}

\subsection{Exact subtraction of the zero-control baseline}

Put \(\theta=N^{-1/4}h\),
\(r_N(\mathcal A,h)=R_N^g(\mathcal A,\theta)/\sqrt N\), and
\(z_N(h)=r_N(0,h)\).
All finite-risk rules have finite expected total energy. Indeed, with
\(A_k=\int_0^1\|U_{k,t}\|^2dt\), linear growth gives
\[
 \frac12A_k+g(X_{k,1})
 \ge\frac14A_k+g(0)-L|W_{k,1}|-L^2\|\theta\|^2.
\]
This is an integrable lower bound, even when \(g\) is unbounded below;
finite expected energy also makes the terminal costs absolutely integrable.
For such a rule, apply It\^o's formula to \(F(t,X_{k,t})\).
The stochastic integral is square integrable and the drift is integrable.
The linear growth of \(F\) and the finite second moment of
\(\sup_{t\le1}|X_{k,t}|\) remove localization. Subtracting the actual
zero-control cost at the same parameter therefore gives the exact identity
\begin{equation}\label{fin:baseline-subtraction}
 r_N(\mathcal A,h)-z_N(h)
 =\frac1{\sqrt N}\mathbb E_\theta\sum_{k=1}^N\int_0^1
 \left\{\frac12\|U_{k,t}\|^2+
 S(t,X_{k,t})\langle\theta,U_{k,t}\rangle\right\}dt .
\end{equation}
In particular the oracle cancels without approximation.

For completeness, the known-parameter value for \(\lambda=\|\theta\|^2>0\) is
\[
 V_\theta(t,x)=-\lambda^{-1}\log P_{1-t}(e^{-\lambda g})(x).
\]
Gaussian exponential integrability follows from linear growth.
Differentiation gives
\(V_t+\tfrac12V_{xx}-\tfrac\lambda2V_x^2=0\),
\(|V_x|\le L\), and \(|V_{xx}|\le M+\lambda L^2\).
Completing the square in It\^o's formula verifies this value for all
finite-energy rules; the bounded Lipschitz feedback \(U=-\theta V_x\)
attains it. At \(\lambda=0\) the value is \(F\).
Consequently, with the quotient continuously extended at zero,
\begin{equation}\label{fin:baseline}
 z_N(h)=\sqrt N\,\frac{\psi_g(\|h\|^2/\sqrt N)}
                              {\|h\|^2/\sqrt N},\qquad
 Z_{N,d}^g(H)=z_N(He_1).
\end{equation}
The latter is exactly the worst normalized regret of zero control.
In fact, under the exponential tilt with density proportional to
\(e^{-u\xi}\), \(\psi_g''(u)=\operatorname{Var}_u(\xi)>0\).
Since \(\psi_g(0)=\psi_g'(0)=0\),
\begin{equation}\label{fin:radial}
 \left(\frac{\psi_g(u)}u\right)'
 =\frac1{u^2}\int_0^u t\psi_g''(t)dt>0\qquad(u>0).
\end{equation}
Thus \(z_N\) is strictly increasing in radius. This assertion concerns
only zero control, not the worst parameter of an arbitrary rule.

\subsection{Energy-weighted lower bounds}

The case \(H=0\) is immediate. Suppose \(H>0\), put
\(r=N^{-1/4}H\), \(b=r/\sqrt d\), and use the sphere-supported prior
\(\theta_i=b\sigma_i\), with independent uniform signs \(\sigma_i\).
In this subsection expectations include this prior and the algorithm's
seed. Let
\[
 Q=\mathbb E\sum_{k=1}^N\int_0^1\|U_{k,t}\|^2dt<\infty .
\]
Concatenate the episode increments into a continuous observation
\(\mathsf X_s\), \(0\le s\le N\); thus
\(\mathsf X_{k-1+t}=\sum_{j<k}X_{j,1}+X_{k,t}\).
The same convention is used for stochastic integrals and controls.
For coordinate \(i\), enlarge the record filtration only in the proof:
\[
 n_{i,s}=\mathbb E[\theta_i\mid\mathcal F_s,\theta_{-i}],\qquad
 m_s=\mathbb E[\theta\mid\mathcal F_s].
\]
No revealed coordinate is supplied to the algorithm.
To justify the posterior formula, fix the revealed coordinates and set
\(\zeta_{i,s}=\sum_{j\ne i}\theta_jU_{j,s}\).
Stop the rule when its cumulative energy reaches a fixed cap.
Relative to the zero-parameter full-record law, the two conditional
densities are the stochastic exponentials with drifts
\(\zeta_i\pm bU_i\). Their drift energies are bounded by \(r^2\)
times the cap, so they are true martingales. Stopped uniqueness
identifies these with the conditional experiment laws. Their log ratio is
\[
 2bT_{i,s},\qquad
 T_{i,s}=\int_0^sU_{i,u}d\mathsf X_u
              -\int_0^sU_{i,u}\zeta_{i,u}du .
\]
Bayes' formula gives \(n_{i,s}=b\tanh(bT_{i,s})\).
Finite path energy removes the cap on the stopped records; bounded
posteriors identify the resulting formula in the original experiment.
Drift projection and L\'evy's characterization give an innovation Brownian
motion \(I^{(i)}\) in each enlarged filtration. It\^o's formula then yields
\begin{equation}\label{fin:posterior}
 dn_{i,s}=(b^2-n_{i,s}^2)U_{i,s}dI_s^{(i)},\qquad n_{i,0}=0,\qquad
 \mathbb E n_{i,s}^2\le b^4\mathbb E\int_0^sU_{i,u}^2du.
\end{equation}
The innovations for different coordinates need not be independent.
The integrands in \eqref{fin:posterior} are square integrable by the
first energy moment. Conditional Jensen gives, for
\(q(s)=\mathbb E\int_0^s\|U_u\|^2du\),
\[
 \mathbb E\|m_s\|^2\le b^4q(s),\qquad q(N)=Q .
\]
Projection in \eqref{fin:baseline-subtraction}, followed by two
Cauchy--Schwarz inequalities, now gives
\[
 \mathbb E\int_0^NS_s\langle\theta,U_s\rangle ds
 \ge-Lb^2\int_0^N\sqrt{q(s)q'(s)}ds
 \ge-Lb^2\sqrt{N/2}\,Q .
\]
Here \(S_s=S(t,X_{k,t})\) on episode \(k\), and
\(\int_0^Nqq'\,ds=Q^2/2\) for the absolutely continuous function \(q\).
All prior support points have the same baseline, so
\begin{equation}\label{fin:global-energy}
 \mathbb E_\pi r_N(\mathcal A,h)
 \ge Z_{N,d}^g(H)+
 \left(\frac12-\frac{H^2L}{d\sqrt2}\right)\frac Q{\sqrt N}.
\end{equation}
This proves the first sufficient condition
\(H^4L^2\le d^2/2\), including equality.

For the sharper task-dependent condition, write
\[
 A_{i,k}=\mathbb E\int_0^1U_{i,k,t}^2dt,\qquad
 Q_{i,k}=\sum_{j\le k}A_{i,j},\qquad Q_{i,0}=0.
\]
Let \(n_{i,k,0}\) be the posterior at the start of episode \(k\).
The actual driver \(W_k\), used only for this proof, is independent of
the parameter, the preceding record, and the seed. Decompose the
projected cross term exactly as \(T_0+T_1+T_2\), where
\begin{align*}
 T_0&=\mathbb E\sum_{i,k}\int_0^1
          n_{i,k,0}U_{i,k,t}S(t,W_{k,t})dt,\\
 T_1&=\mathbb E\sum_{i,k}\int_0^1
          (n_{i,k,t}-n_{i,k,0})U_{i,k,t}S(t,X_{k,t})dt,\\
 T_2&=\mathbb E\sum_{i,k}\int_0^1
          n_{i,k,0}U_{i,k,t}
          [S(t,X_{k,t})-S(t,W_{k,t})]dt .
\end{align*}
Independence and \eqref{fin:heat} imply
\(\mathbb E\int n_{i,k,0}^2S(t,W_{k,t})^2dt
=c\mathbb E n_{i,k,0}^2\le cb^4Q_{i,k-1}\).
No independence of \(U\) and \(W_k\) is needed. Hence
\begin{equation}\label{fin:frozen}
 |T_0|\le b^2\sqrt c\sum_{i,k}\sqrt{Q_{i,k-1}A_{i,k}}
 \le b^2\sqrt{Nc/2}\,Q,
\end{equation}
using
\(\sum_kQ_{i,k-1}A_{i,k}
=(Q_{i,N}^2-\sum_kA_{i,k}^2)/2\le Q_{i,N}^2/2\).
For \(q_{i,k}(t)=\mathbb E\int_0^tU_{i,k,u}^2du\),
\eqref{fin:posterior} gives
\(\mathbb E(n_{i,k,t}-n_{i,k,0})^2\le b^4q_{i,k}(t)\).
The same integral inequality on an interval of length one yields
\begin{equation}\label{fin:within}
 |T_1|\le Lb^2\sum_{i,k}\int_0^1
          \sqrt{q_{i,k}(t)q_{i,k}'(t)}dt
 \le\frac{Lb^2}{\sqrt2}Q .
\end{equation}
Finally, although the coordinates use different enlarged filtrations,
their bounds \(|n_{i,k,0}|\le b\) hold simultaneously. Thus, pathwise,
\[
 \left|\sum_i n_{i,k,0}U_{i,k,t}\right|\le r\|U_{k,t}\|,\qquad
 |S(t,X_{k,t})-S(t,W_{k,t})|
 \le Mr\int_0^t\|U_{k,u}\|du .
\]
Integration and Cauchy--Schwarz on each episode give
\begin{equation}\label{fin:drift}
 |T_2|\le\frac{Mr^2}{2}\mathbb E\sum_k
          \left(\int_0^1\|U_{k,t}\|dt\right)^2
 \le\frac{Mr^2}{2}Q .
\end{equation}
Combining \eqref{fin:frozen}--\eqref{fin:drift} proves
\begin{equation}\label{fin:task-energy}
 \mathbb E_\pi r_N(\mathcal A,h)\ge Z_{N,d}^g(H)+
 \left[\frac12-\frac{H^2}{d}\sqrt{\frac c2}
 -\frac{H^2}{\sqrt N}\left(\frac{L}{d\sqrt2}+\frac M2\right)\right]
 \frac Q{\sqrt N}.
\end{equation}
If the bracket is nonnegative, the Bayes lower bound and zero control
give \(C_{N,d}^g(H)=Z_{N,d}^g(H)\). This completes
Theorem~\ref{thm:finite-zero}. The argument covers the original
unbounded-action class and uses no second moment of control energy.

\subsection{Strictly subcritical minimax expansion}

Put \(\gamma(H)=\tfrac12-\tfrac{H^2}{d}\sqrt{c/2}\).
Whenever \(\gamma(H)>0\) and
\begin{equation}\label{fin:eventual}
 N\ge
 \left[\frac{H^2\{L/(d\sqrt2)+M/2\}}{\gamma(H)}\right]^2,
\end{equation}
\eqref{fin:task-energy} and \eqref{fin:baseline} identify the minimax
value exactly. A common lower bound on \(N\) works for
\(0\le H\le H_*<(d^2/(2c))^{1/4}\).
Since \(g\) is Lipschitz, \(g(Z)\) has all real exponential moments and
\(\psi_g\) is analytic near zero. In particular
\[
 \psi_g(u)=\frac{cu^2}{2}-\frac{\kappa_3u^3}{6}
                    +\frac{\kappa_4u^4}{24}+O_g(u^5).
\]
Substitution into the exact value gives, uniformly on that interval,
\begin{equation}\label{fin:expansion}
 C_{N,d}^g(H)=\frac{cH^2}{2}
 -\frac{\kappa_3H^4}{6\sqrt N}
 +\frac{\kappa_4H^6}{24N}+O_{g,H_*}(N^{-3/2}).
\end{equation}
The formula extends continuously to \(H=0\).
This proves Corollary~\ref{cor:finite-expansion}; it is a minimax
expansion, not merely an expansion of a candidate rule.

\subsection{A full-record second variation}

\begin{lemma}[Uniform second variation]\label{fin:lem:variation}
Consider \(N\ge1\) scalar episodes
\(dX_{k,t}=\theta U_{k,t}dt+dW_{k,t}\), each starting at zero.
Let \(v\) be a fixed parameter-independent predictable scalar path rule,
possibly using an auxiliary seed independent of the drivers.
Assume a deterministic bound \(|v|\le B<\infty\) on its records.
Let \(\rho\) be an independent uniform sign, not supplied as an argument
to \(v\), and set \(U^\varepsilon=\varepsilon\rho v(X)\).
Fix \(r_0<\infty\). Assume these rules and their energy-stopped versions
have nonexplosive unique strong solutions for
\(|\theta|\le r_0\) and all sufficiently small \(|\varepsilon|\).
Under the zero-drift reference law, concatenate the drivers into
\(\mathsf W_s\), \(0\le s\le N\), and define
\[
 M_s=\int_0^sv_u(W)d\mathsf W_u,\quad
 A_s=\int_0^sv_u(W)^2du,\quad
 Q(v)=\mathbb E_0 A_N,\quad
 K(v)=\mathbb E_0\int_0^NS_sM_sv_s(W)ds,
\]
where \(S_s=S(t,W_{k,t})\) on episode \(k\).
For the unnormalized cumulative cost
\[
 J_\theta(U)=\mathbb E_\theta\sum_{k=1}^N
       \left[\frac12\int_0^1U_{k,t}^2dt+g(X_{k,1})\right],
\]
including all randomization in the expectation, one has
\begin{equation}\label{fin:variation}
 J_\theta(U^\varepsilon)-J_\theta(0)
 =\varepsilon^2\left\{\frac12Q(v)+\theta^2K(v)\right\}
       +O_{N,g,B,r_0}(\varepsilon^4),
\end{equation}
uniformly for \(|\theta|\le r_0\). In particular this is a compact-uniform
expansion for every such fixed rule. Subtracting the same oracle from
both costs leaves it unchanged.
\end{lemma}
\begin{proof}
Work on the reference record containing the path, the auxiliary seed,
and \(\rho\). The deterministic bound on \(A_N\) makes
\[
 L_{\theta,\rho}^\varepsilon
 =\exp\{\theta\varepsilon\rho M_N
                       -\theta^2\varepsilon^2A_N/2\}
\]
a true full-record density. Girsanov and uniqueness identify it with
the controlled law. Its conditional expectation given the initial
seeds is one, so their distribution is unchanged.
Under the reference law \(v\), \(M\), and \(A\) do not depend on \(\rho\).
Averaging this sign gives
\begin{equation}\label{fin:cosh}
 \overline L_\theta^\varepsilon
 =e^{-\theta^2\varepsilon^2A_N/2}\cosh(\theta\varepsilon M_N)
 =1+\frac{\theta^2\varepsilon^2}{2}(M_N^2-A_N)
       +O_{L^p}(\varepsilon^4)
\end{equation}
for every fixed finite \(p\), uniformly for \(|\theta|\le r_0\).
Indeed \(A_N\le NB^2\) and
\(\mathbb E_0e^{aM_N}\le e^{a^2NB^2/2}\) for every real \(a\).
Taylor remainders are therefore bounded in \(L^p\) by polynomials
in \(|M_N|\) times an exponential of a fixed multiple of \(|M_N|\).
No conditional Gaussian distribution of \(M_N\) is asserted.

Let \(G=\sum_k[g(W_{k,1})-\mu]=\int_0^NS_s d\mathsf W_s\).
It has all finite moments. Since \(\mathbb E_0\overline L=1\),
the cost difference equals
\(\mathbb E_0[G(\overline L-1)]
+\tfrac{\varepsilon^2}{2}\mathbb E_0[A_N\overline L]\).
The \(L^p\) remainder in \eqref{fin:cosh} can be multiplied by \(G\)
using H\"older's inequality. Moreover, It\^o's formula and isometry give
\[
 M_N^2-A_N=2\int_0^NM_sv_s\,d\mathsf W_s,\qquad
 \mathbb E_0[G(M_N^2-A_N)]=2K(v).
\]
These stochastic integrals are square integrable because \(v\) is bounded.
The change of measure in the energy term contributes only
\(O(\varepsilon^4)\), proving \eqref{fin:variation}.
\end{proof}

\subsection{A bounded direction giving strict improvement}

Let \(N\ge2\) and define
\[
 \beta_g=\frac{\kappa_3}{3c},\qquad
 a_N=\frac1N\sum_{k=2}^N
 \sqrt{\frac{2\sqrt{k-1}}{\sqrt k+\sqrt{k-1}}}.
\]
We prove Theorem~\ref{thm:finite-positive} under its strict condition
\begin{equation}\label{fin:positive-condition}
 \frac{H^2}{d}
 \left\{\sqrt{\frac c2}\,a_N-\frac{\beta_g}{\sqrt N}\right\}>\frac12.
\end{equation}
First construct an unbounded reference direction solely to compute its
second variation. On independent Brownian episodes let
\(\xi_k=g(W_{k,1})-\mu\), \(v_{1,t}=S(t,W_{1,t})\), and, for \(k\ge2\), set
\begin{equation}\label{fin:direction}
 b_k=\left[\frac{\sqrt{k/(k-1)}-1}{c}\right]^{1/2},\qquad
 v_{k,t}=-b_kM_{k-1}S(t,W_{k,t}),\qquad
 M_k=M_{k-1}(1-b_k\xi_k),\quad M_1=\xi_1.
\end{equation}
Here \(M_{k-1}\) is the stochastic integral of \(v\) over all preceding
episodes; it is not reset to zero. For fixed \(N\), all finite moments
exist. Independence and isometry give, exactly,
\begin{equation}\label{fin:moments}
 q_k:=\mathbb E_0M_k^2=c\sqrt k,\qquad Q(v)=q_N=c\sqrt N,
\end{equation}
since \(q_1=c\) and \(q_k=q_{k-1}(1+cb_k^2)\).
Within one episode put
\(Y_t=F(t,W_t)-\mu=\int_0^tS(u,W_u)dW_u\).
It\^o's formula for \(Y^3\) shows
\[
 j_g:=\mathbb E_0\int_0^1Y_tS(t,W_t)^2dt=\kappa_3/3.
\]
The first episode contributes \(j_g\) to \(K(v)\).
In episode \(k\ge2\), the running score is \(M_{k-1}(1-b_kY_t)\),
so its contribution is \(q_{k-1}(-b_kc+b_k^2j_g)\).
Using \(Q(v)=c+c\sum_{k=2}^Nb_k^2q_{k-1}\) gives
\begin{equation}\label{fin:quotient}
 \frac{K(v)}{Q(v)}
 =\frac{j_g}{c}-\frac{c\sum_{k=2}^Nb_kq_{k-1}}{c\sqrt N}
 =\beta_g-\sqrt{\frac{Nc}{2}}\,a_N .
\end{equation}

To obtain a bounded implementable rule, fix \(R>0\) and let
\(\chi_R(x)=R\tanh(x/R)\). On the observed record define
\begin{equation}\label{fin:truncation}
 \begin{gathered}
 v^R_{1,t}=S(t,X_{1,t}),\qquad M^R_1=g(X_{1,1})-\mu,\\
 v^R_{k,t}=-b_k\chi_R(M^R_{k-1})S(t,X_{k,t}),\\
 M^R_k=M^R_{k-1}
       -b_k\chi_R(M^R_{k-1})[g(X_{k,1})-\mu],\qquad k\ge2 .
 \end{gathered}
\end{equation}
These are observable internal scores, not the unobserved drivers.
Indeed, for any controlled episode, the heat equation implies
\(\int_0^1 S(t,X_t)dX_t=g(X_1)-\mu\);
thus the update is exactly the endpoint update of \(\int v^R d\mathsf X\).
For fixed \(N,R\), the rule has a deterministic amplitude bound.
Its coefficient at the start of each episode is already known, and its
within-episode dependence is through the bounded Lipschitz function \(S\).
The controlled equation has a unique strong solution by successive
construction of episodes. The same construction up to an energy cap,
followed by zero control, gives uniqueness for every stopped rule.

Under the reference law, \eqref{fin:truncation} uses independent new
\(\xi_k\). The inequalities \(|\chi_R(x)|\le|x|\) and
\[
 |\chi_R(x)-y|\le |x-y|+|\chi_R(y)-y|
\]
give, by induction, uniform-in-\(R\) finite \(p\)-moment bounds and
\(M_k^R\to M_k\) in \(L^p\) for every fixed finite \(p\).
For the convergence step, \(\chi_R(M_k)\to M_k\) in \(L^p\) by dominated
convergence, and the independent new \(\xi_{k+1}\) has every finite moment.
Since \(S\) is bounded, the same induction gives convergence of \(v^R\)
in \(L^p(ds\otimes d\mathbb P_0)\). In particular \(A_N^R\to A_N\)
in \(L^2\) and \(M_N^R\to M_N\) in \(L^4\).
The isometry identity used in Lemma~\ref{fin:lem:variation}, also valid
for this finite-moment reference direction, therefore proves
\begin{equation}\label{fin:convergence}
 Q(v^R)\longrightarrow Q(v),\qquad K(v^R)\longrightarrow K(v).
\end{equation}
No exponential-moment claim or change of measure for the unbounded
direction is needed.

Choose \(D\) independently and uniformly from the coordinate vectors
\(e_1,\ldots,e_d\), and let \(\rho\) be an independent uniform sign.
Execute the bounded vector rule \(U=\varepsilon\rho Dv^R(X)\).
Given \(D\), the scalar drift parameter is \(\langle\theta,D\rangle\),
and its squared energy is unchanged. Lemma~\ref{fin:lem:variation}
and \(\mathbb E DD^\top=I_d/d\) imply, uniformly for \(\|h\|\le H\),
\begin{equation}\label{fin:ball-variation}
 r_N(\mathcal A,h)-z_N(h)
 =\frac{\varepsilon^2}{\sqrt N}
 \left\{\frac12Q(v^R)+\frac{\|h\|^2}{d\sqrt N}K(v^R)\right\}
       +O_{N,R,g,H}(\varepsilon^4/\sqrt N).
\end{equation}
By \eqref{fin:quotient} and \eqref{fin:positive-condition}, the brace
at radius \(H\) is strictly negative in the limit \(R\to\infty\).
Fix a sufficiently large finite \(R\) retaining a strict margin.
By continuity of the brace in radius, it is at most \(-\eta\) on
some annulus \(H-\alpha\le\|h\|\le H\), with \(\eta>0\) and
\(0<\alpha<H\). On that annulus, sufficiently small \(\varepsilon>0\)
gives risk at most \(Z_{N,d}^g(H)-\eta\varepsilon^2/(2\sqrt N)\).
On the remaining ball, strict radial monotonicity \eqref{fin:radial}
gives the positive baseline gap
\[
 Z_{N,d}^g(H)-z_N(h)
 \ge Z_{N,d}^g(H)-z_N((H-\alpha)e_1)>0 .
\]
The perturbation in \eqref{fin:ball-variation} is uniformly
\(O(\varepsilon^2)\); reduce \(\varepsilon\) to preserve half this gap.
The resulting rule has worst risk strictly below \(Z_{N,d}^g(H)\),
proving Theorem~\ref{thm:finite-positive} on the entire parameter ball.
The order of choices is fixed \(N,H\), then \(R\), then \(\varepsilon\).
This existence argument does not provide a uniform size of improvement.

To quantify the localization, let
\(\mathsf H_n=\sum_{j=1}^n j^{-1}\).
The squared \(k\)-th summand of \(Na_N\) equals
\(1-(\sqrt k+\sqrt{k-1})^{-2}\). The inequality
\(1-\sqrt{1-x}\le x\) therefore gives
\[
 0<1-a_N\le\frac{1+\mathsf H_{N-1}/4}{N}.
\]
With \(\delta_d=\sqrt{2c}H^2/d-1\), our two conditions read
\begin{align}
 \delta_d&\le-\frac{H^2}{\sqrt N}\left(\frac{\sqrt2L}{d}+M\right)
 &&\Longrightarrow\quad C_{N,d}^g(H)=Z_{N,d}^g(H),
 \label{fin:band-zero}\\
 \delta_d&>\frac{2H^2\beta_g}{d\sqrt N}
                  +(1+\delta_d)(1-a_N)
 &&\Longrightarrow\quad C_{N,d}^g(H)<Z_{N,d}^g(H).
 \label{fin:band-positive}
\end{align}
For fixed \(g,d\) and bounded radii these give an \(O(N^{-1/2})\)
localization band. They do not identify an exact finite-\(N\) boundary,
a sharp displacement constant, or a matching improvement rate.

\subsection{Equal variances but different finite-episode values}

Let \(\ell(x)=\log\cosh x\), and put
\[
 c_\ell=\operatorname{Var}(\ell(Z)),\quad
 \kappa_\ell=\mathbb E[\ell(Z)-\mathbb E\ell(Z)]^3,\quad
 a=\mathbb E\ell''(Z)=\mathbb E\operatorname{sech}^2Z>0.
\]
Choose
\[
 0<\eta<\min\{1,\sqrt{3a/(|\kappa_\ell|+1)}\},\qquad
 g_\eta(x)=\frac{x+\eta\ell(x)}{\sqrt{1+\eta^2c_\ell}},
 \qquad g_0(x)=x .
\]
The task \(g_\eta\) is nonaffine, asymmetric, smooth, convex, and has
all the required bounded derivatives. It is allowed to be unbounded below.
Parity, used only for this explicitly specified example, gives
\(\operatorname{Cov}(Z,\ell(Z))=0\), and hence
\(\operatorname{Var}(g_\eta(Z))=\operatorname{Var}(g_0(Z))=1\).
Gaussian integration by parts gives
\[
 \mathbb E[Z^2(\ell(Z)-\mathbb E\ell(Z))]
 =\mathbb E[(Z^2-1)\ell(Z)]=\mathbb E\ell''(Z)=a.
\]
Expansion of the third centered moment, again using parity, yields
\begin{equation}\label{fin:example-cumulant}
 \kappa_3(g_\eta(Z))
 =\frac{3\eta a+\eta^3\kappa_\ell}{(1+\eta^2c_\ell)^{3/2}}>0.
\end{equation}
The specified bound on \(\eta\) proves the sign without any assumption
on the sign of \(\kappa_\ell\) or of general convex-task cumulants.

Fix \(0<H^4<d^2/2\). For \(g_0\), \eqref{fin:global-energy} with
\(L=1\) gives exact zero-control optimality for every allowed \(N\);
\(\psi_{g_0}(u)=u^2/2\) then gives \(C_{N,d}^{g_0}(H)=H^2/2\).
For \(g_\eta\), \eqref{fin:eventual} gives exact zero-control optimality
for all sufficiently large allowed \(N\). Applying \eqref{fin:expansion},
\begin{equation}\label{fin:nonuniversal}
 C_{N,d}^{g_\eta}(H)-C_{N,d}^{g_0}(H)
 =-\frac{\kappa_3(g_\eta(Z))H^4}{6\sqrt N}+O(N^{-1})<0
\end{equation}
for all sufficiently large \(N\).
This proves Proposition~\ref{prop:finite-nonuniversal}. The comparison
is between regrets relative to each task's own exact oracle, not between
their uncentered costs. Variance determines the leading local value
but not these finite-episode corrections.

\section{The Exact Single-Episode Boundary}
\label{one:sec:proof}

We prove Theorem~\ref{thm:one-episode}. Throughout this section $N=1$,
$d\ge1$ is finite, the parameter satisfies $\|\theta\|\le H$, and
\[
 \ell=\|g'\|_\infty
   =\max\{|b_-|,|b_+|\},\qquad
 b_\pm=\lim_{x\to\pm\infty}g'(x).
\]
The limits exist because $g'$ is bounded and nondecreasing.
Nonconstancy gives $\ell>0$.
The sufficient condition $H^4\ell^2\le d^2/2$, including equality,
is Theorem~\ref{thm:finite-zero} with the smallest Lipschitz constant.
It remains to prove strict improvement whenever $H^4\ell^2>d^2/2$.
Set $A=H/\sqrt d$ for the scalar reference construction below.

\subsection{A Bounded Negative Direction for an Affine Task}

Let $\beta$ be Brownian motion on $[0,r]$, where $r>0$, and put
\[
 M_s^\nu=\int_0^s\nu_u\,d\beta_u,\qquad
 Q(\nu)=\mathbb E\int_0^r\nu_s^2\,ds.
\]
For any nonzero $b$ such that $A^4b^2r>1/2$, there is a bounded
predictable rule $\nu$, constant on a finite deterministic grid,
with
\begin{equation}
 \frac12 Q(\nu)+A^2b\,
       \mathbb E\int_0^rM_s^\nu\nu_s\,ds<0.
 \label{one:affine-direction}
\end{equation}
Here the coefficients of the grid rule are measurable functions of
the Brownian path already observed at the grid points.

To prove this, first take $0<\eta<r$ and use the square-integrable
reference direction
\[
 \nu_s=\eta^{-1/2}\quad(0\le s\le\eta),\qquad
 \nu_s=-\operatorname{sgn}(b)\frac{M_s^\nu}{\sqrt{2s}}
       \quad(\eta<s\le r).
\]
The linear equation starts at a positive deterministic time. Its
moments give, exactly,
\[
 \mathbb E(M_s^\nu)^2=\sqrt{s/\eta}\quad(s\ge\eta),
 \qquad Q(\nu)=\sqrt{r/\eta},\qquad
 \mathbb E\int_0^rM_s^\nu\nu_s\,ds
 =-\operatorname{sgn}(b)\frac{r-\eta}{\sqrt{2\eta}}.
\]
Thus the last quantity divided by $Q(\nu)$ tends to
$-\operatorname{sgn}(b)\sqrt{r/2}$ as $\eta\downarrow0$.
The strict assumption yields \eqref{one:affine-direction} for some
positive $\eta$.

Bounded predictable grid rules are dense in
$L^2(ds\otimes d\mathbb P)$ on the canonical Brownian filtration.
Both terms in \eqref{one:affine-direction} are continuous in this norm:
It\^o's isometry and Cauchy--Schwarz give
\[
 \left|\mathbb E\int_0^r M_s^\nu\nu_s\,ds
       -\mathbb E\int_0^r M_s^\mu\mu_s\,ds\right|
 \le\sqrt r\,(\|\nu\|_2+\|\mu\|_2)\|\nu-\mu\|_2.
\]
For completeness, the density follows by approximating predictable
processes by finite sums of rectangles $(s,t]\times A$,
$A\in\mathcal F_s$, truncating their coefficients, and taking the
union of their finitely many time endpoints as a grid. On the
canonical space each coefficient has a version that is a measurable
function of the past path.
Consequently one such bounded grid rule retains the strict inequality.
When this rule is evaluated on an observed diffusion path, its action
on each grid interval is already fixed by the past. Successive
construction therefore gives a unique strong law. Stopping at any
energy cap and then setting the action to zero also preserves
uniqueness. This establishes the implementability needed below,
without treating an arbitrary weak process as a decision rule.

\subsection{A Rare Observable State Selects the Tail Slope}

Choose $b\in\{b_-,b_+\}$ with $|b|=\ell$.
There is a deterministic $t_0\in(0,1)$ so small that
$A^4b^2(1-t_0)>1/2$. Put $r=1-t_0$, and fix a bounded grid rule
$\nu$ satisfying \eqref{one:affine-direction}.
For a finite $R>0$, define the observable event
\[
 E_R=\{X_{t_0}>R\}\quad\hbox{if }b=b_+,
 \qquad
 E_R=\{X_{t_0}<-R\}\quad\hbox{if }b=b_-.
\]
If both choices are available, either is sufficient.
Define a parameter-independent path rule $v$ by
\begin{equation}
 v_t=0\quad(t\le t_0),\qquad
 v_{t_0+s}=\mathbf1_{E_R}
       \nu_s(X_{t_0+\cdot}-X_{t_0})\quad(0<s\le r).
 \label{one:trigger}
\end{equation}
The grid starts after the event has been observed.
Choose a coordinate vector $D$ uniformly from $e_1,\ldots,e_d$
and an independent uniform sign $\rho$, both independent of the driver.
The actual vector control is
$U^\varepsilon=\varepsilon\rho Dv(X)$.
It is bounded, predictable, and admissible by the successive
construction above. Before $t_0$ it has no drift, so the distribution
of $E_R$ is the same under every parameter and has positive probability.

We now work only under the reference law $\theta=0$, where $X=W$.
Write $p_R=\mathbb P(E_R)>0$,
$M_t=\int_0^tv_s\,dW_s$, and
$S(t,x)=P_{1-t}g'(x)$. Brownian increments after $t_0$ are independent
of the entire past at $t_0$, and hence
\begin{equation}
 Q(v)=p_RQ(\nu),\qquad
 \frac1{p_R}\mathbb E\int_0^1S(t,W_t)M_tv_t\,dt
 \longrightarrow b\,\mathbb E\int_0^rM_s^\nu\nu_s\,ds
 \quad(R\to\infty).
 \label{one:tail-limit}
\end{equation}
To justify the limit, condition first on $W_{t_0}=x$.
For each fixed $s$ and subsequent Brownian increment,
$S(t_0+s,x+\beta_s)\to b$ as $x$ tends to the selected tail.
The bound
$|S(t_0+s,x+\beta_s)M_s^\nu\nu_s|
 \le\ell|M_s^\nu\nu_s|$
is integrable in $(s,\beta)$.
Dominated convergence shows that the conditional integral tends to
the displayed constant as $x$ tends to that tail. Averaging this
function over $W_{t_0}$ conditional on $E_R$ proves
\eqref{one:tail-limit}. No controlled state is asserted to be
conditionally Gaussian.

Combining \eqref{one:tail-limit} with \eqref{one:affine-direction},
choose one sufficiently large but finite $R$ for which
\begin{equation}
 F_H:=\frac12Q(v)+\frac{H^2}{d}K(v)<0,\qquad
 K(v)=\mathbb E\int_0^1S(t,W_t)M_tv_t\,dt.
 \label{one:negative}
\end{equation}
From this point on, $R$, $v$, and the grid remain fixed.

\subsection{Strict Improvement over the Full Parameter Ball}

Conditional on $D$, Lemma~\ref{fin:lem:variation} applies with
scalar parameter $\langle\theta,D\rangle$ in the full interval $[-H,H]$.
Averaging over $D$ and using $\mathbb E DD^\top=I_d/d$ gives,
uniformly for $\|\theta\|\le H$,
\begin{equation}
 R_1^g(U^\varepsilon,\theta)
 = B(\theta)+\varepsilon^2F_\theta+O(\varepsilon^4),
 \qquad
 B(\theta)=\mathbb E g(Z)-V_\theta(0,0),\qquad
 F_\theta=\frac12Q(v)+\frac{\|\theta\|^2}{d}K(v).
 \label{one:risk-expansion}
\end{equation}
The independent sign removes odd powers without requiring symmetry of
$g$. The remainder constant may depend on the fixed $R,v,H,g,d$;
it need not be uniform as $R\to\infty$.

The zero-control baseline $B$ is continuous, radial, and strictly
increasing in $\|\theta\|>0$, as proved in
Appendix~\ref{fin:sec:proof}. By \eqref{one:negative}, $K(v)<0$.
Choose $\zeta\in(0,H)$ and $c>0$ such that
$F_\theta\le-c$ whenever $H-\zeta\le\|\theta\|\le H$.
For sufficiently small positive $\varepsilon$, the risks on that
set are at most $B(He_1)-c\varepsilon^2/2$.
On the remaining ball, the baseline gap is at least
$B(He_1)-B((H-\zeta)e_1)>0$, whereas the perturbation in
\eqref{one:risk-expansion} is uniformly $O(\varepsilon^2)$.
Shrinking $\varepsilon$ again makes those risks strictly less than
$B(He_1)$ as well. Thus
\[
 \sup_{\|\theta\|\le H}R_1^g(U^\varepsilon,\theta)<B(He_1),
\]
which proves the necessary direction of Theorem~\ref{thm:one-episode}.
The argument proves a strict minimax improvement, without claiming that
the sphere-supported prior is least favorable or locating the exact maximum
of this perturbed risk.

Finally, if $g$ is nonaffine, then for every $t<1$ and finite $x$,
$|P_{1-t}g'(x)|<\ell$. Indeed, a continuous nonconstant $g'$ cannot
equal either extremal slope almost everywhere under a Gaussian
measure with full support. Therefore
\[
 c_g=\mathbb E\int_0^1S(t,W_t)^2\,dt<\ell^2.
\]
The single-episode critical radius $\sqrt d\,(2\ell^2)^{-1/4}$ is
consequently strictly smaller than the asymptotic radius
$\sqrt d\,(2c_g)^{-1/4}$.
The strict improvement obtained here can be arbitrarily small;
this difference of exact boundaries is compatible with the
quantitative convergence of minimax values.

\end{document}